\documentclass[11pt]{aims}
\usepackage{amsmath, amssymb, mathrsfs}
  \usepackage{paralist}
  \usepackage{graphics}
    \usepackage{bbm}
  \usepackage{epsfig} 
\usepackage{graphicx}  \usepackage{epstopdf}
 \usepackage[colorlinks=true]{hyperref}
 \hypersetup{urlcolor=blue, citecolor=red}
 \usepackage[toc,page]{appendix}
\usepackage{chngcntr}

\usepackage{titlesec}\titleformat{\section}{\Large\bfseries}{\thesection.}{4pt}{}
\titleformat{\subsection}{\large\bfseries}{\thesection.\arabic{subsection}.}{4pt}{}
\titleformat{\subsubsection}{\bfseries}{\thesection.\arabic{subsection}.\arabic{subsubsection}.}{4pt}{}
\titleformat*{\paragraph}{\bfseries}
\titleformat*{\subparagraph}{\bfseries}
  \def\currentyear{}
   \def\currentmonth{}
    \def\ppages{}
     \def\DOI{}

\newtheorem{theorem}{Theorem}[section]

\newtheorem{lemma}[theorem]{Lemma}
\newtheorem{proposition}[theorem]{Proposition}
\theoremstyle{definition}
\newtheorem{definition}[theorem]{Definition}
\newtheorem{remark}[theorem]{Remark}

\newcommand{\Rd}{\mathbb{R}^d}
\newcommand{\Rb}{\mathbb{R}}
\newcommand{\indic}{\mathbbm{1}}

\newcommand{\Cc}{\mathcal{C}}
\newcommand{\Dc}{\mathcal{D}}
\newcommand{\Ec}{\mathcal{E}}
\newcommand{\Fc}{\mathcal{F}}

\newcommand{\Hc}{\mathcal{H}}
\newcommand{\Lc}{\mathcal{L}}
\newcommand{\Oc}{\mathcal{O}}

\newcommand{\Nc}{\mathcal{N}}

\newcommand{\N}{\mathbb N}

\newcommand{\Z}{\mathbb Z}

\newcommand{\py}{\partial_y}
\newcommand{\ps}{\partial_s}

\newcommand{\pt}{\partial_t}
\newcommand{\p}{\partial}
\newcommand{\Ls}{\mathscr{L}}
\newcommand{\Hs}{\mathscr{H}}
\newcommand{\As}{\mathscr{A}}
\newcommand{\Es}{\mathscr{E}}

\numberwithin{equation}{section}

\title[1-corotational harmonic map heat flow in supercritical dimensions] 
      {Stable type II blowup for the 7 dimensional 1-corotational energy supercritical harmonic map heat flow}

\author[T. Ghoul]{}

 \email[T. Ghoul]{teg6@nyu.edu}
 
\thanks{\today}
\begin{document}
\maketitle


\centerline{\scshape Tej-eddine Ghoul$^a$}
\medskip
{\footnotesize
 \centerline{$^a$Department of Mathematics, New York University in Abu Dhabi,}
   \centerline{Saadiyat Island, P.O. Box 129188, Abu Dhabi, United Arab Emirates.}
}

\bigskip

\begin{abstract} We consider the energy-supercritical harmonic map heat flow from $\Rb^d$ into $\mathbb{S}^d$, under an additional assumption of 1-corotational symmetry. We are interested by the 7 dimensional case which is the borderline between the Type I blowup regime and Type II.
We construct for this problem a stable finite time blowup solution under the condition of corotational symmetry that blows up via concentration of the universal profile
$$u(r,t) \sim Q\left(\frac{r}{\lambda(t)}\right),$$
where $Q$ is the stationary solution of the equation and the speed is given by the rate
$$\lambda(t) \sim \frac{\sqrt{(T-t)}}{|\log(T-t)|},$$
which corresponds to the speed predicted in \cite{BIEnon2015}.
\end{abstract}

\section{Introduction.} We consider the harmonic map heat flow which is defined as the negative gradient flow of the Dirichlet energy of maps between manifolds. Indeed, if $\Phi$ is a map from $\Rb^d\times[0,T)$ to a compact Riemannian  manifold $\mathcal{M}\subset \Rb^n$, with second fundamental form $\Upsilon$,  then $\Phi$ solves
\begin{equation}\label{p0}
\left\{\begin{array}{l}
\partial_t \Phi - \Delta \Phi = \Upsilon(\Phi)(\nabla \Phi,\nabla\Phi),\\
\Phi(t = 0) = \Phi_0.
\end{array}
 \right.
\end{equation}

We assume that the target manifold is the $d$-sphere $\mathbb{S}^d\subset\Rb^{d+1}$.
Hence, \eqref{p0} becomes

\begin{equation}\label{p1}
\left\{\begin{array}{l}
\partial_t \Phi - \Delta \Phi = |\nabla \Phi|^2\Phi,\\
\Phi(t = 0) = \Phi_0.
\end{array}
 \right.
\end{equation}

 We will study the problem under an additional assumption of 1-corotational symmetry, with the following  corotational ansatz
\begin{equation}
\Phi(x,t) = \binom{\cos(u(|x|,t))}{\frac{x}{|x|}\sin(u(|x|,t))},
\end{equation}
where  $u = u(r,t)$is a radial function. Under this ansatz, the problem \eqref{p1} reduces to the one dimensional semilinear heat equation
 \begin{equation}\label{Pb}
\left\{\begin{array}{ll}
\partial_t u &= \partial^2_r u + \frac{(d-1)}{r}\partial_r u - \frac{(d-1)}{2r^2}\sin(2u),\\
u(t=0) &= u_0,
\end{array}\right.
\end{equation}
where $u(t): r \in \Rb_+ \to u(r,t) \in [0, \pi]$. The set of solutions to \eqref{Pb} is invariant by the scaling symmetry
$$u_\lambda(r,t) = u\left(\frac{r}{\lambda}, \frac{t}{\lambda^2} \right), \quad \forall\lambda > 0.$$

The energy associated to \eqref{Pb} is given by 
\begin{equation}\label{def:Enut}
\Ec[u](t) = \int_0^{+\infty} \left(|\partial_r u|^2 + \frac{d-1}{r^2}\sin^2(u)\right)r^{d-1}dr,
\end{equation}
which satisfies
$$\Ec[u_\lambda] = \lambda^{d-2}\Ec[u].$$
The criticality of the problem is  reflected by the fact that the energy  \eqref{def:Enut} is left invariant by the scaling property when $d = 2$, hence, the case $d \geq 3$ corresponds to the energy supercritical case. 
The problem \eqref{Pb} is locally wellposed for data which are close in $L^\infty$ to a uniformly continuous map (see Koch and Lamm \cite{KLajm12}) or in $BMO$ by Wang \cite{Warma11}.
Actually, Eells and Sampson \cite{ESajm64} introduced the harmonic map heat flow as a process to deform any smooth map $\Phi_0$ into a harmonic map via \eqref{p1}. They also proved that the solution exists globally if the sectional curvature of the target manifold is negative. There exist other assumptions for the global existence such as the image of the initial data $u_0$ is contained in a ball of radius $\frac{\pi}{2\sqrt{\kappa}}$, where $\kappa$ is an upper bound on the sectional curvature of the target manifold $\mathcal{M}$ (see Jost \cite{Jmm81} and Lin-Wang \cite{LWwsp08}).
Without these assumptions, the solution $u(r,t)$ may develop singularities in some finite time (see for examples, Coron and Ghidaglia \cite{CGcrsa89}, Chen and Ding \cite{CDinvm90} for $d \geq 3$, Chang, Ding and Yei \cite{CDYjdg92} for $d = 2$). In this case, we say that $u(r,t)$ blows up in a finite time $T < +\infty$ in the sense that 
$$\lim_{t \to T}\|\nabla u(t)\|_{L^\infty} = + \infty.$$
Here we call $T$ the blowup time of $u(x,t)$. The blowup has been divided by Struwe \cite{Sams96} into two types:
\begin{align*}
\qquad & \text{$u$ blows up with type I if:} \quad \limsup_{t \to T} (T-t)^\frac{1}{2}\|\nabla u(t)\|_{L^\infty} < +\infty,\\
\qquad & \text{$u$ blows up with type II if:} \quad \limsup_{t \to T} (T-t)^\frac{1}{2}\|\nabla u(t)\|_{L^\infty} = +\infty.
\end{align*}

In the energy critical case, i.e. $d = 2$, Van de Berg, Hulshof and King \cite{VHKsiam03}, through a formal analysis based on the matched asymptotic technique of Herrero and Vel\'azquez \cite{HVcras94}, predicted that there are type II blowup solutions to \eqref{Pb} of the form 
\begin{equation*}
u(r,t) \sim Q\left(\frac{r}{\lambda(t)}\right),
\end{equation*} 
where 
\begin{equation}\label{eq:Qrd2}
Q(r) = 2\tan^{-1}(r)
\end{equation}
is the unique (up to scaling) solution of \eqref{eq:Qr}, and the blowup speed governed by the quantized rates is
\begin{equation*}
\lambda(t) \sim \frac{(T-t)^\ell}{|\log (T-t)|^\frac{2\ell}{2\ell - 1}} \qquad \text{for}\;\; \ell \in \mathbb{N}^*.
\end{equation*}
This result was later confirmed by Rapha\"el and Schweyer \cite{RSapde2014}. Note that the case $\ell = 1$ was treated in \cite{RScpam13} and corresponds to a stable blowup.  
In particular, the authors in \cite{RSapde2014}, \cite{RScpam13} adapted the strategy developed by Merle, Rapha\"el and Rodnianski \cite{RRmihes12}, \cite{MRRmasp11} for the study of wave and Schr\"odinger maps to construct for equation \eqref{Pb} type II blowup solutions.\\
Note that the case $\ell=1$ which corresponds to the stable case for $d=2$ has been proved without any symmetry and for general domains in the recent result of Davilla, Del Pino and Wei in \cite{DDW}. \\
In \cite{Sjdg88}, Struwe shows that the type I singularities are asymptotically self-similar, that is their profile is given by a smooth shrinking function
$$u(r,t) = \phi\left(\frac{r}{\sqrt{T-t}}\right), \quad \forall t \in [0,T),$$
where $\phi$ solves the equation 
\begin{equation}\label{eq:phis}
\phi'' + \left(\frac{6}{y} + \frac{y}{2}\right)\phi' - \frac{3}{y^2}\sin(2\phi) = 0.
\end{equation}
Thus, the study of Type I blowup reduces to the study of nonconstant solutions of equation \eqref{eq:phis}.
When $3 \leq d \leq 6$, by using a shooting method, Fan \cite{Fscsm99} proved that there exists an infinite sequence of globally regular solutions $\phi_n$ of \eqref{eq:phis} which are called "shrinkers (corresponding to the existence of Type I blowup solutions of \eqref{Pb}), where the integer index $n$ denotes the number of intersections of the function $\phi_n$ with $\frac{\pi}{2}$. More detailed quantitative properties of such solutions were studied by Biernat and Bizo{\'n} \cite{BBnon11}, where the authors conjectured that $\phi_1$ is linearly stable and provide numerical evidences supporting that $\phi_1$ corresponds to a generic profile of Type I blow-up. Very recently, Biernat, Donninger and Skhorkhuber \cite{BDSar16} proved the existence of a stable self-similar blow up solution for $d=3$.
Since \eqref{Pb} is not time reversible there exist another family of self similar solutions that are called "expanders" which were introduced by Germain and Rupflin in \cite{GRihp08}.
Recently, Germain , Ghoul and Miura proved the nonlinear stability of the expanders in \cite{GGMar16}, and use the expanders to construct nonunique solutions verifying the local energy inequality. Up to our knowledge, the question on the existence of Type II blowup solutions  for \eqref{Pb} remains open for $3 \leq d \leq 6$.\\

When $d \geq 7$, Bizo{\'n} and Wasserman \cite{BWimrn15} proved that equation \eqref{Pb} has no self-similar shrinking solutions. According to Struwe \cite{Sjdg88}, this result implies that in dimensions $d \geq 7$, all singularities for equation \eqref{Pb} must be of type II (see also Biernat \cite{BIEnon2015} for a recent analysis of such singularities). Recently, Biernat and Seki \cite{BSarxiv2016}, via the matched asymptotic method developed by Herero and Vel\'azquez \cite{HVcras94}, construct for equation \eqref{Pb} a countable family of Type II blowup solutions, each characterized by a different blowup rate
\begin{equation}\label{eq:blrate}
\lambda(t) \sim (T-t)^\frac{\ell}{\gamma} \quad \text{as} \quad t \to T,
\end{equation} 
where $\ell \in \mathbb{N}^*$ such that $2\ell > \gamma$ and $\gamma = \gamma(d)$ is given by
\begin{equation}\label{def:gamome}
\gamma(d) = \frac{1}{2}(d - 2 - \tilde{\gamma}) \in (1,2] \quad \text{for}\;\; d \geq 7,
\end{equation}
where $\tilde{\gamma}=\sqrt{d^2-8d+8}$.
The blowup rate \eqref{eq:blrate} is in fact driven by the asymptotic behavior of a stationary solution of \eqref{Pb}, say $Q$, which is the unique  (up to scaling) solution of the equation
\begin{equation}\label{eq:Qr}
Q'' + \frac{6}{r}Q' - \frac{3}{r^2}\sin(2Q) = 0, \quad Q(0) = 0, \; Q'(0) = 1,
\end{equation} 
and admits the behavior for $r$ large,
\begin{equation}\label{exp:Qr}
Q(r) = \dfrac{\pi}{2} - \dfrac{a_0}{ r^{\gamma}} + \Oc\left(\dfrac 1{r^{2 + \gamma}}\right) \quad \text{for some}\;\; a_0 = a_0(d) > 0,
\end{equation}
(see Appendix in \cite{BIEnon2015} for a proof of the existence of $Q$). 
Note that the authors of \cite{BSarxiv2016} only claim an existence result of type II blowup solution with the rate \eqref{eq:quanblrate} which relates to a spectrum problem and say nothing about the dynamical description of the solution. 
Recently, Ghoul, Ibrahim and Nguyen  provided a new proof of \cite{BSarxiv2016}, where they constructed a family of $\mathcal{C}^{\infty}$ solutions which blow up in finite time  via concentration of the universal profile
$$u(r,t) \sim Q\left(\frac{r}{\lambda(t)}\right),$$
where $Q$ is the stationary solution of the equation and the speed is given by the quantized rates
$$\lambda(t) \sim c_u(T-t)^\frac{\ell}{\gamma}, \quad \ell \in \mathbb{N}^*, \;\; 2\ell > \gamma = \gamma(d) \in (1,2].$$

In, addition they proved that those family of solutions are stable on a manifold of codimension $\ell-1$.
Note that the condition $\frac{\ell}{\gamma}>\frac{1}{2}$ is to ensure that the blow up is of type II, and the case $\frac{\ell}{\gamma}=\frac{1}{2}$ only happens in dimension $d = 7$ where we should expect type I blow up. In this case, Biernat in \cite{BIEnon2015} said that if one forgets about the underlying geometric setup and allows $d$ to be positive real values, then for $d$ slightly less than 7 numerical evidence indicates a presence of a generic type I blow-up. On the other hand, if $d$ is slightly bigger than $7$ from \eqref{eq:blrate} the blow up should be of type II. Naively, one could think that when $d=7$ the blow up should be of type I.  
However, Biernat by using the method of matching asymptotics\cite{HVcras94} formally derived the blowup rate
\begin{equation}\label{eq:rated7}
\lambda(t) \sim \frac{(T-t)^\frac 12}{|\log (T-t)|} \quad \text{as} \quad t \to T,
\end{equation}
which is of type II.
He also provided numerical evidences supporting that this case $\ell = 1$ in \eqref{eq:blrate} or \eqref{eq:rated7} corresponds to a stable blowup solution.\\

The goal of the present work is to prove rigorously, the existence of a stable finite time blowup solution for $d=7$ and $\ell=1$. This is the main result in this paper: 
\begin{theorem}[Existence of type II blowup solutions to \eqref{Pb} with prescibed behavior] \label{Theo:1} Let $d=7$.Then there exist, $L>0$ large enough and an open set of initial data of the form 
$$u_0 = Q + q_0, \quad q_0 \in \Hs_{2L+2}, $$
such that the corresponding solution to equation \eqref{Pb} satisfies
\begin{equation}\label{eq:uQq}
u(r,t) = Q\left(\frac{r}{\lambda(t)}\right) + q\left(\frac{r}{\lambda(t)}, t\right)
\end{equation}
where 
\begin{equation}\label{eq:quanblrate}
\lambda(t) = \frac{C\sqrt{(T-t)}}{|\log(T-t)|} (1 + O_{t \to T}(|\log(T-t)|^{-1})), \quad C > 0,
\end{equation}
and 
\begin{equation}\label{eq:asypqs}
\lim_{t \to T}\|q\|_{\Hs_{2L+2}} = 0,
\end{equation}
where $\Hs_{2L+2}$ is defined in \eqref{locnorm}.
\end{theorem}

\begin{remark} It is worth to mention that our analysis relies only on the study of supercritical Sobolev norms buit on the linearized operator, thus, the finiteness of the $H^1$ norm of the initial data is not a requirement. Therefore, the constructed solutions $u$ can be taken to be of finite energy, if its energy is initially finite, thanks to the following trivial identity (testing \eqref{p1} by $\pt \Phi$ and using $\Phi \cdot \pt \Phi = 0$)
$$\frac 12 \frac d{dt}\int |\nabla \Phi|^2 + \int |\pt \Phi|^2 = 0.$$
In addition, if we add the assumption that the solution $u$ is of finite energy, it is required to suppose that the error $q\sim -Q$ as $x$ goes to Infinity to ensure that $u$ is zero at infinity.
\end{remark}

\begin{remark} It is worth to remark that the harmonic heat flow shares many features with the semilinear heat equation
\begin{equation}\label{eq:she}
\partial_t u = \Delta u + |u|^{p-1}u \quad \text{in}\; \Rb^d.
\end{equation}
A remarkable fact is that two important critical exponents appear when considering the dynamics of \eqref{eq:she}:
$$p_S = \frac{d + 2}{d-2} \quad \text{and} \quad p_{JL} = \left\{\begin{array}{ll} +\infty &\quad \text{for}\;\; d \leq 10,\\
 1 + \frac{4}{d - 4 - 2\sqrt{d - 1}}&\quad \text{for}\;\; d \geq 11,
\end{array}  \right.$$
correspond to the cases $d = 2$ and $d = 7$ in the study of equation \eqref{Pb} respectively. 

When $1 < p \leq p_S$, Giga and Kohn \cite{GKiumj87}, Giga, Matsui and Sasayama \cite{GMSiumj04} showed that all blowup solutions are of type I. Here the type I blowup means that
$$\limsup_{t \to T}(T-t)^\frac{1}{p-1}\|u(t)\|_{L^\infty} < +\infty,$$
otherwise we say the blowup solution is of type II.

When $p = p_S$, Filippas, Herrero and Vel\'azquez \cite{FHVslps00} formally constructed for \eqref{eq:she} type II blowup solutions in dimensions $3 \leq d \leq 6$, however, they could not do the same in dimensions $d \geq 7$. This formal result is partly confirmed by Schweyer \cite{Schfa12} in dimension $d = 4$. Interestingly, Collot, Merle and Rapha\"el \cite{CMRar16} show that type II blowup is ruled out in dimension $d \geq 7$ near the solitary wave.

When $p_S < p < p_{JL}$, Matano and Merle \cite{MMcpam04} (see also Mizoguchi \cite{Made04}) proved that only type I blowup occurs in the radial setting. 

When $p > p_{JL}$, Herrero and Vel\'azquez formally derived in \cite{HVcras94} the existence of type II blowup solutions with the quantized rates:
$$\|u(t)\|_{L^\infty} \sim (T-t)^ \frac{2\ell}{(p-1)\alpha(d,p)}, \quad \ell \in \mathbb{N}, \; 2\ell > \alpha.$$
The formal result was clarified in \cite{MMjfa09}, \cite{Mma07} and \cite{Car16}. The collection of these works yields a complete classification of the type II blowup scenario for the radially symmetric energy supercritical case. 

In comparison to the case of the semilinear heat equation \eqref{eq:she}, it might be possibe to prove that all blowup solutions to equation \eqref{Pb} are of type I in dimension $3\leq d \leq 6$. However, due to the lack of monotonicity of the nonlinear term, the analysis of the harmonic map heat flow \eqref{Pb} is much more difficult than the case of the semilinear heat equation \eqref{eq:she} treated in \cite{MMcpam04}. 
\end{remark}

\bigskip

Let us briefly explain the main steps of the proof of Theorem \ref{Theo:1}, which follows the method of  \cite{RSapde2014} treated for the critical case $d = 2$. We would like to mention that this kind of method has been successfully applied for various nonlinear evolution equations. In particular in the dispersive setting for the nonlinear Schr\"odinger equation both in the mass critical \cite{MRam05, MRcmp05, MRim04, MRgfa03} and mass supercritical \cite{MRRcmj2016} cases; the mass critical gKdV equation \cite{MMRasp15, MMRjems15, MMRam14};  the energy critical \cite{DKMcjm13}, \cite{HRapde12} and supercritical \cite{Car161} wave equation; the two dimensional critical geometric equations: the wave maps \cite{RRmihes12}, the Schr\"odinger maps \cite{MRRim13} and the harmonic heat flow \cite{RScpam13, RSapde2014}; the semilinear heat equation \eqref{eq:she} in the energy critical \cite{Sjfa12} and supercritical \cite{Car16} cases; and the two dimensional Keller-Segel model \cite{RSma14,GMarx16}. In all these works, the method relies on two arguments:
\begin{itemize}
\item Reduction of an infinite dimensional problem to a finite dimensional one, through the derivation of suitable Lyapunov functionals and the robust energy method as mentioned in the two step procedure above.
\item The control of the finite dimensional problem thanks to a topological argument based on index theory.
\end{itemize}
Note that this kind of topological arguments has proved to be successful also for the construction of type I blowup solutions for the semilinear heat equation \eqref{eq:she} in \cite{BKnon94}, \cite{MZdm97}, \cite{NZens16} (see also \cite{NZsns16} for the case of logarithmic perturbations, \cite{Breiumj90}, \cite{Brejde92} and \cite{GNZpre16a} for the exponential source, \cite{NZcpde15} for the complex-valued case), the Ginzburg-Landau equation in \cite{MZjfa08} (see also \cite{ZAAihn98} for an earlier work), a non-variational parabolic system in \cite{GNZpre16c} and the semilinear wave equation in \cite{CZcpam13}.\\ 
Note also that here we don't use the topological argument because the blow-up is stable.   

For the reader's convenience and for a better explanation, let's first introduce notations used throughout this paper.\\
\noindent - \textbf{Notation.} We will fix $d=7$, for all the rest of the paper. Given a large integer $L \gg 1$, we set
\begin{equation}\label{def:kbb}
\eta = \frac{3}{4L}.
\end{equation}
Note that for $d=7$, $~\gamma(7)=2$ where $\gamma$ has been defined in \eqref{def:gamome}. 
Let $\chi \in \Cc_0^\infty([0, +\infty))$ be a positive nonincreasing cutoff function with $\text{supp}(\chi) \subset [0,2]$ and $\chi \equiv 1$ on $[0,1]$. For all $M > 0$, we define
\begin{equation}\label{def:chiM}
\chi_M(y) = \chi\left(\frac y M\right).
\end{equation}
Given $b_1 > 0$ and $\lambda > 0$, we define 
\begin{equation}\label{def:B0B1}
B_0 = \frac{C_\chi}{\sqrt b_1}, \quad B_1 = B_0^{1 + \eta},
\end{equation}
where 
\begin{equation}\label{defofCchi}
C_\chi=\frac{(1+\int_1^2\chi(x)dx)^2(1/3+\int_1^2x^2\chi(x)dx)}{(1/2+\int_1^2x\chi(x)dx)^3},
\end{equation}

and denote by
$$f_\lambda(r) = f(y) \quad \text{with} \quad y = \frac{r}{\lambda}.$$

We also introduce the differential operator 
$$\Lambda f = y\partial_y f,$$
and the Schr\"odinger operator
\begin{equation}\label{def:Lc}
\Ls  = -\partial_{yy} - \frac{6}{y}\partial_y  + \frac{Z}{y^2}, \quad \text{with}\;\; Z(y)= 6\cos(2Q(y)).
\end{equation}

\medskip

\noindent - \textbf{Strategy of the proof.}  We now summarize the main ideas of the proof of Theorem \ref{Theo:1}, which follows the route map of \cite{RSapde2014} and \cite{MRRcmj2016}.\\

\noindent $(i)$ \textit{Renormalized flow and iterated resonances.} Following the scaling invariance of \eqref{Pb}, let us make the change of variables
\begin{equation*}
w(y,s) = u(r,t), \quad y = \frac{r}{\lambda(t)}, \quad \frac{ds}{dt} = \frac{1}{\lambda^2(t)},
\end{equation*}
which leads to the following renormalized flow:
\begin{equation}\label{eq:wys_i}
\partial_sw = \partial_y^2 w + \frac{6}{y}\partial_y w  + b_1\Lambda w - \frac{3}{y^2}\sin(2w), \quad b_1 = -\frac{\lambda_s}{\lambda}.
\end{equation}
Assuming that the leading part of the solution $w(y,s)$ is given by the ground state profile $Q$ admitting the asymptotic behavior \eqref{exp:Qr}, hence the remaining part is governed by the Schr\"odinger operator $\Ls$ defined by \eqref{def:Lc}. The linear operator $\Ls$ admits the factorization (see Lemma \ref{lemm:factorL} below) 
\begin{equation}\label{eq:facL_i}
\Ls = \As^*\As, \quad \As f = -\Lambda Q \py \left(\frac{f}{\Lambda Q} \right), \quad \As^*f = \frac{1}{y^{6}\Lambda Q}\py\left(y^{6}\Lambda Q f\right), 
\end{equation}
which directly implies
$$\Ls(\Lambda Q) = 0,$$
where from a direct computation,
$$\Lambda Q \sim \frac{2a_0}{y^2} \quad \text{as} \quad y \to +\infty.$$
More generally, we can compute the kernel of the powers of $\Ls$ through the iterative scheme
\begin{equation}\label{def:Tk_i}
\Ls T_{k + 1} = - T_k, \quad T_0 = \Lambda Q,
\end{equation}
which displays a non trivial tail at infinity (see Lemma \ref{lemm:GenLk} below),
\begin{equation}\label{eq:Tk_i}
T_k(y) \sim c_ky^{2k - 2} \quad \text{for}\quad y \gg 1.
\end{equation}

\noindent $(ii)$ \textit{Tail dynamics.} Following the approach in \cite{RSapde2014}, we look for a slowly modulated approximate solution to \eqref{eq:wys_i} of the form 
$$w(y,s) = Q_{b(s)}(y),$$
where 
\begin{equation}\label{def:Qb_i}
b = (b_1, \cdots, b_L), \quad Q_{b(s)}(y) = Q(y) + \sum_{i = 1}^Lb_iT_i(y) + \sum_{i = 2}^{L+2}S_i(y)
\end{equation}
with a priori bounds
$$b_i \sim b_1^{i+\frac{1}{2}+\frac{\eta}{10}}, \quad |S_i(y)| \lesssim b_1^{i+\frac{1}{2}+\frac{\eta}{10}} y^{2i - 3},$$
so that $S_i$ is in some sense homogeneous of degree $i$ in $b_1$, and behaves better than $T_i$ at infinity. 
The construction of $S_i$ is called the \textit{tail computation}. Let us illustrate the procedure of the \textit{tail computation}. We plug the decomposition \eqref{def:Qb_i} into \eqref{eq:wys_i} and choose the law for $(b_i)_{1 \leq i \leq L}$ which cancels the  leading order terms at infinity.\\
- At the order $\Oc(b_1)$: we cannot adjust the law of $b_1$ for the first term \footnote{if $(b_1)_s = -c_1 b_1$, then $-\lambda_s/\lambda \sim b_1 \sim e^{-c_1 s}$, hence after an integration in time, $|\log \lambda| \lesssim 1$ and there is no blowup.} and obtain from \eqref{eq:wys_i}, 
$$b_1(\Ls T_1 + \Lambda Q) = 0.$$
- At the order $\Oc(b_1^2, b_2)$: We obtain
$$(b_1)_sT_1 + b_1^2\Lambda T_1 + b_2\Ls T_2 + \Ls S_2 = b_1^2 NL_1(T_1, Q),$$
where $NL_1(T_1, Q)$ corresponds to nonlinear interaction terms. Note from \eqref{eq:Tk_i} and \eqref{def:Tk_i}, we have
$$T_1\sim -C_0+\frac{C_1}{y} \quad\text{ for } y\gg 1,$$
and
$$\Lambda T_1 \sim \frac{-C_1}{y} \quad \text{for}\quad y \gg 1, \quad \Ls T_2 = - T_1.$$
In order to minimize the growth of the tails of the profiles $S_k$ at infinity we will introduce a correction.
Approximately $S_2$ will solve:
$$\Ls S_2= -\Big(\underbrace{\Lambda T_1+\frac{(b_1)_s}{b_1^2}T_1}_{\sim\theta_1}\Big)+l.o.t,$$
to minimize the growth of $S_2$ at infinity we need to cancel the growth $\frac{1}{y}$ of $\Lambda T_1$ at infinity. To do so, set $\frac{(b_1)_s}{b_1^2}=-C_{b_1}$. If we use $C_{b_1} T_1$ to cancel the growth $\frac{1}{y}$ of $\Lambda T_1$ in the parabolic zone $y=\frac{1}{\sqrt{b_1}}$ formally by using that $T_1=-C_0+\dfrac{C_1}{y} + \Oc\left(\frac{|\log(y)|}{y^2}\right)\;\; \text{as} \;\; y \to + \infty$ we deduce that $C_{b_1}\sim C\sqrt{b_1}$, where $C>0$ is a constant related to the asymptotic behaviour of $Q$.
Hence, 
$$(b_1)_sT_1 + b_1^2\Lambda T_1 + b_2\Ls T_2 \sim \big[(b_1)_s + b_1^{\frac 52} - b_2\big]T_1.$$
It follows that the leading order growth for $y$ large is canceled by the choice 
$$(b_1)_s + b_1^{\frac 52} - b_2 = 0.$$
We then solve for 
$$\Ls S_2 = -b_1^2(\Lambda T_1 - C\sqrt{b_1}T_1) + b_1^2 NL_1(T_1, Q),$$
and check the improved decay 
$$|S_2(y)| \lesssim b_1^2y^{-1} \quad \text{for} \quad y \gg 1.$$
- At the order $\Oc(b_1^{k+1}, b_{k+1})$: we obtain an elliptic equation of the form
$$(b_k)_sT_k + b_1b_k\Lambda T_k + b_{k + 1}\Ls T_{k+1} + \Ls S_{k+1} = b_1^{k+1}NL_k(T_1, \cdots, T_k, Q).$$ 
From \eqref{eq:Tk_i} and \eqref{def:Tk_i}, we have 
$$(b_k)_sT_k + b_1b_k\Lambda T_k + b_{k + 1}\Ls T_{k+1} \sim \big[(b_k)_s + (2k - 2+C\sqrt{b_1})b_1b_k - b_{k+1}\big]T_{k},$$
which leads to the choice 
$$(b_k)_s + (2k - 2+C\sqrt{b_1})b_1b_k - b_{k+1} = 0,$$
for the cancellation of the leading order growth at infinity.
We then solve for the remaining $S_{k+1}$ term and check that $|S_{k+1}(y)|\lesssim b_1^{k+\frac{3}{2}}y^{2 k - 3}$ for $y$ large. We refer to Proposition \ref{prop:1} for all details of the \textit{tail computation}.\\

\noindent $(iii)$ \textit{The universal system of ODEs.} The above procedure leads to the following universal system of ODEs after $L$ iterations,
\begin{equation}\label{sys:bk_i}
\left\{ \begin{array}{l}
(b_k)_s + (2k - 2+C\sqrt{b_1})b_1b_k - b_{k+1} = 0, \quad 1 \leq k \leq L, \quad b_{L+1} = 0,\\
\quad \\
-\dfrac{\lambda_s}{\lambda} = b_1, \quad \dfrac{ds}{dt} = \frac{1}{\lambda^2}.
\end{array}\right.
\end{equation}
Unlike the critical case treated in \cite{RSapde2014}, there is no further logarithmic correction to take into account, and unlike the cases $d\geq 8,\quad \ell\geq1$ or $d=7,\quad\ell>1$ there are no unstable directions in the ODE system, and the key point here is the introduction of this $\sqrt{b_1}$ correction. 
Indeed, the solutions to \eqref{sys:bk_i} will behave approximately as 
\begin{equation}\label{eq:solbk_i}
\left\{\begin{array}{l}
b_1(s)\sim \frac{1}{s^{\frac 23}},\quad\quad
b_k(s) \sim \frac{1}{s^{\frac{2k}{3}+\frac 13}}, \quad 2 \leq k \leq L,\\
\lambda(s) \sim e^{-3s^{\frac 13}}.
\end{array}\right.
\end{equation}
In the original time variable $t$, this implies that $\lambda(t)$ goes to zero in finite time $T$ with the asymptotic 
$$\lambda(t) \sim \frac{\sqrt{T-t}}{|\log(T-t)|}.$$

\noindent $(iv)$ \textit{Decomposition of the flow and modulation equations.} Let the approximate solution $Q_b$ be given by \eqref{def:Qb_i} which by construction generates an approximate solution to the renormalized flow \eqref{eq:wys_i},
$$\Psi_b = \ps Q_b - \Delta Q_b + b \Lambda Q_b + \frac{3}{y^2}\sin(2Q_b) = Mod(t) + O(b_1^{2L + 2+\frac 32}),$$
where the modulation equation term is roughly of the form
$$Mod(t) = \sum_{i = 1}^L \big[(b_i)_s + (2i - 2+C\sqrt{b_1})b_1b_i - b_{i+1}\big]T_i.$$
We localize $Q_b$ in the zone $y \leq B_1$ to avoid the irrelevant growing tails for $y \gg \frac{1}{\sqrt{b_1}}$. We then take initial data of the form
$$u_0(y) = Q_{b(0)}(y) + q_0(y),$$
where $q_0$ is small in some suitable sense and $b(0)$  is chosen to be close to the exact solution \eqref{eq:solbk_i}. By a standard modulation argument, we introduce the decomposition of the flow
\begin{equation}\label{eq:dec_i}
u(r,t) = w(y,s) = \big(Q_{b(s)} + q\big)(y,s) = \big(Q_{b(t)} + v\big)\left(\frac{r}{\lambda(t)},t\right),
\end{equation}
where $L+1$ modulation parameters $(b(t), \lambda(t))$ are chosen in order to manufacture the orthogonality conditions:
\begin{equation}\label{eq:orh_i}
\left<q, \Ls^i \Phi_M \right> = 0, \quad 0 \leq i \leq L,
\end{equation}
where $\Phi_M$ (see \eqref{def:PhiM}) is some fixed direction depending on some large constant $M$, generating an approximation of the kernel of the powers of $\Ls$. This orthogonal decomposition \eqref{eq:dec_i}, which follows from the implicit function theorem, allows us to compute the modulation equations governing the parameters $(b(t), \lambda(t))$ (see Lemmas \ref{lemm:mod1} and \ref{lemm:mod2} below),
\begin{equation}\label{eq:mod_i}
\left|\frac{\lambda_s}{\lambda} + b_1\right| +  \sum_{i = 1}^L \big|(b_i)_s + (2i - 2+C\sqrt{b_1})b_1b_i - b_{i+1}\big| \lesssim \|q\|_{\Hs_{2\Bbbk}} + b_1^{L + 1 + \frac 34 +\frac{\eta}{4}},
\end{equation}
where 
\begin{align}\label{locnorm}
\|q\|_{\Hs_{2\Bbbk}}:= \int |\Ls^\Bbbk f|^2 + \int \frac{|\As (\Ls^{\Bbbk - 1} f)|^2}{y^2} + \sum_{m = 0}^{\Bbbk - 1} \int \frac{|\Ls^m f|^2}{y^4(1 + y^{4(\Bbbk - 1 - m)})} + \sum_{m = 0}^{\Bbbk-2}\frac{|\As(\Ls^{m}f)|^2}{y^6(1 + y^{4(\Bbbk - m - 2)})}.
\end{align}
measures a spatially localized norm of the error $q$.\\

\noindent $(v)$ \textit{Control of Sobolev norms.} According to \eqref{eq:mod_i}, we need to show that local norms of $q$ are under control and do not perturb the dynamical system \eqref{sys:bk_i}. This is achieved via high order mixed energy estimates which provide controls of the Sobolev norms adapted to the linear flow and based on the iterative powers of the linear operator $\Ls$. In particular, we have the following coercivity of the high energy under the orthogonality conditions \eqref{eq:orh_i} (see Lemma \ref{lemm:coeLk} and Lemma \ref{lemm:interbounds}),
$$\Es_{2L+2}(s)  = \int |\Ls^{L+1} q|^2 \gtrsim \int |\nabla^{2L+2} q|^2 + \int \frac{|q|^2}{1 + y^{4L+4}},$$
where $L+1$ is given by \eqref{def:kbb}. Here the factorization \eqref{eq:facL_i} will help to simplify the proof. As in \cite{RRmihes12}, \cite{RSapde2014} and \cite{MRRcmj2016}, the control of $\Es_{2L+2}$ is done through the use of the linearized equation in the original variables $(r,t)$, i.e. we work with $v$ in \eqref{eq:dec_i} and not $q$. The energy estimate is of the form (see Proposition \ref{prop:E2k})
\begin{equation}\label{eq:Ek_i}
\frac{d}{ds} \left\{\frac{\Es_{2L+2}}{\lambda^{4L-1}}\right\} \lesssim \frac{b_1^{2L + 1 + \frac{3}{2}+\frac{\eta}{4}}}{\lambda^{4L-1}},
\end{equation}
where the right hand side is controlled by the size of the error $\Psi_b$ in the construction of the approximate profile $Q_b$ above. An integration of \eqref{eq:Ek_i} in time by using initial smallness assumptions, $b_1 \sim \frac{1}{s^{\frac 23}}$ and $\lambda(s) \sim e^{-3s^{\frac 13}}$ yields the estimate
$$
\int |\nabla^{2L+2} q|^2 + \int \frac{|q|^2}{1 + y^{4L+4}} \lesssim \Es_{2L+2}(s) \lesssim b_1^{2L +\frac{3}{2}+\frac{\eta}{4}},$$
which is good enough to control the local norms of $q$ and close the modulation equations \eqref{eq:mod_i}.

Note that we also need to control lower energies $\Es_{2m}$ for $2 \leq m \leq L$ because the control of the high energy $\Es_{2L+2}$ alone is not enough to handle the nonlinear term. In particular, we exhibit a Lyapunov functional with the dynamical estimate
$$\frac{d}{ds}\left\{\frac{\Es_{2m}}{\lambda^{4m - 7}}\right\} \lesssim \frac{b_1^{2(m-1) +1+(\frac{1}{2}+2m)\frac{\eta}{2}}}{\lambda^{4m - 7}},$$
then, an integration in time yields
$$ \Es_{2m}(s) \lesssim b_1^{2(m-1) +(\frac{1}{2}+2m)\frac{\eta}{2}},$$
which is enough to control the nonlinear term. Let us remark that the condition $m \geq 2$ ensures $4m - 7 > 0$ so that $\Es_{2m}$ is always controlled.\\

The above scheme designs a bootstrap regime (see Definition \ref{def:1} for a precise definition) which traps blowup solution with speed \eqref{eq:quanblrate}, and the proof of Theorem \ref{Theo:1} follows.\\

\bigskip 

The paper is organized as follows. In Section \ref{sec:2}, we give the construction of the approximate solution $Q_b$ of \eqref{Pb} and derive estimates on the generated error term $\Psi_b$ (Proposition \ref{prop:1}) as well as its localization (Proposition \ref{prop:localProfile}). In Section \ref{sec:3}, we set up the bootstrap argument, and in the last section we close the bootsrap bounds which will imply Theorem \ref{Theo:1}.

\section{Construction of an approximate profile.}\label{sec:2}
This section is devoted to the construction of a suitable approximate solution to \eqref{Pb}  by using the same approach developed in \cite{RRmihes12}. Similar approachs can also be found in \cite{RScpam13}, \cite{HRapde12}, \cite{RSma14}, \cite{Sjfa12}, \cite{MRRcmj2016}. The key to this construction is the fact that the linearized operator $\Ls$ around $Q$ is completely explicit in the radial setting thanks to the explicit formulas of the kernel elements.

Following the scaling invariance of \eqref{Pb}, we introduce the following change of variables:
\begin{equation}\label{def:simiVars}
w(y,s) = u(r,t), \quad y = \frac{r}{\lambda(t)}, \quad \frac{ds}{dt} = \frac{1}{\lambda^2(t)},
\end{equation}
which leads to the following renormalized flow:
\begin{equation}\label{eq:wys}
\partial_sw = \partial_y^2 w + \frac{6}{y}\partial_y w  -\frac{\lambda_s}{\lambda}\Lambda w - \frac{3}{y^2}\sin(2w),
\end{equation}
where $\lambda_s = \frac{d\lambda}{ds}$. Noticing that in the setting \eqref{def:simiVars}, we have 
$$\partial_r u(r,t) = \frac{1}{\lambda(t)}\partial_y w(y,s)$$
and since we deal with the finite time blowup of the problem \eqref{Pb}, we would naturally impose the condition 
$$\lambda(t) \to 0 \quad \text{as} \quad t \to T,$$
for some $T \in (0, +\infty)$. Hence, $\partial_r u(r,t)$ blows up in finite time $T$.

Let us assume that the leading part of the solution of \eqref{eq:wys} is given by the harmonic map $Q$, which is the unique  solution (up to scaling) of the equation
\begin{equation}\label{eq:Qy}
Q'' + \frac{6}{y}Q' - \frac{3}{y^2}\sin(2Q) = 0, \quad Q(0) = 0, \; Q'(0) = 1.
\end{equation}
We aim at constructing an approximate solution of \eqref{eq:wys} close to $Q$. The natural way is to linearize equation \eqref{eq:wys} around $Q$, which generates the Schr\"odinger operator defined by \eqref{def:Lc}. Let us now recall the main properties of $\Ls$ in the following subsection.

\subsection{Structure of the linearized Hamiltonian.}
In this subsection, we recall the main properties of the linearized Hamiltonian close to $Q$, which is the heart of both construction of the approximate profile and the derivation of the coercivity properties serving for the high Sobolev energy estimates. Let us start by recalling the following result from Biernat \cite{BIEnon2015}, which gives the asymptotic behavior of the harmonic map $Q$:
\begin{lemma}[Development of the harmonic map $Q$] Let $d=7$, there exists a unique solution $Q$ to equation \eqref{eq:Qy}, which admits the following asymptotic behavior: For any $k \in\N^*$,\\
$(i)$ (Asymptotic behavior of $Q$) 
\begin{equation}\label{eq:asymQ}
Q(y) = \left\{\begin{array}{ll}
y + \sum \limits_{i = 1}^kc_iy^{2i + 1} + \Oc(y^{2k + 3}) &\text{as}\quad y \to 0, \\
&\\
\dfrac{\pi}{2} - \dfrac{a_0}{ y^2}+\dfrac{a_1}{y^3}+ \Oc\left(\dfrac 1{y^{4}}\right)\quad &\text{as} \quad y \to + \infty,
\end{array}
\right.
\end{equation}
where $a_0 > 0$ and $a_1>0$.\\
$(ii)$ (Degeneracy)
\begin{equation}\label{eq:asymLamQ}
\Lambda Q > 0, \quad \Lambda Q(y) = \left\{\begin{array}{ll}
y + \sum\limits_{i = 1}^kc_i' y^{2i + 1} + \Oc(y^{2k + 3}) &\text{as}\quad y \to 0, \\
&\\
\dfrac{2a_0}{ y^2}-\dfrac {3a_1}{y^3} + \Oc\left(\dfrac 1{y^4}\right)\quad &\text{as} \quad y \to + \infty,
\end{array}
\right.
\end{equation}
\end{lemma}

\begin{proof} The proof of \eqref{eq:asymQ} is done through the introduction of the variables $x = \log y$ and $v(x) = 2Q(y) - \pi$ and consists of the phase portrait analysis of the autonomous equation 
$$v''(x) + 5v'(x) + 5\sin(v(x)) = 0.$$
All details of the proof can be found at pages 184-185 in \cite{BIEnon2015}. The proof of \eqref{eq:asymLamQ} directly follows from the expansion \eqref{eq:asymQ}. However, the author was not interested of proving that $a_1>0$. Nevertheless, his proof for $a_0>0$ can easily be adapted to prove that $a_1>0$.
\end{proof}

The linearized operator $\Ls$ displays a remarkable structure given by the following lemma:
\begin{lemma}[Factorization of $\Ls$] \label{lemm:factorL}  Let $d=7$ and define the first order operators
\begin{align}
\As w  &= -\partial_y w + \frac{V}{y}w = - \Lambda Q \partial_y \left(\frac{w}{\Lambda Q}\right), \label{def:As}\\ 
\As^* w &= \frac{1}{y^6}\partial_y \big(y^6w\big) + \frac{V}{y}w  =  \frac{1}{y^6\Lambda Q} \partial_y \left(y^6 \Lambda Q w\right),\label{def:Astar}
\end{align}
where
\begin{equation}\label{eq:asympV}
V(y) := \Lambda \log(\Lambda Q) =  \left\{\begin{array}{ll}
1 + \Oc(y^2)\quad &\text{as}\quad y \to 0, \\
&\\
-2 + \Oc\left(\dfrac 1{y^{2}}\right)+ \Oc\left(\dfrac 1{y}\right)\quad &\text{as} \quad y \to + \infty.
\end{array}
\right.
\end{equation}
We have
\begin{equation}\label{eq:reLAAst}
\Ls = \As^* \As, \quad \tilde{\Ls} = \As \As^*,
\end{equation}
where $\tilde{\Ls}$ stands for the conjugate Hamiltonian.
\end{lemma}

\begin{remark} The adjoint operator $\As^*$ is defined with respect to the Lebesgue measure
$$\int _0^{+\infty}(\As u) w y^{6}dy = \int_{0}^{+\infty} u(\As^* w)y^{6}dy.$$
\end{remark}

\begin{remark} We have 
\begin{equation}\label{eq:relLsLam}
\Ls(\Lambda w) = \Lambda (\Ls w) + 2 \Ls w - \frac{\Lambda Z}{y^2}w.
\end{equation}
Since $\Ls(\Lambda Q) = 0$, one can express the definition of $Z$ through the potential $V$ as follows: 
\begin{equation}\label{def:ZbyV}
Z(y) = V^2 + \Lambda V + 5V.
\end{equation}
Let $\tilde{Z}$ be defined by 
\begin{equation}\label{def:LstilbyZtil}
\tilde{\Ls} = -\partial_{yy} - \frac{6}{y}\py + \frac{\tilde{Z}}{y^2},
\end{equation}
then, a direct computation yields
\begin{equation}\label{def:ZtilbyV}
\tilde{Z}(y) = (V + 1)^2 + 5(V+1) - \Lambda V.
\end{equation}
\end{remark}

\bigskip

\noindent From \eqref{def:As} and \eqref{def:Astar}, we see that the kernel of $\As$ and $\As^*$ are explicit:
$$\left\{\begin{array}{ll}
\As w = 0 & \quad \text{if and only if}\quad w \in \text{Span}(\Lambda Q),\\
\As^* w = 0& \quad \text{if and only if}\quad w \in \text{Span}\left(\frac 1{y^{6}\Lambda Q}\right).\\
\end{array}\right.$$
Hence, the elements of the kernel of $\Ls$ are given by 
\begin{equation}\label{eq:kernalLc}
\Ls w = 0 \quad \text{if and only if}\quad w \in \text{Span}(\Lambda Q, \Gamma),
\end{equation}
where $\Gamma$ can be found from the Wronskian relation
\begin{equation}\label{eq:relWrons}
\Gamma' \Lambda Q - \Gamma (\Lambda Q)' = \frac{1}{y^{6}},
\end{equation}
that is
\begin{equation*}
\Gamma(y) = \Lambda Q(y) \int_1^y \frac{d\xi}{\xi^{6} (\Lambda Q(\xi))^2},
\end{equation*}
which admits the asymptotic behavior:
\begin{equation}\label{eq:asymGamma}
\Gamma(y) = \left\{\begin{array}{ll}
\dfrac{1}{7 y^{6}} + \Oc(y)\;\; &\text{as}\;\; y \to 0, \\
&\\
\dfrac{1}{2a_0y^3} + \Oc\left(\dfrac 1{y^5}\right)\;\; &\text{as} \;\; y \to + \infty,
\end{array}
\right.
\end{equation}
From \eqref{eq:kernalLc}, we may formally invert $\Ls$ as follows:
\begin{equation}\label{eq:invLc}
\Ls^{-1}f =  -\Gamma(y)\int_0^y f(x)\Lambda Q(x) x^{6}dx + \Lambda Q(y)\int_0^y f(x) \Gamma(x)x^{6}dx. 
\end{equation}
We define the following adapted derivatives of a function $f$ for all $i\geq 0$:
\begin{eqnarray}\label{adapt_deriv}
f_0=f,\quad f_{i+1}=\left\{\begin{array}{ll}
\As^*f_i\quad \mbox{ if } i \mbox{ is odd }\\
\As f_i\quad \mbox{ if } i \mbox{ is even }.
\end{array}
\right.
\end{eqnarray}
We introduce the following formal notation:
\begin{eqnarray}\label{Adef}
f_i=A^if_0=\left\{\begin{array}{ll}
\As\Ls^if_0\quad \mbox{ if } i \mbox{ is odd }\\
\Ls^i f_0\quad \mbox{ if } i \mbox{ is even }.
\end{array}
\right.
\end{eqnarray}
To avoid confusion we insist on the fact that this notation will be used only for the letters $f$, $v$, and $q$. If the notation $_i$ is used with an other letter it will have another meaning.
The factorization of $\Ls$ allows us to compute $\Ls^{-1}$ in an elementary two step processe that will help us to avoid tracking the cancellation in the formula \eqref{eq:invLc} induced by the Wronskian relation when estimating the growth of $\Ls^{-1}f$. In particular, we have the following:
\begin{lemma}[Inversion of $\Ls$] \label{lemm:inversionL} Let  $f$ be a $\Cc^\infty$ radially symmetric function and $w = \Ls^{-1}f$ be given by \eqref{eq:invLc}, then 
\begin{equation}\label{eq:relaAL}
\Ls w = f, \quad \As w = \frac{1}{y^{6}\Lambda Q} \int_0^yf(x) \Lambda Q(x) x^6 dx, \quad w = -\Lambda Q\int_0^y\frac{\As w(x)}{\Lambda Q(x)}dx.
\end{equation}
\end{lemma}
\begin{proof} From the relation \eqref{eq:relWrons}, we compute
$$\As \Gamma = - \frac{1}{y^{6}\Lambda Q}.$$
Applying $\As$ to \eqref{eq:invLc} and using the cancellation $\As (\Lambda Q) = 0$, we obtain
$$\As w = \frac{1}{y^{6}\Lambda Q} \int_0^yf(x) \Lambda Q(x) x^{6}dx.$$
From the definition \eqref{def:As} of $\As$, we write
$$w = - \Lambda Q \int_0^y \frac{\As w}{\Lambda Q} dx.$$
This concludes the proof of Lemma \ref{lemm:inversionL}.
\end{proof}

\subsection{Admissible functions.} We define a class of admissible functions which displays a suitable behavior both at the origin and infinity.
\begin{definition}[Admissible function] \label{def:Admitfunc} Fix $\gamma > 0$, we say that a smooth function $f \in \Cc^\infty(\Rb_+, \Rb)$ is admissible of degree $(p_1, p_2) \in \mathbb{N} \times \mathbb{Z}$ if \\
$(i)\;$ $f$ admits a Taylor expansion to all orders around the origin, 
 $$f(y) = \sum_{k = p_1}^p c_ky^{2k + 1} + \Oc(y^{2p + 3});$$
$(ii)\;$ $f$ and its derivatives admit the bounds, for $y \geq 1$,
$$\forall k \in \mathbb{N}, \quad |\partial^k_y f(y)| \lesssim y^{2p_2 -2- k}.$$
\end{definition}
\begin{remark} Note from \eqref{eq:asymLamQ} that $\Lambda Q$ is admissible of degree $(0,0)$.
\end{remark}

One note that $\Ls$ naturally acts on the class of admissible function in the following way:
\begin{lemma}[Action of $\Ls$ and $\Ls^{-1}$ on admissible functions] \label{lemm:actionLL} Let $f$ be an admissible function of degree $(p_1, p_2) \in \mathbb{N} \times \mathbb{Z}$, then:\\
$(i)\;$ $\Lambda f$ is admissible of degree $(p_1, p_2)$.\\
$(ii)\,$ $\Ls f$ is admissible of degree $(\max\{0,p_1 - 1\}, p_2 - 1)$.\\
$(iii)\,$ $\Ls^{-1}f$ is admissible of degree $(p_1 + 1, p_2 + 1)$.
\end{lemma}
\begin{proof}
$(iii)$ We aim at proving that if $f$ is admissible of degree $(p_1, p_2)$, then $w = \Ls^{-1}f$ is admissible of degree $(p_1 + 1, p_2 + 1)$. To do so, we use Lemma \ref{lemm:inversionL} to estimate \\
- for $y \ll 1$, 
$$\As w = \frac{1}{y^6\Lambda Q}\int_0^yf\Lambda Q x^6dx =  \Oc \left(\frac{1}{y^7}\int_0^y x^{2p_1 +8}dx\right) = \Oc(y^{2p_1 + 2}),$$
$$w = - \Lambda Q \int_0^y\frac{\As w}{\Lambda Q}dx = \Oc\left(y\int_0^y x^{2p_1 + 1}dx \right) = \Oc(y^{2(p_1 + 1) + 1}),$$
- for $y \geq 1$, 
$$\As w =  \Oc \left(\frac{1}{y^4}\int_0^y x^{2p_2 +2}dx\right) = \Oc(y^{p_2-1}),$$
$$w = \Oc\left(\frac{1}{y^2}\int_0^y x^{2p_2 + 1}\right) = \Oc(y^{2(p_2 + 1) - 2}).$$
From the last formula in \eqref{eq:relaAL} and \eqref{eq:asympV}, we estimate
$$\partial_y w = - \partial_y \Lambda Q \int_0^y \frac{\As w}{\Lambda Q}dx - \As w = -\frac{\partial_y \Lambda Q}{\Lambda Q} w  - \As w = \Oc(y^{2(p_2 + 1) - 2 - 1}).$$
Using $\Ls w = f$, we get
$$\partial_{yy}w = \Oc \left(\frac{|\partial_y w|}{y} + \frac{|w|}{y^2} + |f|\right) =\Oc (y^{2(p_2+1) - 2 - 2}).$$
By taking radial derivatives of $\Ls w = f$, we obtain by induction
$$|\partial_y^k w| \lesssim y^{2(p_2 + 1) - 2 - k}, \quad k \in \mathbb{N}, \; y \geq 1.$$
$(i) - (ii)$ are a consequence of Definition \ref{def:Admitfunc}. One can prove them by contradiction, by supposing that $\Lambda f$ and $\Ls f$ are not admissible of degree $(p_1,p_2)$ and respectively $(\max\{0,p_1 - 1\}, p_2 - 2)$. Hence, this will induce that $f$ is not admissible of degree $(p_1,p_2)$, from the action of $\Ls^{-1}$, which is a contradiction.
This concludes the proof of Lemma \ref{lemm:actionLL}.
\end{proof}

The following lemma is a consequence of Lemma \ref{lemm:actionLL}:
\begin{lemma}[Generators of the kernel of $\Ls^k$] \label{lemm:GenLk} Let the sequence of profiles 
\begin{equation}\label{def:Tk}
T_k = (-1)^k\Ls^{-k} \Lambda Q, \quad k \in \mathbb{N},
\end{equation}
then\\
$T_k$ is admissible of degree $(k,k)$ for $k \in \mathbb{N}$.\\
\end{lemma}
\begin{proof} $(i)$ We note from \eqref{eq:asymLamQ} that $\Lambda Q$ is admissible of degree $(0,0)$. By induction and part $(iii)$ of Lemma \ref{lemm:actionLL}, the conclusion then follows.\\
\end{proof}
Unlike the case $d\geq 8$, this particular case of $d=7$ requires that we introduce another notion of admissible fucntion.
\begin{definition}[$b_1$-admissible functions]\label{def:b1Admitfunc}
Let $f$ be a smooth function in $C^\infty(\Rb_+,\Rb)$, $f$ is called $b_1$-admissible of degree $(p_1,p_2)\in\N\times\Z$ if the following hold:
\begin{itemize}
\item For $y\leq 1$, 
\begin{equation}\label{tayl_f_b1-adm}
f(b_1,y)=\sum_{j=1}^J h_j(b_1)u_j(y),
\end{equation}
with $J\in\N^*$, $h_j(b_1)$ a smooth function away from the origin such that
\begin{align}\label{hjdef}
\forall i\in\N, \quad |\p^i_{b_1}h_j|\lesssim b_1^{\frac{1}{2}-i},
\end{align}
and $u_j(y)$ a smooth function such that
\begin{align}\label{ujdef}
\forall y\leq 1 \quad u_j(y)=\sum_{k=p_1}^pc_{k,j}.y^{2k+1}+O(y^{2p+3}).
\end{align}
\item  For all $y\geq2$ and all $i\geq0$

\begin{align}\label{fidef}
| A^if|\lesssim (\sqrt{b_1}y^{2p_2-i-2}+y^{2p_2-3-i})\indic_{\{y\leq 2B\}}+\Big(|\log y|y^{2p_2-4-i}+\frac{y^{2p_2-5-i}}{\sqrt{b_1}}\Big)\indic_{\{y\geq 2B\}},
\end{align}
where $B$ is defined in \eqref{defB} and for all $y\geq 2$, $i\geq0$ and $l\geq0$
\begin{align}\label{fbldef}
|\p_{b_1}^lA^if|\lesssim b_1^{\frac{1}{2}-l}y^{2p_2-2-i}\indic_{\{y\leq 2B\}}+b_1^{-\frac{1}{2}-l}y^{2p_2-5-i}\indic_{\{y\geq 2B\}}.
\end{align}
\end{itemize}
\end{definition}
In addition, we notice also that $\Ls$ naturally acts on the class of $b_1$-admissible function in the following way:
\begin{lemma}[Action of $\Ls$ and $\Ls^{-1}$ on $b_1$-admissible functions] \label{lemm:actionLLb_1} Let $f$ be a $b_1$-admissible function of degree $(p_1, p_2) \in \mathbb{N} \times \mathbb{Z}$, then: for all $k\geq1$\\
$(i)\,$ $\Ls^k f$ is $b_1$-admissible of degree $(\max\{0,p_1 - k\}, p_2 - k)$.\\
$(ii)\,$ $\Ls^{-k}f$ is $b_1$-admissible of degree $(p_1 + k, p_2 + k)$.
\end{lemma}

\begin{proof}
The proof is similar to the proof of Lemma \ref{lemm:actionLL} except that here we use the fact that $\Ls$ and $\Ls^{-1}$ are independent of $b_1$.
\end{proof}

Before we strart the construction of the approximate blow up profile, we will construct a familly of functions that will minimize the growth of the tails of the  approximate blow up profile at infinity.
Indeed, to do so we introduce the following family of smooth functions.
\begin{lemma}[Slowly growing tails]\label{lemm:slowgrowtails}
Let $(T_k)_{k\geq1}$ be the family of functions defined iteratively by \eqref{def:Tk}. Then the following family of profiles for all $k\geq1$ is $b_1$-admissible of degree $(k,k)$:
\begin{align}\label{def:tetak}
\theta_k=\Lambda T_k-(2k-2)T_k-(-1)^{k+1}\Ls^{-k+1}\Sigma_{b_1},
\end{align}
where 
\begin{align}\label{def:Sigmab}
\Sigma_{b_1}=\Ls^{-1}[-C_{b_1}\Lambda Q\chi_{B_0}-4a_0C_1(1-\chi_B)\Gamma].
\end{align}

With 
\begin{equation}\label{eq:Cb1}
C_{b_1}=\frac{4C_1\sqrt{b_1}}{a_0}+O(b_1),
\end{equation}

\begin{align}\label{defB}
B=\frac{\Big(1+\int_1^2\chi dx\Big)^3\Big(\frac{1}{3}+\int_1^2x^2\chi dx\Big)^2}{\Big(\frac{1}{2}+\int_1^2x\chi dx\Big)^5\sqrt{b_1}}+O(\sqrt{b_1}),
\end{align}

and $a_0,\quad C_1=\frac{3a_1}{2}>0$ are constants coming from the asymptotic behavior of the stationnary solution $Q$ defined in \eqref{eq:asymLamQ}.
\end{lemma}

\begin{proof}

\begin{itemize}
\item[\underline{Step 1: Structure of $T_1$}.]
By definition \eqref{def:Tk} $T_1=-\Ls^{-1}\Lambda Q$ and thanks to Lemma \ref{lemm:inversionL} we can inverse $\Ls$ and compute $T_1$:

\begin{equation}\label{T1}
T_1(y)=\left\{\begin{array}{ll}
-\frac{y^3}{18}+\Oc(y^5)\;\; &\text{as}\;\; y \to 0, \\
&\\
-C_0+\dfrac{C_1}{y} + \Oc\left(\frac{|\log(y)|}{y^2}\right)\;\; &\text{as} \;\; y \to + \infty,
\end{array}
\right.
\end{equation}
with $C_0=\frac{a_0}{3}$, $C_1=\frac{3a_1}{2}$,
and similarly:

\begin{equation}\label{LT1}
\Lambda T_1=\left\{\begin{array}{ll}
-\frac{y^3}{6}+\Oc(y^5)\;\; &\text{as}\;\; y \to 0, \\
&\\
\dfrac{-C_1}{y} + \Oc\left(\frac{|\log(y)|}{y^2}\right)\;\; &\text{as} \;\; y \to + \infty.
\end{array}
\right.
\end{equation}
Now we will prove by induction that $\theta_k$ are $b_1$-admissible of degree $(k,k)$.
\item[\underline{Step 2: Intialization $k=1$}.]
We will prove here that $\theta_1=\Lambda T_1-\Sigma_{b_1}$ is a $b_1$-admissible function of degree $(1,1)$. To do so we need to estimate first $\Sigma_{b_1}$, but before estimating $\Sigma_{b_1}$ we will give a brief explanation of why do we introduce this correction $\Sigma_{b_1}$.
Actually, the introduction of $\Sigma_{b_1}$ will allow us to derive the law of $b_1$.  We will derive formally the law of $b_1$ here in this part and rigourosly later. Indeed, we are going to construct a solution of the form $Q_{b_1}=Q+\sum_{k=1}^Lb_1^kT_k+\sum_{k=2}^{L+2}S_k(b_1,y)$, the $T_k$ are defined by \eqref{def:Tk} and $S_k$ will solve some elliptic problem. In order to minimize the growth of the tails of the profiles $S_k$ at infinity we introduce the correction $\Sigma_{b_1}$.

Approximately $S_2$ will solve:
$$\Ls S_2= -\Big(\underbrace{\Lambda T_1+\frac{(b_1)_s}{b_1^2}T_1}_{\sim\theta_1}\Big)+l.o.t,$$
to minimize the growth of $S_2$ at infinity we need to cancel the growth $\frac{1}{y}$ of $\Lambda T_1$ at infinity. To do so, set $\frac{(b_1)_s}{b_1^2}=-C_{b_1}$. If we use $-C_{b_1} T_1$ to cancel the growth $\frac{1}{y}$ of $\Lambda T_1$ in the zone $y=\frac{1}{\sqrt{b_1}}$ formally by using that $T_1=C_0+\dfrac{C_1}{y} + \Oc\left(\frac{|\log(y)|}{y^2}\right)\;\; \text{as} \;\; y \to + \infty$ we deduce that $C_{b_1}\sim \sqrt{b_1}$. However, we need to cancel the growth in the whole zone $\{y\geq\frac{1}{\sqrt{b_1}}\}$, and using $C_{b_1}T_1$ will not be adapted.
Hence, we introduce a new function $\Sigma_{b_1}$ that will behave like $C_{b_1}T_1$ in the zone $\{y\leq\frac{1}{\sqrt{b_1}}\}$ and like $\frac{1}{y}$ in the zone $\{y\geq\frac{1}{\sqrt{b_1}}\}$ to cancel the growth of $\Lambda T_1$.
If one naively construct $\Sigma_{b_1}$ by using matching asymptotic then the profile $S_2$ that we will obtain will not be smooth.
Hence, to overcome all these difficulties we construct $\Sigma_{b_1}$ by solving :
$$\Ls\Sigma_{b_1}=-C_{b_1}\Lambda Q\chi_{B_0}-\alpha(1-\chi_B)\Gamma,$$
where $C_{b_1}$,$\alpha$ and $B$ will be determined to force $\Sigma_{b_1}$ to behave like $C_{b_1}T_1$ in the zone $\{y\leq\frac{1}{\sqrt{b_1}}\}$ and like $\frac{1}{y}$ in the zone $\{y\geq\frac{1}{\sqrt{b_1}}\}$.
In addition, if we apply the linear operator $\Ls$ twice to $\Sigma_{b_1}$, we get that $\Ls^2(\Sigma_{b_1})=0$ for all $y\in (0,B_0)\cup(2B,\infty)$ since $\Lambda Q$ and $\Gamma$ are in the kernel of $\Ls$.
Indeed, by using \eqref{eq:invLc} we deduce
\begin{align}
\Sigma_{b_1}&=C_{b_1}\Gamma(y)\int_0^y (\Lambda Q)^2\chi_{B_0} x^{6}dx - C_{b_1}\Lambda Q\int_0^y \Lambda Q\chi_{B_0} \Gamma x^{6}dx\nonumber\\
&+\alpha \Gamma\int_0^y \Gamma\Lambda Q(1-\chi_{B}) x^{6}dx - \alpha\Lambda Q\int_0^y (1-\chi_{B})\Gamma^2x^{6}dx.
\end{align}

Hence, we will fix $C_{b_1}$, $\alpha$, and $B$ to get that 
\begin{align}
\Sigma_{b_1}=\left\{\begin{array}{ll}
C_{b_1} T_1\quad y\leq B_0\\
-\frac{C_1}{y}+O\Big(\frac{1}{y^2}\Big)+O\Big(\frac{1}{\sqrt{b_1}y^3}\Big)\quad y\geq 2B.
\end{array}
\right.
\end{align}

Indeed, suppose $B\geq B_0$, when $y\geq 2B$ we compute
\begin{align}
&\hspace{1cm}\Sigma_{b_1}=-\Gamma\Bigg[\alpha B^2\Big(\frac{1}{2}+\int_1^2x\chi dx\Big)-C_{b_1}\int_0^\infty(\Lambda Q)^2\chi_{B_0}x^6dx\Bigg]\nonumber\\
&\hspace{0.5cm}-\Lambda Q\Bigg[C_{b_1}\int_0^{+\infty}\Gamma\Lambda Q\chi_{B_0}x^6dx-\alpha\frac{B}{4a_0^2}\Big(1+\int_1^2\chi dx\Big)\Bigg]-\frac{\alpha}{4a_0y}+O\Bigg(\frac{1}{\sqrt{b_1}y^3}\Bigg).
\end{align}
Hence, if we choose $C_{b_1}$, $\alpha$ and $B$ such that
\begin{align}
\alpha B^2\Big(\frac{1}{2}+\int_1^2x\chi dx\Big)-C_{b_1}\int_0^\infty(\Lambda Q)^2\chi_{B_0}x^6dx&=0\\
C_{b_1}\int_0^{+\infty}\Gamma\Lambda Q\chi_{B_0}x^6dx-\alpha\frac{B}{4a_0^2}\Big(1+\int_1^2\chi dx\Big)&=0,
\end{align}
and
$$\alpha=4a_0C_1.$$
It follows that
\begin{align}\label{B}
B=\frac{\int_0^\infty(\Lambda Q)^2\chi_{B_0}x^6dx\Big(1+\int_1^2\chi dx\Big)}{4a_0^2\int_0^{+\infty}\Gamma\Lambda Q\chi_{B_0}x^6dx\Big(\frac{1}{2}+\int_1^2x\chi dx\Big)},
\end{align}
\begin{align}\label{C_b}
C_{b_1}=\frac{C_1\int_0^\infty(\Lambda Q)^2\chi_{B_0}x^6dx\Big(1+\int_1^2\chi dx\Big)^2}{4a_0^3\Bigg(\int_0^{+\infty}\Gamma\Lambda Q\chi_{B_0}x^6dx\Bigg)^2\Big(\frac{1}{2}+\int_1^2x\chi dx\Big)},
\end{align}
and
$$\alpha=4a_0C_1.$$
In addition, we estimate
\begin{align}
\int_0^\infty(\Lambda Q)^2\chi_{B_0}x^6dx=\frac{4a_0^2B_0^3}{3}\Big(\frac{1}{3}+\int_1^2x^2\chi dx\Big)+O(B_0^2),\\
\int_0^\infty\Lambda Q\Gamma\chi_{B_0}x^6dx=B_0^2\Big(\frac{1}{2}+\int_1^2 x\chi_{B_0}dx\Big)+O(B_0).
\end{align}
Hence, by the choice of $B_0=\frac{C_\chi}{\sqrt{b_1}}$ with $C_\chi=\frac{(1+\int_1^2\chi dx)^2(1/3+\int_1^2x^2\chi dx)}{(1/2+\int_1^2x\chi dx)^3}$ and plugging those above estimates in \eqref{B} and \eqref{C_b} we deduce that
\begin{align}
C_{b_1}=\frac{C_1}{a_0}\sqrt{b_1}+O(b_1),
\end{align}
and
\begin{align}
B=B_0\frac{\Big(1+\int_1^2\chi dx\Big)\Big(\frac{1}{3}+\int_1^2x^2\chi dx\Big)}{\Big(\frac{1}{2}+\int_1^2x\chi dx\Big)^2}+O(\sqrt{b_1}).
\end{align}

Now we can estimate $\Sigma_{b_1}$ everywhere.
We estimate first $\Sigma_{b_1}$ for $B_0\leq y\leq\frac{16}{3}B_0$
\begin{align}
\Sigma_{b_1}=O(\sqrt{b_1})+O\Big(\frac{1}{y}\Big).
\end{align}
Consequently, we deduce that
\begin{align}\label{sigmab_est}
\Sigma_{b_1}=\left\{\begin{array}{ll}
C_{b_1} T_1\quad y\leq B_0\\
O(\sqrt{b_1})+O\Big(\frac{1}{y}\Big) \quad B_0\leq y \leq 2B\\
\frac{C_1}{y}+O\Big(\frac{1}{y^2}\Big)+O\Big(\frac{1}{\sqrt{b_1}y^3}\Big)\quad  y\geq 2B.
\end{array}
\right.
\end{align}
One can choose $\chi$ so that $B\geq B_0$, this choice is possible since the constant $\tilde{C}_\chi=\frac{\Big(1+\int_1^2\chi dx\Big)\Big(\frac{1}{3}+\int_1^2x^2\chi dx\Big)}{\Big(\frac{1}{2}+\int_1^2x\chi dx\Big)^2}$ depends on the decay of the cut-off function in the zone $x\in(1,2)$.
And if one take the slowest decay or the fastest decay possible for a cut-off function $\chi$ the constant $\tilde{C}_\chi$ will be bigger than 1 for those choices of $\chi$.

Hence, we deduce the following on $\Sigma_{b_1}$, it will have for $y\leq 1$ the taylor expansion of a $b_1$-admissible function \eqref{tayl_f_b1-adm}, with $J=1$, $h_j(b_1)=C_{b_1}$, and $u_1=T_1$, which implies that the degree $p_1$ of $\Sigma_{b_1}$ is 1 close to zero. In addition, an easy computation yields the bound for all $l\geq0$
$$|\p_{b_1}^l C_{b_1}|\lesssim b_1^{\frac{1}{2}-l},$$
which implies \eqref{hjdef}.
Hence, now for $1\leq y\leq 2B$ we have from \eqref{sigmab_est} and \eqref{LT1} that

$$\theta_1=\Lambda T_1-\Sigma_{b_1}=O(\sqrt{b_1})+O\Big(\frac{1}{y}\Big),$$
and for $y\geq 2B$,
$$\theta_1=O\Big(\frac{|\log(y)|}{y^2}\Big)+O\Big(\frac{1}{\sqrt{b_1}y^3}\Big),$$
which implies the bound \eqref{fidef} for $i=0$ in the definition of $b_1$-admissible functions.
Now we verify the bound \eqref{fidef} for $i\geq1$.
Indeed, we have the following bound thanks to \eqref{eq:relaAL}:
\begin{equation}
\As\Sigma_{b_1}=\left\{\begin{array}{ll}
O\Big(\frac{\sqrt{b_1}}{y}\Big)+O\Big(\frac{1}{y^2}\Big) \mbox{ for } 1\leq y\leq 2B\\
\frac{C_1}{y^2}+O\Big(\frac{1}{y^3}\Big)+O\Big(\frac{1}{\sqrt{b_1}y^4}\Big) \mbox{ for } y\geq 2B. 
\end{array}
\right.
\end{equation}
And for all $y\geq 2$
$$\As\Lambda T_1=-\frac{C_1}{y^2}+O\big(\frac{|\log y|}{y^3}\Big),$$
it follows that 
\begin{equation}
|\As\theta_1|\lesssim \left\{\begin{array}{ll}
\frac{\sqrt{b_1}}{y}+\frac{1}{y^2} \mbox{ for } 1\leq y\leq 2B\\
\frac{|\log y|}{y^3}+\frac{1}{\sqrt{b_1}y^4} \mbox{ for } y\geq 2B
\end{array}
\right.
\end{equation}
which concludes the proof of the bound \eqref{eq:relaAL} for $i=1$.
In addition, we have that for all $y\geq 1$:
$$\Ls\Lambda T_1=-\frac{4C_1}{y^3}+O\Big(\frac{|\log y|}{y^4}\Big),$$
and
\begin{equation}
\Ls\Sigma_{b_1}=\left\{\begin{array}{ll}
O\Big(\frac{\sqrt{b_1}}{y^2}\Big)+O\Big(\frac{1}{y^3}\Big) \mbox{ for } 1\leq y\leq 2B\\
\frac{-4C_1}{y^3}+O\Big(\frac{1}{y^4}\Big)+O\Big(\frac{1}{\sqrt{b_1}y^5}\Big) \mbox{ for } y\geq 2B, 
\end{array}
\right.
\end{equation}
which imply
\begin{equation}
|\Ls\theta_1|\lesssim \left\{\begin{array}{ll}
\frac{\sqrt{b_1}}{y^2}+\frac{1}{y^3} \mbox{ for } 1\leq y\leq 2B\\
\frac{|\log y|}{y^4}+\frac{1}{\sqrt{b_1}y^5} \mbox{ for } y\geq 2B,
\end{array}
\right.
\end{equation}
and the control of higher order derviatives follows by iterating the same process. It remains to look at the derivatives with respect to $b_1$ of $\theta_1$. For all $l\geq 1$
\begin{equation}
\p_{b_1}^l\theta_1=-\p_{b_1}^l\Sigma_{b_1},
\end{equation}
and notice that on one hand for $y\geq 2B$ we have by the definition of $\Sigma_{b_1}$
$$\p_{b_1}^l\theta_1=0.$$
On the other hand by using Leibniz rule and $\Big|\p_{b_1}^l\chi_{B}\Big|\lesssim \frac{\indic_{\{B\leq y \leq 2B\}}}{b_1^l}$ we obtain for $1\leq y\leq 2B$
\begin{align}
|\p_{b_1}^l\theta_1|\lesssim b_1^{\frac{1}{2}-l}\Big(1+\frac{1}{y}\Big).
\end{align}
To control the derivatives of higher order of the form $\p_{b_1}^lA^i\theta_1$ one can iterate the same process to get the bound \eqref{fbldef}, which concludes the proof of $\theta_1$ being a $b_1$-admissible function of degree $(1,1)$.
\item[\underline{Step 3:$k\rightarrow k+1$}.] We suppose that $\theta_k$ is a $b_1$-admissible function of degree $(k,k)$.
Let's prove that $\theta_{k+1}$ is a $b_1$-admissible function of degree $(k+1,k+1)$.
From \eqref{eq:relLsLam} we obtain
\begin{align}
&\Ls \theta_{k+1} =  \Ls \Lambda T_{k+1} - 2k\Ls T_{k+1}-(-1)^k\Ls^{-k+1}\Sigma_{b_1}\nonumber\\
&\qquad \qquad  = \Lambda T_k - (2k - 2)T_k - \frac{\Lambda Z}{y^2}T_{k+1}+(-1)^{k+1}\Ls^{-k+1}\Sigma_{b_1}\nonumber\\
&\qquad \qquad  =-\theta_k- \frac{\Lambda Z}{y^2}T_{k+1}.
\label{eq:tmpLTk1}
\end{align}
From part $(i)$ of Lemma\ref{lemm:GenLk}, we know that $T_{k+1}$ is admissible of degree $(k+1,k+1)$. From \eqref{def:ZbyV} and \eqref{eq:asympV}, one can check that $\frac{\Lambda Z}{y^2}T_{k+1}$ admits the asymptotic:
$$\frac{\Lambda Z}{y^2}T_{k+1} = \Oc(y^{2k + 1})  \quad \text{as} \quad y \to 0,$$
and 
$$\partial_y^j\left(\frac{\Lambda Z}{y^2}T_{k+1}\right) = \Oc(y^{2(k+1) - j - 5}) \quad \text{as}\quad y \to +\infty.$$
Together with the induction hypothesis and the fact that $\frac{\Lambda Z}{y^2}T_{k+1}$ is independent of $b_1$, we deduce that the right hand side of \eqref{eq:tmpLTk1} is $b_1$-admissible of degree $(k, k)$. The conclusion then follows by using part $(iii)$ of Lemma \ref{lemm:actionLL}. This ends the proof of Lemma \ref{lemm:GenLk}.
\end{itemize}
\end{proof}

We end this subsection by introducing a simple notion of homogeneous admissible function.
\begin{definition}[Homogeneous admissible function] Let $L \gg 1$ be an integer and $m = (m_1, \cdots, m_L) \in \mathbb{N}^L$, we say that a function $f(b,y)$ with $b = (b_1, \cdots, b_L)$ is homogeneous of degree $(p_1, p_2, p_3) \in \mathbb{N} \times \mathbb{Z}\times \mathbb{N}$ if it is a finite linear combination of monomials
$$\tilde f (b_1,y)\prod_{k = 1}^Lb_k^{m_k},$$
with $\tilde{f}(b_1,y)$ $b_1$-admissible of degree $(p_1, p_2)$ in the sense of Definition \ref{def:b1Admitfunc} and
$$(m_1, \cdots, m_L) \in \mathbb{N}^L, \quad\sum_{k = 1}^L km_k = p_3.$$
We set 
$$\text{deg}(f):= (p_1, p_2, p_3).$$  
\end{definition}

\subsection{Slowly modulated blow-up profile.}
In this subsection, we use the explicit structure of the linearized operator $\Ls$ to construct an approximate blow-up profile. In particular, we claim the following:
 \begin{proposition}[Construction of the approximate profile] \label{prop:1}  Let $L \gg 1$ be an integer. Let $M > 0$ be a large enough universal constant, then there exist a small enough universal constant $b^*(M,L) > 0$ such that the following holds true. Let a $\Cc^1$ map
 $$b = (b_1, \cdots, b_L):[s_0,s_1] \mapsto (-b^*, b^*)^L,$$
with a priori bounds in $[s_0,s_1]$:
\begin{equation}\label{eq:relb1bk}
0 < b_1 < b^*, \quad |b_k| \lesssim b_1^{k+\frac{1}{2}}, \quad 2 \leq k \leq L, 
\end{equation}
Then there exist homogeneous profiles 
$$S_1 = 0, \quad S_k = S_k(b,y), \quad 2 \leq k \leq L + 2,$$
such that
\begin{equation}\label{eq:Qbform}
Q_{b(s)}(y) = Q(y) + \sum_{k = 1}^L b_k(s)T_k(y) + \sum_{k = 2}^{L+2}S_k(b,y) \equiv Q(y) + \Theta_{b(s)}(y),
\end{equation}
generates an approximate solution to the remormalized flow \eqref{eq:wys}:
\begin{equation}\label{def:Psib}
\partial_s Q_{b} - \partial_{yy}Q_b - \frac{6}{y}\partial_yQ_b + b_1 \Lambda Q_b + \frac{3}{y^2}\sin(2Q_b) = \Psi_b + \textup{Mod}(t),
\end{equation}
with the following property:\\
$(i)$ (Modulation equation)
\begin{equation}\label{eq:Modt}
\textup{Mod}(t) = \sum_{k = 1}^L\Big[(b_k)_s + (2k - 2+C_{b_1})b_1b_k - b_{k + 1}\Big] \left[T_k + \sum_{j = k + 1}^{L+2}\frac{\partial S_j}{\partial b_k}\right], 
\end{equation}
where we use the convention $b_{j} = 0$ for $j \geq L+1$.\\
$(ii)$ (Estimate on the profiles) The profiles $(S_k)_{2 \leq k \leq L+2}$ are homogeneous with 
\begin{align*}
&\text{deg}(S_k) = (k,k,k) \quad \text{for} \quad 2 \leq k \leq L+2,\\
&\frac{\partial S_k}{\partial b_m} = 0 \quad \text{for}\quad 2 \leq k \leq m \leq L.
\end{align*}
$(iii)$ (Estimate on the error $\Psi_b$) For all $0 \leq m \leq L$, there holds:\\
- (global weight bound)
\begin{equation}\label{eq:estGlobalPsib}
\int_{\{y \leq B_1\}}|\Ls^{m + 1} \Psi_b|^2 + \int_{\{y \leq B_1\}} \frac{|\Psi_b|^2}{1 + y^{4(m + 1 )}} \lesssim b_1^{2(m+1)+1-\frac{\eta}{2}},
\end{equation}
\begin{equation}\label{eq:estlocalPsibL}
\int_{\{y \leq B_1\}}|\Ls^{L + 1} \Psi_b|^2 + \int_{\{y \leq B_1\}} \frac{|\Psi_b|^2}{1 + y^{4(L + 1 )}} \lesssim b_1^{2(L+1)+\frac{5}{2}-\frac{7\eta}{2}}
\end{equation}
where $B_1$, is defined in \eqref{def:B0B1}.
\begin{equation}\label{eq:estlocalPsibB04}
\int_{\{y \leq \frac{B_0}{2}\}}|\Ls^{L + 1} \Psi_b|^2 + \int_{\{y \leq \frac{B_0}{2}\}} \frac{|\Psi_b|^2}{1 + y^{4(L + 1 )}} \lesssim b_1^{2(L+1)+\frac{5}{2}}.
\end{equation}
For $M<B_0$, we get
\begin{equation}\label{eq:verylocalPsiM}
\int_{\{y \leq M\}}|\Ls^{L + 1} \Psi_b|^2 + \int_{\{y \leq M\}} \frac{|\Psi_b|^2}{1 + y^{4(L + 1 )}} \lesssim b_1^{2L+7}.
\end{equation}
\end{proposition}

\begin{proof}
We want to construct the profiles $T_k$ such that $\Psi_b(y)$ defined from \eqref{def:Psib} has the \emph{least possible growth} as $y \to +\infty$. The key to this construction is the fact that the structure of the linearized operator $\Ls$ defined in \eqref{def:Lc} is completely explicit and the introduction of $\Sigma_{b_1}$ which will allow us to cancel the worst growth of $T_k$ at infinity.
 This procedure will lead to the leading-order modulation equation
\begin{equation}
(b_k)_s = -(2k - 2+C_{b_1})b_1b_k + b_{k+1} \quad \text{for}\quad 1 \leq k \leq L,
\end{equation}
which actually cancels the worst growth of $S_k$ as $y \to +\infty$.

\paragraph{$\bullet$ Expansion of $\Psi_b$}. From \eqref{def:Psib} and \eqref{eq:Qy}, we write
\begin{align*}
&\partial_s Q_{b} - \partial_{yy}Q_b - \frac{6}{y}\partial_y Q_b + b_1 \Lambda Q_b + \frac{3}{y^2}\sin(2Q_b)\\
&= \underbrace{b_1\Lambda Q + \partial_s \Theta_b +\Ls(\Theta_b) + b_1 \Lambda \Theta_b}_{:=A_1}+ \underbrace{\frac{3}{y^2}\left[\sin(2Q + 2\Theta_b) - \sin(2Q) - 2\cos(2Q)\Theta_b\right]}_{:=A_2}.
\end{align*}
Using the expression \eqref{eq:Qbform} of $\Theta_b$, the definition \eqref{def:Tk} of $T_k$ (note that $\Ls T_k = - T_{k-1}$ with the convention $T_0 = \Lambda Q$)and the definition of $\theta_k$ \eqref{def:tetak} , we write 
\begin{align*}
A_1&= b_1 \Lambda Q + \sum_{k = 1}^L \Big[(b_k)_s T_k + b_k \Ls T_k + b_1b_k\Lambda T_k \Big] + \sum_{k = 2}^{L+2}\Big[\partial_s S_k + \Ls S_k + b_1 \Lambda S_k\Big]\\
&=  \sum_{k = 1}^L \Big[(b_k)_sT_k - b_{k+1}T_k + b_1b_k \Lambda T_k \Big] + \sum_{k = 2}^{L+2}\Big[\partial_s S_k + \Ls S_k + b_1 \Lambda S_k\Big]\\
&= \sum_{k = 1}^L \Big[(b_k)_s - b_{k+1} + (2k - 2+C_{b_1})b_1b_k\Big]T_k\\
&\qquad + \sum_{k = 1}^L\Big[\Ls S_{k + 1} + \partial_s S_k + b_1b_k \big[\Lambda T_k - (2k - 2+C_{b_1})T_k\big] + b_1 \Lambda S_k\Big]\\
& \qquad\qquad+ \Big[\Ls S_{L+2} + \partial_s S_{L+1} + b_1\Lambda S_{L+1}\Big] + \Big[\partial_sS_{L+2} + b_1\Lambda S_{L+2}\Big],
\end{align*}
where $C_{b_1}$ is defined by \eqref{eq:Cb1}.
We now use the fact that we expect $(b_j)_s\sim -(2j - 2+C_{b_1})b_1b_j + b_{j+1}$ to write
\begin{align*}
\partial_sS_k = \sum_{j = 1}^L(b_j)_s\frac{\partial S_k}{\partial b_j} &= \sum_{j = 1}^L\Big[(b_j)_s + (2j - 2+C_{b_1})b_1b_j - b_{j + 1} \Big]\frac{\partial S_k}{\partial b_j} \\
&\qquad - \sum_{j = 1}^L\Big[(2j - 2+C_{b_1})b_1b_j - b_{j + 1} \Big]\frac{\partial S_k}{\partial b_j}.
\end{align*} 
Hence, 
\begin{align*}
A_1 = \text{Mod}(t) &+ \sum_{k = 1}^{L+1}\left[\Ls S_{k + 1} + E_k\right] + E_{L+2},
\end{align*}
where for $k = 1, \cdots, L$,
\begin{equation}\label{def:Ek1}
E_k= b_1b_k \big[\Lambda T_k - (2k - 2+C_{b_1})T_k\big] + b_1 \Lambda S_k - \sum_{j = 1}^{k-1}\Big[(2j - 2+C_{b_1})b_1b_j - b_{j + 1} \Big]\frac{\partial S_k}{\partial b_j},
\end{equation} 
and for $k = L+1, L+2$,
\begin{equation}\label{def:EkL12}
E_k = b_1 \Lambda S_k - \sum_{j = 1}^{L}\Big[(2j - 2+C_{b_1})b_1b_j - b_{j + 1} \Big]\frac{\partial S_k}{\partial b_j}.
\end{equation}

For the expansion of the nonlinear term $A_2$, let us denote 
$$f(x) = \sin(2x)$$
and use a Taylor expansion to write (see pages 1740 in \cite{RSapde2014} for a similar computation)
\begin{align*}
A_2 = \frac{3}{y^2}\left[\sum_{i = 2}^{L+2}\frac{f^{(i)}(Q)}{i!}\Theta_b^i + R_2\right] = \frac{3}{y^2}\left[\sum_{i = 2}^{L+2}P_i + R_1 + R_2\right],
\end{align*}
where 
\begin{equation}\label{def:Pj}
P_i = \sum_{j = 2}^{L+2}\frac{f^{(j)}(Q)}{j!}\sum_{|J|_1 = j, |J|_2 = i}c_J \prod_{k = 1}^Lb_k^{i_k}T_k^{i_k}\prod_{k=2}^{L+2}S_k^{j_k},
\end{equation}
\begin{equation}\label{def:R_1}
R_1 = \sum_{j = 2}^{L+2}\frac{f^{(j)}(Q)}{j!}\sum_{|J|_1 = j, |J|_2 \geq L+3}c_J \prod_{k = 1}^Lb_k^{i_k}T_k^{i_k}\prod_{k=2}^{L+2}S_k^{j_k},
\end{equation}
\begin{equation}\label{def:R2}
R_2 = \frac{\Theta_b^{L+3}}{(L+2)!}\int_0^1(1 - \tau)^{L+2}f^{(L+3)}(Q + \tau \Theta_b)d\tau,
\end{equation}
with $J = (i_1, \cdots, i_L, j_2, \cdots, j_{L+2}) \in \mathbb{N}^{2L + 1}$ and 
\begin{equation}\label{def:J1J2}
|J|_1 = \sum_{k = 1}^L i_k + \sum_{k = 2}^{L+2}j_k, \quad |J|_2 = \sum_{k=1}^L ki_k + \sum_{k = 2}^{L+2}kj_k.
\end{equation}
We now use the definition \eqref{def:tetak} of $\theta_k$ to rewrite:
$$\Lambda T_k-(2k-2+C_{b_1})T_k=\theta_k+(-1)^{k+1}\Ls^{-k+1}\Sigma_{b_1}-C_{b_1}T_k=\theta_k+(-1)^{k+1}\Ls^{-k+1}(\Sigma_{b_1}-C_{b_1}T_1).$$
Hence, $E_k$ for $1\leq k\leq L$ becomes
\begin{align*}
E_k&=b_1b_k \big[\theta_k+(-1)^{k+1}\Ls^{-k+1}(\Sigma_{b_1}-C_{b_1}T_1)\big] + b_1 \Lambda S_k - \sum_{j = 1}^{k-1}\Big[(2j - 2+C_{b_1})b_1b_j - b_{j + 1} \Big]\frac{\partial S_k}{\partial b_j}\nonumber\\
&=\tilde{E}_k+(-1)^{k+1}b_1b_k\Ls^{-k+1}(\Sigma_{b_1}-C_{b_1}T_1),
\end{align*}
with
\begin{align}\label{def:tildeEk}
\tilde{E}_k=b_1b_k\theta_k + b_1 \Lambda S_k - \sum_{j = 1}^{k-1}\Big[(2j - 2+C_{b_1})b_1b_j - b_{j + 1} \Big]\frac{\partial S_k}{\partial b_j}
\end{align}

which induces the final expression for the error $\Psi_b$: 
\begin{align}\label{eq:expanPsib}
\Psi_b &= \sum_{k=1}^L(-1)^{k+1}b_1b_k\Ls^{-k+1}(\Sigma_{b_1}-C_{b_1}T_1)+\sum_{k = 1}^{L}\left[\Ls S_{k + 1} + \tilde{E}_k + \frac{3}{y^2}P_{k+1}\right]\nonumber\\
 &+\Big[\Ls S_{L+2}+E_{L+1}+\frac{3}{y^2}P_{L+2}\Big] + E_{L+2} + \frac{3}{y^2}(R_1 + R_2).
\end{align}

\paragraph{$\bullet$ Construction of $S_k$.} From the expression of $\Psi_b$ given in \eqref{eq:expanPsib}, we construct iteratively the sequences of profiles $(S_k)_{1 \leq k \leq L+2}$ through the scheme
\begin{equation}\label{def:Sk}
\left\{\begin{array}{lll}
S_1 &= 0, \\
S_k &= - \Ls^{-1}F_k, \quad 2 \leq k \leq L,\\
S_{L+2}&=-\Ls^{-1}F_{L+2}, 
\end{array}\right.
\end{equation}
where 
$$F_k = \tilde{E}_{k-1} + \frac{3}{y^2}P_{k} \quad \text{for} \quad 2\leq k \leq L+1,$$
and
$$F_{L+2} = E_{L+1} + \frac{3}{y^2}P_{L+2}.$$
We claim by induction on $k$ that $F_k$ is homogeneous with 
\begin{eqnarray}\label{eq:degFk}
\text{deg}(F_k) = (k-1, k-1, k) \quad \text{for} \quad 2 \leq k \leq L+2,\\
\end{eqnarray}
and 
\begin{eqnarray}\label{eq:estparFk}
\frac{\partial F_k}{\partial b_m} = 0 \quad \text{for}\quad 2 \leq k \leq m \leq L+2.
\end{eqnarray}
From item $(iii)$ of Lemma \ref{lemm:actionLL} and \eqref{eq:degFk}, we deduce that $S_k$ is homogeneous of degree
$$\text{deg}(S_k) = (k,k,k) \quad \text{for}\quad 2 \leq k \leq L+2,$$
and from \eqref{eq:estparFk}, we get
$$\frac{\partial S_k}{\partial b_m} = 0 \quad \text{for} \quad 2 \leq k \leq m \leq L+2,$$
which is the conclusion of item $(ii)$.\\

Let us now give the proof of \eqref{eq:degFk} and \eqref{eq:estparFk}. We proceed by induction.\\
\noindent - Case $k = 2$: We compute explicitly from \eqref{def:tildeEk} and \eqref{def:Pj},
$$F_2 = \tilde{E}_1 + \frac{3}{y^2}P_2 = b_1^2\left[\theta_1 + \frac{6f''(Q)}{2y^2}T_1^2\right],$$
which directly follows \eqref{eq:estparFk}. From Lemma \ref{lemm:GenLk} and \ref{lemm:slowgrowtails} , we know that $T_1$ and $\theta_1$ are admissible of degree $(1,1)$ and $b_1$-admissible of degree $(1,1)$ respectively. Using \eqref{eq:asymQ}, one can check the bound
\begin{equation}\label{eq:estfjm}
\forall m,j \in \mathbb{N}^2, \quad\left|\partial_y^m \left(\frac{f^{(j)}(Q)}{y^2}\right)\right| \lesssim y^{-2 - m} \quad \text{as}\quad y \to +\infty.
\end{equation}
Since $T_1$ is admissible of degree $(1,1)$, we have that
$$\forall m \in \mathbb{N}, \quad |\partial_y^m(T_1^2)| \lesssim y^{- m}  \quad \text{as}\quad y \to +\infty.$$
By the Leibniz rule we get that
$$\forall m,j \in \mathbb{N}^2, \quad\left|\partial_y^m \left(\frac{f^{(j)}(Q)}{y^2} T_1^2\right)\right|  \lesssim y^{-2
 - m}.$$
We also have the expansion near the origin, 
$$\frac{f^{(j)}(Q)}{y^2}T_1^2 = \sum_{i = 2}^k c_iy^{2i + 1} + \Oc(y^{2k + 3}), \quad k \geq 1.$$
Hence, $\frac{f''(Q)}{y^2}T_1^2$ is admissible of degree $(2,0)$, which concludes the proof of \eqref{eq:degFk} for $k = 2$.\\

\noindent - Case $k \to k+1$: Estimate \eqref{eq:estparFk} holds by direct inspection. We suppose that $S_k$ is homogeneous of degree $(k, k, k)$ and prove that $S_{k + 1}$ is homogeneous of degree $(k+1, k+1, k+1)$. In particular, the claim immediately follows from part $(iii)$ of Lemma \ref{lemm:actionLL} once we show that $F_{k+1}$ is homogeneous with
\begin{equation}\label{eq:Fk1EkPk1}
\text{deg}(F_{k+1}) =  (k, k, k+1).
\end{equation}
From part $(ii)$ of Lemma \ref{lemm:GenLk} and the a priori assumption \eqref{eq:relb1bk}, we see that $b_1b_k\theta_k$ is homogeneous of degree $(k, k, k+1)$. From part $(i)$ of Lemma \ref{lemm:actionLL} and the induction hypothesis, $b_1 \Lambda S_k$ is also homogeneous of degree $(k, k, k+1)$. By definition, $b_j\frac{\partial S_k}{\partial b_j}$ is homogeneous and has the same degree as $S_k$. For $j\geq2$, by definition and induction we have that 
$$\Big[(2j - 2+C_{b_1})b_1b_j - b_{j + 1} \Big]\frac{\partial S_k}{\partial b_j}$$
is homogeneous of degree $(k,k,k+1)$. We just prove the initialization of the induction the rest is left to the reader. For $j=1$ we have that:
$$\left(C_{b_1}b_1- \frac{b_{2}}{b_{1}}\right)\left(b_1\frac{\partial S_k}{\partial b_1}\right)$$
is homogeneous of degree $(k, k, k+1)$. From definitions \eqref{def:Ek1} and \eqref{def:EkL12}, we derive  
\begin{equation}
\left\{\begin{array}{lll}
\text{deg}(\tilde{E}_k) &= (k, k, k+1), \quad 1\leq k \leq L, \\
\text{deg}(E_{L+1}) &= (L+1, L+1, L+2). 
\end{array}\right.
\end{equation}
It remains to control the term $\frac{P_{k+1}}{y^2}$. From the definition \eqref{def:Pj}, we see that $\frac{P_{k+1}}{y^2}$ is a linear combination of monomials of the form 
$$M_J(y) = \frac{f^{(j)}(Q)}{y^2}\prod_{m = 1}^L b_m^{i_m}T_m^{i_m}\prod_{m = 2}^{L+2}S_m^{j_m},$$
with 
$$J= (i_1, \cdots, i_L, j_2, \cdots, j_{L+2}), \quad |J|_1 = j, \; |J|_2 = k+1, \; 2 \leq j \leq k+1.$$
Recall from part $(i)$ of Lemma \ref{lemm:GenLk} the bound
$$\forall n \in \mathbb{N}, \quad |\partial_y^n T_m| \lesssim y^{2m -2- n} \quad \text{as} \quad y \to +\infty,$$
and from the induction hypothesis and the a priori bound \eqref{eq:relb1bk},
$$\forall n \in \mathbb{N}, \quad |\partial_y^n S_m| \lesssim b_1^{m+\frac{1}{2}}y^{2m - 2 - n} \quad \text{as} \quad y \to +\infty. $$
Together with the bound \eqref{eq:estfjm}, we obtain the following bound at infinity,
$$| M_J|\lesssim b_1^{|J|_2+\frac{|J|_1}{2}} y^{2|J|_2-2|J|_1 - 2} \lesssim b_1^{k+1}y^{2(k+1) - 2}.$$
The control of $\partial_y^n M_J$ follows by the Leibniz rule and the above estimates. One can also check that $M_J$ is of order $y^{2(k+1) + j - 1}$ near the origin. This concludes the proof of \eqref{eq:Fk1EkPk1} as well as part $(ii)$ of Proposition \ref{prop:1}.\\

\paragraph{$\bullet$ Estimate on $\Psi_b$.} From \eqref{eq:expanPsib} and \eqref{def:Sk}, the expression of $\Psi_b$ is now reduced to 
\begin{align}
\Psi_b =\underbrace{\sum_{k=1}^L(-1)^{k+1}b_1b_k\Ls^{-k+1}\overline{\Sigma}_{b_1}}_{\Psi_b^0}+ E_{L+2} + \frac{6}{y^2}(R_1 + R_2),
\end{align}
where $E_{L+2}$, $R_1$, $R_2$ are given by \eqref{def:EkL12}, \eqref{def:R_1} and \eqref{def:R2}. 
\begin{equation*}
\overline{\Sigma}_{b_1}=\Sigma_{b_1}-C_{b_1}T_1.
\end{equation*}
Notice first that on one hand 
\begin{equation}\label{supsigbar}
\text{Supp}~\overline{\Sigma}_{b_1} \subset\Big\{y\geq B_0\Big\}.
\end{equation}
One the other hand we have from \eqref{sigmab_est} that
\begin{align*}
|\Ls^{-k+m+2}\overline{\Sigma}_{b_1}|\lesssim \sqrt{b_1}y^{2k-2m-4}.
\end{align*}
Thus from \eqref{eq:relb1bk} we get for all $m\geq1$
\begin{align}
\int_{\{y \leq B_1\}}|\Ls^{m + 1}\Psi_b^0|^2 y^6dy&\lesssim b_1^4\int_{\{y\leq 2B_1\}}|\Ls^{m+1}\overline{\Sigma}_{b_1}|^2y^6dy\nonumber\\
&+\sum_{k=2}^Lb_1^{2k+3}\int_{\{y\leq 2B_1\}}|\Ls^{-k+2+m}\overline{\Sigma}_{b_1}|^2y^6dy\nonumber\\
&\lesssim b_1^{2(m+1)+1+(2m+\frac{1}{2})\eta},
\end{align}
Remark that when $m=L$ we have
\begin{align}
\int_{\{y \leq B_1\}}|\Ls^{L + 1}\Psi_b^0|^2 y^6dy &\lesssim b_1^4\int_{\{y\leq 2B_1\}}|\Ls^{L+1}\overline{\Sigma}_{b_1}|^2y^6dy\nonumber\\
&+\sum_{k=2}^Lb_1^{2k+3}\int_{\{y\leq 2B_1\}}|\Ls^{-k+2+L}\overline{\Sigma}_{b_1}|^2y^6dy\nonumber\\
&\lesssim b_1^{2(L+1)+\frac{5}{2}+\frac{\eta}{2}}.
\end{align}

We start by estimating $E_{L+2}$ term defined by \eqref{def:EkL12}. Since $S_{L+2}$ is homogeneous of degree $(L+2, L+2, L+2)$ and thus so are $\Lambda S_{L+2}$ and $b_1 \frac{\partial S_{L+2}}{\partial b_1}$. This implies that $E_{L+2}$ is homogeneous of degree $(L+2, L+2, L+3)$. 
Using part $(ii)$ of Lemma \ref{lemm:actionLL}, we estimate for all $0 \leq m \leq L$
\begin{align*}
\int_{y \leq 2B_1}|\Ls^{m + 1}E_{L+2}|^2 &\lesssim b_1^{2L+7}\int_{1\leq y \leq 2B_1}|y^{2(L+2) - 3 - 2 m -2}|^2 y^{6}dy\\
&\qquad \quad \lesssim b_1^{2(m+1) + 1+(2m-\frac{7}{2})\eta},
\end{align*}
where $\eta = \frac{3}{4L}$.
And when $m=L$ we get similarly,
\begin{align*}
\int_{y \leq 2B_1}|\Ls^{L + 1}E_{L+2}|^2 &\lesssim b_1^{2(L+1) + \frac{5}{2}-\frac{7\eta}{2}}.
\end{align*}
We now turn to the control of the term $\frac{R_1}{y^2}$, which is a linear combination of terms of the form (see \eqref{def:R_1})
$$\tilde{M}_J = \frac{f^{(j)}(Q)}{y^2}\prod_{n = 1}^Lb_n^{i_n}T_n^{i_n}\prod_{n = 2}^{L+2}S_n^{j_n},$$
with 
$$J = (i_1, \cdots, i_L, j_2, \cdots, j_{L+2}), \; |J|_1 = j, \; |J|_2 \geq L+3, \; 2 \leq j \leq L+2.$$
Using the admissibility of $T_n$ and the homogeneity of $S_n$, we get the bound
$$|\tilde{M}_J| \lesssim b_1^{|J|_2+\frac{|J|_1}{2}}y^{2|J|_2 - 2|J|_1 - 2}  \quad \text{as}\quad y \to +\infty,$$
and similarly for higher derivatives by the Leibniz rule. Thus, we obtain the round estimate for all $1 \leq m \leq L$,
\begin{align*}
\int_{y \leq 2B_1}\left|\Ls^{m + 1}\left(\frac{R_1}{y^2}\right)\right|^2 &\lesssim b_1^{2|J|_2+|J|_1}\int_{y \leq 2B_1}|y^{2|J|_2 -2|J|_1-6 -2m}|^2 y^6dy\\
&\quad\lesssim b_1^{2(m+1) + 1-\frac{\eta}{2}}.
\end{align*}
The term $\frac{R_2}{y^2}$ is estimated exactly as for the term $\frac{R_1}{y^2}$ using the definition \eqref{def:R2}. Similarly, the control of $\int_{y \leq 2B_1} \frac{|\Psi_b|^2}{(1 + y^{4(m + 1)})}$ is obtained along the exact same lines as above. This concludes the proof of \eqref{eq:estGlobalPsib}. The local estimate \eqref{eq:estlocalPsibB04} directly follows from the homogeneity of $S_k$ and the admissibility of $T_k$ and the fact that for $y\leq \frac{B_0}{4}$ $\psi_b^0=0$. This concludes the proof of Proposition \ref{prop:1}.
\end{proof}

\bigskip

We now proceed to a simple localization of the profile $Q_b$ to avoid the growth of tails in the region $y \geq 2B_1 \gg B_0$. More precisely, we claim the following:

\begin{proposition}[Estimates on the localized profile] \label{prop:localProfile} Under the assumptions of Proposition \ref{prop:1}, we assume in addition the a priori bound 
\begin{equation}\label{eq:apriorib1}
|(b_1)_s| \lesssim b_1^{\frac{5}{2}}.
\end{equation}
Consider the localized profile
\begin{equation}\label{def:Qbtil}
\tilde{Q}_{b(s)}(y) = Q(y) + \sum_{k = 1}^Lb_k\tilde{T}_k + \sum_{k = 2}^{L+2}\tilde{S}_k \quad \text{with}\quad \tilde{T}_k = \chi_{B_1}T_k, \; \tilde{S}_k = \chi_{B_1}S_k,
\end{equation}
where $B_1$ and $\chi_{B_1}$ are defined as in \eqref{def:B0B1} and \eqref{def:chiM}. Then
\begin{equation}\label{def:Psibtilde}
\partial_s \tilde Q_{b} - \partial_{yy}\tilde Q_b - \frac{6}{y}\partial_y \tilde Q_b + b_1 \Lambda \tilde Q_b + \frac{3}{y^2}\sin(2\tilde Q_b) = \tilde \Psi_b + \chi_{B_1}\, \textup{Mod}(t),
\end{equation}
where $\tilde \Psi_{b}$ satisfies the bounds:\\

\noindent $(i)\;$ (Large Sobolev bound) For all $1 \leq m \leq L-1$,
\begin{align}
&\int |\Ls^{m + 1} \tilde{\Psi}_b|^2 + \int \frac{|\As\Ls^{m} \tilde{\Psi}_b|^2}{1 + y^2} \nonumber\\
& \qquad + \int \frac{|\Ls^m \tilde{\Psi}_b|^2}{1 + y^4} + \int \frac{|\tilde \Psi_b|^2}{1 + y^{4(m + 1)}} \lesssim b_1^{2(m+1)+ (\frac{1}{2}+2m)\eta},\label{eq:estPsibLarge1}
\end{align}
and 
\begin{align}
&\int |\Ls^{L + 1} \tilde{\Psi}_b|^2 + \int \frac{|\As\Ls^{L} \tilde{\Psi}_b|^2}{1 + y^2} \nonumber\\
& \qquad + \int \frac{|\Ls^{L} \tilde{\Psi}_b|^2}{1 + y^4} + \int \frac{|\tilde \Psi_b|^2}{1 + y^{4(L + 1)}} \lesssim b_1^{2(L+1)+ \frac{3}{2}+\frac{\eta}{2}},\label{eq:estPsibLarge2}
\end{align}
\noindent $(ii)\;$(Improved local bound) for all $0<M<\frac{B_1}{2}$ there exists $C$ such that for all $1\leq m\leq L$,
\begin{align}
&\int_{\{y\leq 2M\}} |\Ls^{m+1
} \tilde{\Psi}_b|^2 \lesssim M^C b_1^{2L+7},
\label{eq:localProfile}
\end{align}
\noindent $(iii)\;$ (Very Local Bound) for all $1\leq m\leq L$,
\begin{align}
&\int_{\{y\leq\frac{B_0}{2}\}} |\Ls^{m + 1} \tilde{\Psi}_b|^2 + \int_{\{y\leq\frac{B_0}{2}\}} \frac{|\As\Ls^{m} \tilde{\Psi}_b|^2}{1 + y^2} \nonumber\\
& \qquad + \int_{\{y\leq\frac{B_0}{2}\}} \frac{|\Ls^m \tilde{\Psi}_b|^2}{1 + y^4} + \int_{\{y\leq\frac{B_0}{2}\}} \frac{|\tilde \Psi_b|^2}{1 + y^{4(m + 1)}} \lesssim b_1^{2(m+1)+\frac{5}{2}},\label{eq:estPsiblocalB0}
\end{align}
%
\end{proposition}

\begin{proof} By a direct computation, we have
\begin{align*}
&\partial_s \tilde Q_{b} - \partial_{yy}\tilde Q_b - \frac{6}{y}\partial_y \tilde Q_b + b_1 \Lambda \tilde Q_b + \frac{3}{y^2}\sin(2\tilde Q_b)\\
& =\chi_{B_1}\left[\partial_s Q_{b} - \partial_{yy}Q_b - \frac{6}{y}\partial_y Q_b + b_1 \Lambda Q_b + \frac{3}{y^2}\sin(2 Q_b) \right] \\
& \quad + \Theta_b \left[ \partial_s \chi_{B_1} - \left(\partial_{yy}\chi_{B_1} + \frac{6}{y}\partial_y \chi_{B_1}\right)  + b_1 \Lambda \chi_{B_1}\right]  - 2\partial_y \chi_{B_1}\partial_y\Theta_b + b_1(1 - \chi_{B_1})\Lambda Q \\
& \qquad + \frac{3}{y^2}\left[\sin(2\tilde Q_b) - \sin(2Q) - \chi_{B_1}(\sin(2Q_b) - \sin(2Q)) \right].
\end{align*}
According to \eqref{def:Psib} and \eqref{def:Psibtilde}, we write
$$\tilde{\Psi}_b = \chi_{B_1}\Psi_b + \hat\Psi_b,$$
where 
\begin{align*}
\hat \Psi_b &=\underbrace{ b_1(1 - \chi_{B_1})\Lambda Q}_{\hat \Psi_b^{(1)}}\\
& \quad + \underbrace{\frac{3}{y^2}\left[\sin(2\tilde Q_b) - \sin(2Q) - \chi_{B_1}(\sin(2Q_b) - \sin(2Q)) \right]}_{\hat \Psi_b^{(2)}}\\
& \qquad +\underbrace{ \Theta_b \left[ \partial_s \chi_{B_1} - \left(\partial_{yy}\chi_{B_1} + \frac{6}{y}\partial_y \chi_{B_1}\right)  + b_1 \Lambda \chi_{B_1}\right]  - 2\partial_y \chi_{B_1}\partial_y\Theta_b}_{\hat \Psi_b^{(3)}}.
\end{align*}

The contribution of the term $\chi_{b_1}\Psi_b$ to the bounds \eqref{eq:estPsibLarge1}, \eqref{eq:estPsibLarge2} and \eqref{eq:estPsiblocalB0} follows exactly the same as in the proof of \eqref{eq:estGlobalPsib}. We therefore left to estimate the term $\hat{\Psi}_b$. Since all the terms in the expression of $\hat{\Psi}_b$ are localized in $B_1 \leq y \leq 2B_1$ except for the first one whose support is a subset of $\{y \geq B_1\}$. Hence, the estimates \eqref{eq:localProfile} and \eqref{eq:estPsiblocalB0} directly follow from \eqref{eq:estlocalPsibB04} and \eqref{eq:verylocalPsiM}. 

Let us now find the contribution of $\hat \Psi_b$ to the bounds \eqref{eq:estPsibLarge1} and \eqref{eq:estPsibLarge2}. We estimate 
\begin{equation*}
\forall n \in \mathbb{N}, \quad \left|\frac{d^n}{dy^n}(1 - \chi_{B_1})\Lambda Q\right| \lesssim \frac{1}{y^{2 + n}}\mathbf{1}_{y \geq B_1},
\end{equation*}
hence, using the definition \eqref{def:B0B1} of $B_1$, we estimate for all $0 \leq m \leq L$,
\begin{equation*}
\int|\Ls^{m + 1} \hat \Psi_b^{(1)}|^2 \lesssim b_1^2 \int_{y \geq B_1}\frac{y^{6}}{y^{4(m + 1) + 4}} \lesssim b_1^{2(m +1)+\frac{1}{2}+(2m+\frac{1}{2})\eta}.
\end{equation*}
If $m=L$ and $\eta=\frac{3}{4L}$ we have
\begin{equation*}
\int|\Ls^{L + 1} \hat \Psi_b^{(1)}|^2 \lesssim b_1^2 \int_{y \geq B_1}\frac{y^{6}}{y^{4(L + 1) + 4}} \lesssim b_1^{2(L +1)+ \frac{1}{2}+(2L+\frac{1}{2})\eta}=b_1^{2(L+1)+2+\frac{\eta}{2}}.
\end{equation*}
For the nonlinear term $\hat \Psi_b^{(2)}$, we use that
\begin{equation}\label{eq:hatpsi2}
|\hat \Psi_b^{(2)}|\lesssim\frac{\chi_{B_1}|\Theta_b|}{y^2}.
\end{equation}
In addition, we note from the admissibility of $T_k$ and the homogeneity of $S_k$ that for $y \leq 2B_1$,
\begin{equation}\label{eq:estThetab}
\forall n \in \mathbb{N}, \quad \left|\partial_y^n \Theta_b \right| \lesssim b_1y^{-n}\mathbf{1}_{y \geq B_1}+\sum_{k = 2}^L b_1^{k+\frac{1}{2}}y^{2k - 2- n}\mathbf{1}_{y \geq B_1}.
\end{equation}
Using \eqref{eq:estThetab} and \eqref{eq:hatpsi2}, we obtain the round bound
\begin{align*}
\left|\partial_y^n \hat \Psi_b^{(2)}\right|\quad \lesssim b_1y^{-2-n}\chi_{B_1}+\sum_{k = 2}^L b_1^{k+\frac{1}{2}}y^{2(k-1) - 2 - n} \chi_{B_1}.
\end{align*}
We then estimate for $1 \leq m \leq L$, 
\begin{align*}
\int|\Ls^{m + 1} \hat \Psi_b^{(2)}|^2 &\lesssim b_1^2\int_{y\leq2B_1}y^{6-4m-8}dy+ \sum_{k=2}^L b_1^{2k+1} \int_{y \leq 2B_1}y^{4k-6-4m}dy\\
& \quad \lesssim b_1^{2(m+1) + (2m+\frac{1}{2})\eta}.
\end{align*}
If $m=L$ and $\eta=\frac{3}{4L}$, we get
\begin{align*}
\int|\Ls^{L + 1} \hat \Psi_b^{(2)}|^2 \lesssim b_1^{2(L+1)+\frac{3}{2} + \frac{5}{2}\eta}
\end{align*}
To control $\hat \Psi_b^{(3)}$, we first note that its support is included in $\{B_1\leq y\leq2B_1\}$. From the definition \eqref{def:chiM} and the assumption \eqref{eq:apriorib1} we obtain that 
$$|\partial_s \chi_{B_1}| \lesssim \frac{(b_1)_s}{b_1} \frac{y}{B_1} \mathbf{1}_{B_1\leq y \leq 2B_1}\lesssim b_1^{\frac{3}{2}} \mathbf{1}_{B_1\leq y \leq 2B_1}.$$
Using \eqref{eq:estThetab}, we estimate for $1 \leq m \leq L$, 
\begin{align*}
\int |\Ls^{m + 1} \hat \Psi_b^{(3)}|^2 &\lesssim b_1^2\int_{y\leq2B_1}y^{6-4m-8}dy+ \sum_{k=2}^L b_1^2 b_1^{2k+1}\int_{B_1 \leq y \leq 2B_1}y^{6+4k-4m-8}dy\\
&\quad  \lesssim b_1^{2(m+1) + (2m+\frac{1}{2})\eta}.
\end{align*}
Gathering all the bounds yields
\begin{align*}
\int |\Ls^{m + 1}\hat \Psi_b|^2 &\lesssim b_1^{2(m+1) + (2m+\frac{1}{2})\eta},
\end{align*}
and for $m=L$,
\begin{align*}
\int |\Ls^{L + 1}\hat \Psi_b|^2 &\lesssim b_1^{2(L+1)+\frac{3}{2} +\frac{1}{2}\eta}.
\end{align*}
The control of  $\int \frac{|\As \Ls^{m}\tilde \Psi_b|^2}{1 + y^2}$, $\int \frac{|\Ls^{m}\tilde \Psi_b|^2}{1 + y^4}$ and $\int \frac{|\hat \Psi_b|^2}{1 + y^{4( m + 1)}}$ are obtained along the exact same lines as above. 
The improved local bound is coming from the fact that $Supp (\hat{\Psi}_b)\bigcap\{y\leq2M\}=\emptyset$.
Indeed, for all $y\leq 2M$ we have that $\tilde{\Psi}_b=\Psi_b$.
Hence,
$$\int_{\{y\leq2M\}}|\Ls^{m+1} \Psi_b|^2\lesssim M^Cb_1^{2L+7}.$$
This concludes the proof  of \eqref{eq:estPsibLarge1}, \eqref{eq:estPsibLarge2}, \eqref{eq:localProfile} and \eqref{eq:estPsiblocalB0} as well as Proposition \eqref{prop:localProfile}.


\end{proof}
%

\section{The trapped regime.} \label{sec:3}
This section is devoted to introduce the open set of our initial data.  We proceed in 2 subsections:\\
- In the first subsection, we give an equivalent formulation of the linearization of the problem in the setting \eqref{eq:dec_i}.\\
- In the second subsection, we prepare the set of initial data such that the solutions trapped in this set satisfies the conclusion of Theorem \ref{Theo:1}.
\subsection{Linearization of the problem.} 
Let $L \gg 1$ be an integer and $s_0 \gg 1$, we introduce the renormalized variables:
\begin{equation}
y = \frac{r}{\lambda(t)}, \quad s = s_0 + \int_0^t \frac{d\tau}{\lambda^2(\tau)},
\end{equation} 
and the decomposition
\begin{equation}\label{def:qys}
u(r,t) = w(y,s) = \big(\tilde{Q}_{b(s)} + q\big)(y,s) = \big(\tilde{Q}_{b(t)} + q\big)\left(\frac{r}{\lambda(t)}, t\right),
\end{equation}
where $\tilde Q_b$ is constructed in Proposition \ref{prop:localProfile} and the modulation parameters 
$$\lambda(t) > 0, \quad b(t) = (b_1(t), \cdots, b_L(t))$$
are determined from the $L+1$ orthogonality conditions:
\begin{equation}\label{eq:orthqPhiM}
\left<q, \Ls^k \Phi_M \right> = 0, \quad 0 \leq k \leq L,
\end{equation} 
where $\Phi_M$ is a fixed direction depending on some large constant $M$ defined by
\begin{equation}\label{def:PhiM}
\Phi_M = \sum_{k = 0}^L c_{k,M}\Ls^k(\chi_M \Lambda Q),
\end{equation}
with 
\begin{equation}\label{def:ckM}
c_{0,M} = 1, \quad c_{k, M} = (-1)^{k+1}\frac{\sum_{j = 0}^{k - 1}c_{j,M}\left<\chi_M \Ls^j(\chi_M \Lambda Q), T_k\right> }{\left<\chi_M \Lambda Q, \Lambda Q\right>}, \quad 1 \leq k \leq L.
\end{equation}
Here, $\Phi_M$ is build to ensure the nondegeneracy 
\begin{equation}\label{eq:PhiMLamQ}
\left<\Phi_M, \Lambda Q\right> = \left<\chi_M \Lambda Q, \Lambda Q\right> \gtrsim M^{3},
\end{equation}
and the cancellation 
\begin{equation}\label{id:TkPhiM0}
\left<\Phi_M, T_k\right> = \sum_{j = 0}^{k - 1}c_{j,M}\left<\Ls^j(\chi_M\Lambda Q), T_k\right> + c_{k,M} (-1)^k \left<\chi_M \Lambda Q, \Lambda Q\right> = 0,
\end{equation}
In particular, we have 
\begin{equation}\label{id:TkPhiMi}
\left<\Ls^iT_k, \Phi_M \right> = (-1)^k\left<\chi_M \Lambda Q, \Lambda Q\right>\delta_{i,k}, \quad 0 \leq i,k\leq L.
\end{equation}

\medskip

From \eqref{eq:wys}, we see that $q$ satisfies the equation:
\begin{equation}\label{eq:qys}
\partial_s q - \frac{\lambda_s}{\lambda} \Lambda q + \Ls q = - \tilde{\Psi}_b  - \widehat{\textup{Mod}} + \Hc(q) - \Nc(q) \equiv \Fc,
\end{equation}
where 
\begin{equation}\label{def:Modhat}
\widehat{\textup{Mod}} = - \left(\frac{\lambda_s}{\lambda} + b_1\right)\Lambda \tilde{Q}_b - \chi_{B_1}\textup{Mod},
\end{equation}
and $\Hc$ is the linear part given by 
\begin{equation}\label{def:Lq}
\Hc(q) = \frac{6}{y^2}\big[\cos(2Q) - \cos(2\tilde Q_b)\big]q,
\end{equation}
and $\Nc$ is the nonlinear term
\begin{equation}\label{def:Nq}
\Nc(q) = \frac{3}{y^2}\big[\sin(2\tilde Q_b + 2q) - \sin(2\tilde Q_b) - 2q \cos(2\tilde{Q}_b) \big].
\end{equation}
We also need to write the equation \eqref{eq:qys} in the original variables. To do so, let define the rescaled linearized operator:
\begin{equation}\label{def:Llambda}
\Ls_\lambda = - \partial_{rr} - \frac{6}{r}\partial_r + \frac{Z_\lambda}{r^2},
\end{equation}
and the renormalized function
$$v(r,t) = q(y,s), \quad \partial_t v = \frac{1}{\lambda^2(t)}\left(\partial_s q - \frac{\lambda_s}{\lambda} \Lambda q\right)_\lambda$$
then from \eqref{eq:qys}, $v$ satisfies the equation
\begin{equation}\label{eq:vrt}
\partial_t v + \Ls_\lambda v = \frac{1}{\lambda^2}\Fc_\lambda, \quad \Fc_\lambda(r,t) = \Fc(y,s).
\end{equation}
Note that 
$$\Ls_\lambda = \frac{1}{\lambda^2}\Ls.$$

\subsection{Preparation of the initial data.}
We describe in this subsection the set of initial data $u_0$ of the problem \eqref{Pb} as well as the initial data for $(b, \lambda)$ leading to the blowup scenario of Theorem \ref{Theo:1}. Assume that $u_0$ takes the form 
\begin{equation}\label{eq:asumUosmall}
u_0 = Q + q_0 \quad \text{with} \quad \|q_0\|_{\Hs_{2L+2}}\ll 1,
\end{equation}
where 
$$\|f\|_{\Hs_{2L+2}}:= \int |\Ls^{L+1} f|^2 + \int \frac{|\As (\Ls^{L} f)|^2}{y^2} + \sum_{m = 0}^{L} \int \frac{|\Ls^m f|^2}{y^4(1 + y^{4(L - m)})} + \sum_{m = 0}^{\Bbbk-2}\frac{|\As(\Ls^{m}f)|^2}{y^6(1 + y^{4(L - m - 1)})}.$$
By a standard argument (see for example \cite{MRRim13}, \cite{RScpam13, RSapde2014}), the smallness assumption \eqref{eq:asumUosmall} is propagated on a small time interval $[0,t_1)$. The existence of the decomposition
\begin{equation}\label{def:qrt}
u(r,t) = \left(\tilde Q_{b(t)} + q\right)\left(\frac{r}{\lambda(t)}, t\right), \quad \lambda(t) > 0, \; b = (b_1, \cdots, b_L),
\end{equation}
is a standard consequence of the implicit function theorem and the explicit relations
$$\left.\frac{\partial}{\partial \lambda}(\tilde{Q}_{b(t)})_{\lambda}, \frac{\partial}{\partial b_1}(\tilde{Q}_{b(t)})_{\lambda}, \cdots, \frac{\partial}{\partial b_L}(\tilde{Q}_{b(t)})_{\lambda}  \right|_{\lambda = 1, b = 0} = (\Lambda Q, T_1, \cdots, T_L),$$
which implies the nondegeneracy of the Jacobian
$$\left|\left<\frac{\partial}{\partial (\lambda, b_j)}(\tilde{Q}_{b(t)})_{\lambda}, \Ls^i\Phi_M \right>_{1 \leq j \leq L, 0 \leq i \leq L} \right|_{\lambda = 1, b = 0} = \left|\left<\chi_M \Lambda Q, \Lambda Q\right>\right|^{L+1} \ne 0.$$
We now set up the bootstrap for the control of the parameters $(b,\lambda)$ and the radiation $q$. We will measure the regularity of the map through the following coercive norms of $q$:\\
\begin{equation}
\Es_{2k} = \int|\Ls^{k}q|^2 \geq C(M)\sum_{m=0}^{k-1} \int \frac{|\Ls^m q|^2}{1 + y^{4(k - m)}} \quad \text{for}\quad 2 \leq k \leq L+1,
\end{equation}

Our construction is build on a careful choice of the initial data for the modulation parameter $b$ and the radiation $q$ at time $s = s_0$. In particular, we will choose it in the following way:

\begin{definition}[Choice of the initial data] \label{def:1} Let $s_0 \geq 1$, we assume that 
\begin{itemize}
\item Smallness of the initial perturbation for the $b_k$ modes: 
\begin{equation}\label{eq:initbk}
|b_k(s_0)| <s_0^{-k+\frac{1}{3}} \text{for}\;\; 1 \leq k \leq L,
\end{equation}
\item Smallness of the data:
\begin{equation}\label{eq:intialbounE2m}
\sum_{k =2}^{L+1} \Es_{2k}(s_0) +\|q(s_0)\|_{\Hs_{2L+2}}< b_1(s_0)^{10L+4},
\end{equation}
\item Normalization: up to a fixed rescaling, we may always assume
\begin{equation}
\lambda(s_0) = 1.
\end{equation}
\end{itemize}
\end{definition}

\begin{proposition}[Bootstrap]\label{proposition:Bootstrap}
Let $K>0$ be a large enough constant, and  for all $s\in[s_0,s_1]$ and for all $2\leq k\leq L$ assume that
\begin{equation}\label{bootstrap_b}
|b_k|\leq K_1b_1^{k+\frac{1}{2}+\frac{\eta}{10}},
\end{equation}
and
\begin{equation}\label{bootstrapboundE2k}
\Es_{2k}\leq Kb_1^{2(k-1)+(2k+\frac{1}{2})\eta+\frac{\eta}{10}},
\end{equation}
\begin{eqnarray}\label{bootstrapboundE2L2}
\Es_{2L+2}+\|q\|_{\Hs_{2L+2}}\leq Kb_1^{2L+\frac{3}{2}+\frac{\eta}{10}}.
\end{eqnarray}
Then the regime is trapped, in particular when $\frac{s_1}{2}<s\leq s_1$ we obtain the same inequalities \eqref{bootstrap_b},\eqref{bootstrapboundE2k}, and \eqref{bootstrapboundE2L2} with $K_1$ and $K$ replaced by $\frac{K_1}{2}$ and $\frac{K}{2}$.
\end{proposition}
\begin{remark} Note that  $\|q\|_{\Hs_{2L+2}}$ is controlled by $\Es_{2L+2}$ thanks to the coercive property (see Lemma \ref{lemm:coeLk})
$$\Es_{2L+2} \gtrsim \int \frac{|\As (\Ls^{L} q)|^2}{y^2} + \sum_{m = 0}^{L} \int \frac{|\Ls^m q|^2}{y^4(1 + y^{4(L- m)})} + \sum_{m = 0}^{L-1}\frac{|\As(\Ls^{m}q)|^2}{y^6(1 + y^{4(L - m - 1)})},$$
hence, 
$$\|q(s)\|_{\Hs_{2L+2}} \lesssim \Es_{2L+2}.$$
\end{remark}
\section{Closing the bootstrap and proof of Theorem \ref{Theo:1}.} \label{sec:4}

In this section, we aim at proving Proposition \ref{proposition:Bootstrap} which is the heart of our analysis. We proceed in three separate subsections:

- In the first subsection, we derive the laws for the parameters $(b,\lambda)$ thanks to the orthogonality condition \eqref{eq:orthqPhiM} and the coercivity of the powers of $\Ls$.

- In the second subsection, we prove the main monotonicity tools for the control of the infinite dimensional part of the solution. In particular, we derive a suitable Lyapunov functional for the $\Es_{2L+2}$ energy as well as the monotonicity formula for the lower Sobolev energy.

- In the third subsection, we conclude the proof of Proposition \ref{proposition:Bootstrap} thanks to the identities obtained in the first two parts.

\subsection{Modulation equations.}
We derive here the modulation equations for $(b, \lambda)$. The derivation is mainly based on the orthogonality \eqref{eq:orthqPhiM} and the coercivity of the powers of $\Ls$. Let us start with elementary estimates related to the fixed direction $\Phi_M$.
\begin{lemma}[Estimate for $\Phi_M$]\label{lemm:estPhiM} Given $\Phi_M$ as defined in \eqref{def:PhiM}, we have the followings:
$$|c_{k,M}| \lesssim M^{2k} \quad \text{for all}\;\; 1 \leq k \leq L,$$
$$\int |\Phi_M|^2 \lesssim M^3, \quad \int|\Ls \Phi_M|^2 \lesssim \frac{1}{M}.$$
\end{lemma}
\begin{proof}
Arguing by induction, we assume that
\begin{equation*}
|c_{j,M}| \lesssim M^{2j}, \quad 1 \leq j \leq k.
\end{equation*}
Using the fact that $\Ls^j T_{i}$ is admissible of degree $(\max\{0,i-j\},i-j)$, we estimate from the definition \eqref{def:ckM},
\begin{align*}
|c_{k+1,M}|& \lesssim \frac{1}{M^{3}}\sum_{j = 0}^k M^{2j} \int |\chi_M \Lambda Q \Ls^j(T_{k+1})|\\
&\lesssim \frac{1}{M^{3}}\sum_{j = 0}^k M^{2j} \int_{y \leq M}\frac{y^{6}}{y^2}y^{2(k+1 - j) - 2}dy \lesssim M^{2(k+1)}.
\end{align*}
Using the estimate for $c_{k,M}$ yields
\begin{equation*}
\int |\Phi_M|^2 \lesssim \int |\chi_M \Lambda Q|^2 + \sum_{j = 1}^L |c_{j,M}|^2\int|\Ls^j(\chi_M \Lambda Q)|^2 \lesssim M^3,
\end{equation*}
and 
\begin{equation*}
\int|\Ls \Phi_M|^2 \lesssim \sum_{j = 0}^L |c_{j,M}|^2\int |\Ls^{j+1}(\chi_M \Lambda Q)|^2  \lesssim \frac{1}{M}.
\end{equation*}
This concludes the proof of Lemma \ref{lemm:estPhiM}.
\end{proof}

From the orthogonality conditions \eqref{eq:orthqPhiM} and equation \eqref{eq:qys}, we claim the following:
\begin{lemma}[Modulation equations] \label{lemm:mod1} For $K \geq 1$, we assume that there is $s_0(K) \gg 1$ such that $b(s)\mbox{ and } q(s)$ verify the bootstrap assumption for $s \in [s_0, s_1]$ for some $s_1 \geq s_0$. Then, the followings hold for $s \in [s_0, s_1]$:
\begin{equation}\label{eq:ODEbkl}
\sum_{k = 1}^{L-1}\left|(b_k)_s + (2k - 2+C_{b_1})b_1b_k - b_{k+1} \right| + \left|b_1 + \frac{\lambda_s}{\lambda}\right| \lesssim \frac{\sqrt{\Es_{2L+2}}}{M^{\frac{3}{2}}}+b_1^{L+\frac{7}{2}},
\end{equation}
and 
\begin{equation}\label{eq:ODEbL}
\left|(b_L)_s + (2L - 2+C_{b_1})b_1b_L \right| \lesssim  \frac{\sqrt{\Es_{2L+2}}}{M^{\frac{3}{2}}} + b_1^{L +\frac{7}{2}}.
\end{equation}
\end{lemma}
\begin{proof} We start with the law for $b_L$. Let 
$$D(t) = \left|b_1 + \frac{\lambda_s}{\lambda}\right| + \sum_{k = 1}^L \left|(b_k)_s + (2k - 2+C_{b_1})b_1b_k - b_{k+1}\right|,$$
where we recall that $b_{k} \equiv 0$ if $k \geq L + 1$.

Now, we take the inner product of \eqref{eq:qys} with $\Ls^L \Phi_M$ and use the orthogonality \eqref{eq:orthqPhiM} to write
\begin{align}
\left<\widehat{\text{Mod}}(t), \Ls^L\Phi_M\right> &= -\left<\Ls^L\tilde{\Psi}_b, \Phi_M \right> -\left<\Ls^{L+1} q, \Phi_M \right>\nonumber\\
&\qquad  -\left<-\frac{\lambda_s}{\lambda}\Lambda q - \Lc(q) + \Nc(q), \Ls^L\Phi_M \right>.\label{eq:bL}
\end{align}
From the definition \eqref{def:PhiM}, we see that $\Phi_M$ is localized in $y \leq 2M$. From \eqref{def:Modhat} and \eqref{eq:Modt}, we compute by using the identity \eqref{id:TkPhiMi}, 
$$\left<\widehat{\text{Mod}}(t), \Ls^L\Phi_M\right> = (-1)^L\left<\Lambda Q, \Phi_M \right>\left[(b_L)_s + (2L - 2+C_{b_1})b_1b_L\right] + \Oc(M^C b_1D(t)).$$
The error term is estimated by using \eqref{eq:localProfile} with $m = L - 1$ and Lemma \ref{lemm:estPhiM}, 
\begin{align*}
\left|\left<\Ls^L\tilde{\Psi}_b, \Phi_M \right>\right| &\leq \left(\int_{y \leq 2M}|\Ls^L \tilde{\Psi_b}|^2 \right)^\frac 12 \left(\int_{y \leq 2M} |\Phi_M|^2 \right)^\frac 12\\
&\quad \lesssim M^{\frac{3}{2}}b_1^{L+\frac{7}{2}} \lesssim b_1^{L+\frac{7}{2}}.
\end{align*}
For the linear term, we apply the Cauchy-Schwartz inequality,
\begin{align*}
\left|\left<\Ls^{L+1} q,\Phi_M \right>\right| \leq \sqrt{\Es_{2L+2}} \left(\int |\Phi_M|^2 \right)^{\frac 12} \leq M^{\frac{3}{2}}\sqrt{\Es_{2L+2}}.
\end{align*}
The remaining terms are easily estimated by using the following bound coming from Lemma \ref{lemm:coeLk} and Lemma \ref{lemm:coerL},
\begin{equation}\label{est:Es2K1}
\Es_{2L+2}(q) \gtrsim \int \frac{|\Ls q|^2}{y^4(1 + y^{4(L - 1)})} \gtrsim \int \frac{|\py q|^2}{y^4(1 + y^{4(L - 1) + 2})} + \int \frac{q^2}{y^6(1 + y^{4(L - 1) + 2})},
\end{equation}
which implies
$$\left|\left<-\frac{\lambda_s}{\lambda} \Lambda q + \Hc(q) + \Nc(q) , \Ls^L\Phi_M\right>\right| \lesssim M^Cb_1\left(\sqrt{\Es_{2L+2}} + D(t)\right). $$
Put all the above estimates into \eqref{eq:bL} and use \eqref{eq:PhiMLamQ}, we arrive at
\begin{equation}\label{eq:bLtm1}
\left|(b_L)_s + (2L - 2+C_{b_1})b_1b_L\right| \lesssim \frac{\sqrt{\Es_{2L+2}}}{M^{\frac{3}{2}}} + b_1^{L+\frac{7}{2}} + M^Cb_1D(t).
\end{equation}

\medskip

For the modulation equations for $b_k$ with $1 \leq k \leq L-1$, we take the inner product of \eqref{eq:qys} with $\Ls^k \Phi_M$ and use the orthogonality \eqref{eq:orthqPhiM} to write for $1 \leq k \leq L-1$,
\begin{align*}
\left<\widehat{\text{Mod}}(t), \Ls^k\Phi_M\right> &= -\left<\Ls^k\tilde{\Psi}_b, \Phi_M \right> -\left<-\frac{\lambda_s}{\lambda}\Lambda q - \Lc(q) + \Nc(q), \Ls^k\Phi_M \right>.
\end{align*}
Proceed as for $b_L$, we end up with 
\begin{equation}\label{eq:bktm1}
\left|(b_k)_s + (2k - \gamma)b_1b_k - b_{k+1}\right| \lesssim b_1^{L+\frac{7}{2}} + M^Cb_1\left(\sqrt{\Es_{2L+2}} + D(t)\right).
\end{equation}
Similarly, we have by taking the inner product of \eqref{eq:qys} with $\Phi_M$,
\begin{equation}\label{eq:lamtm1}
\left|\frac{\lambda_s}{\lambda} + b_1\right| \lesssim b_1^{L+\frac{7}{2}} + M^Cb_1\left(\sqrt{\Es_{2L+2}} + D(t)\right).
\end{equation}
From \eqref{eq:bLtm1}, \eqref{eq:bktm1} and \eqref{eq:lamtm1}, we obtain the round bound
$$D(t) \lesssim M^C \sqrt{\Es_{2L+2}} + b_1^{L+\frac{7}{2}}.$$
The conclusion then follows by substituting this bound into \eqref{eq:bLtm1}, \eqref{eq:bktm1} and \eqref{eq:lamtm1}. This ends the proof of Lemma \ref{lemm:mod1}.

\end{proof}

From the bound for $\Es_{2L+2}$ given in Proposition \ref{proposition:Bootstrap} and the modulation equation \eqref{eq:ODEbL}, we only have the pointwise bound
$$|(b_L)_s + (2L - 2+C_{b_1})b_1b_L| \lesssim b_1^{L+\frac{3}{4}+\frac{\eta}{4}},$$
which is not good enough  since we expect
$$|(b_L)_s + (2L - 2+C_{b_1})b_1b_L| \lesssim b_1^{L+1+\frac{1}{2}+\frac{\eta}{4}}.$$

We claim that the main linear term can be removed up to an oscillation in time leading to the improved modulation equation for $b_L$ as follows:
Set 
\begin{align}\label{dotprodLQ}
D_L^{b_1}=\left<\Lambda Q+(-1)^L\Ls^L(\sum_{k=0}^2\frac{\p S_{L+k}}{\p b_L}), \chi_{\frac{B_0}{4}}\Lambda Q\right>.
\end{align}

\begin{lemma}[Improved modulation equation for $b_L$] \label{lemm:mod2} Under the assumption of Lemma \ref{lemm:mod1}, the following bound holds for all $s \in [s_0, s_1]$:
\begin{align}
&\left|(b_L)_s + (2L - 2+C_{b_1})b_1b_L + \frac{d}{ds} \left\{ \frac{\left<\Ls^L q, \chi_{\frac{B_0}{4}}\Lambda Q\right>}{\left<\Lambda Q+(-1)^L\Ls^L(\sum_{k=0}^2\frac{\p S_{L+k}}{\p b_L}),\chi_{\frac{B_0}{4}} \Lambda Q \right>}\right\}\right|\nonumber\\
& \qquad \qquad \qquad \lesssim  \frac 1 {B_0^{\frac{3}{2}}} \left[C(M)\sqrt{\Ec_{2L+2}} + b_1^{L+1+\frac{3}{4}+\frac{\eta}{4}}\right].\label{eq:ODEbLimproved}
\end{align}
\end{lemma}
\begin{proof} We apply $\Ls^L$ to \eqref{eq:qys} and take the inner product with $\chi_{\frac{B_0}{4}}\Lambda Q$ ( the choice of $\chi_{\frac{B_0}{4}}$ is to improve the size of $\tilde{\Psi}_b$, since the support of all the bad contributions are beyond $B_0$) to get
\begin{align}
&D_L^{b_1} \left\{\frac{d}{ds}\left[\frac{ \left<\Ls^L q, \chi_{\frac{B_0}{4}}\Lambda Q\right>}{D_L^{b_1}}\right]-  \left<\Ls^L q, \chi_{\frac{B_0}{4}}\Lambda Q\right> \frac{d}{ds}\left[\frac{1}{ D_L^{b_1}} \right] \right\}\nonumber\\
& = \left<\Ls^L q, \Lambda Q \partial_s(\chi_{\frac{B_0}{4}})\right> - \left<\Ls^{L+1} q, \chi_{\frac{B_0}{4}}\Lambda Q\right> + \frac{\lambda_s}{\lambda}\left<\Ls^L\Lambda q,  \chi_{\frac{B_0}{4}}\Lambda Q\right> \nonumber\\
& \quad - \left<\Ls^L \tilde{\Psi}_b,\chi_{\frac{B_0}{4}}\Lambda Q\right> - \left< \Ls^L\widehat{\textup{Mod}}(t), \chi_{\frac{B_0}{4}}\Lambda Q\right>  + \left<\Ls^L\big(\Hc(q) - \Nc(q)\big), \chi_{\frac{B_0}{4}}\Lambda Q\right>.\label{eq:LqLamQt1}
\end{align}

We recall from \eqref{eq:asymLamQ} that 
\begin{equation}\label{est:LQ2chiB0}
B_0^{3} \lesssim |\left<\Lambda Q+(-1)^L\Ls^L(\sum_{k=0}^2\frac{\p S_{L+k}}{\p b_L}), \chi_{\frac{B_0}{4}}\Lambda Q\right> | \lesssim B_0^{3}.
\end{equation}
Let us estimate the second term in the left hand side of \eqref{eq:LqLamQt1}. We use  Cauchy-Schwartz and Lemma \ref{lemm:coeLk} to estimate 
\begin{align}
\left|\left<\Ls^{L} q, \chi_{\frac{B_0}{4}}\Lambda Q\right>\right|& \lesssim B_0^2\|\chi_{\frac{B_0}{4}}\Lambda Q\|_{L^2}\left(\int\frac{|\Ls^L q|^2}{1 + y^4} \right)^\frac 12\nonumber\\
&\quad \lesssim B_0^{\frac 72}\sqrt{\Es_{2L+2}}.\label{est:LsLqchiB0LQ}
\end{align}
We estimate from \eqref{est:LQ2chiB0} 
\begin{align*}
\left|\left<\Ls^L q, \chi_{\frac{B_0}{4}}\Lambda Q\right> \frac{d}{ds}\left[\frac{1}{ D_L^{b_1}} \right]\right|& \lesssim \frac{|\left<\Ls^L q, \chi_{\frac{B_0}{4}}\Lambda Q\right>|}{|D_L^{b_1}|^2} |(b_1)_s||\p_{b_1}D_L^{b_1}|\lesssim \frac{\sqrt{\Es_{2L+2}}}{B_0^{\frac{5}{2}}}.
\end{align*}

For the first three terms in the right hand side of \eqref{eq:LqLamQt1}, we use the Cauchy-Schwartz, Lemma \ref{lemm:coeLk} and the fact that $\Ls(\Lambda Q) = 0$ to find that
\begin{align*}
\left|\left<\Ls^{L} q,\Lambda Q \partial_s(\chi_{\frac{B_0}{4}})\right>\right|& \lesssim \left|\frac{(b_1)_s}{b_1} \right| \left(\int_{\frac{B_0}{4} \leq y \leq \frac{B_0}{2}} (1 + y^4) |\Lambda Q|^2 \right)^\frac 12 \left(\int\frac{|\Ls^L q|^2}{1 + y^4} \right)^\frac 12\\
&\quad \lesssim b_1^{\frac 32} B_0^{\frac 72}\sqrt{\Es_{2L+2}} \lesssim B_0^{\frac 12}\sqrt{\Es_{2L+2}}
\end{align*}
and
\begin{align*}
\left|\left<\Ls^{L+1} q, \chi_{\frac{B_0}{4}}\Lambda Q\right>\right|& \lesssim \left(\int |\chi_{\frac{B_0}{4}}\Lambda Q|^2 \right)^\frac 12 \left(\int|\Ls^{L+1} q|^2 \right)^\frac 12\\
&\quad \lesssim B_0^{\frac 32}\sqrt{\Es_{2L+2}},
\end{align*}
and 
\begin{align*}
\left|\frac{\lambda_s}{\lambda}\left<\Ls^L\Lambda q,  \chi_{\frac{B_0}{4}}\Lambda Q\right> \right| & \lesssim b_1 \left(\int (1 + y^{4L+2})|\Ls^L(\chi_{\frac{B_0}{4}} \Lambda Q)|^2 \right)^\frac 12 \left(\int \frac{|\py q|^2}{1 + y^{4L + 2}}\right)^\frac 12 \\
&\lesssim B_0^{\frac 12}\sqrt{\Es_{2L+2}}.
\end{align*}
The error term is estimated by using \eqref{eq:estPsiblocalB0},
\begin{align*}
\left|\left<\Ls^L \tilde{\Psi}_b,\chi_{\frac{B_0}{4}}\Lambda Q\right>\right| &\lesssim \left(\int (1 + y^{4(L + 1)})|\Ls^L(\chi_{\frac{B_0}{4}} \Lambda Q)|^2 \right)^\frac 12\left(\int_{\{y\leq\frac{B_0}{2}\}}\frac{|\tilde{\Psi}_b|^2}{1 + y^{4(L + 1)}} \right)^\frac 12\\
&\quad \lesssim B_0^{\frac 72} b_1^{L + 1 + \frac 54 -\frac{\eta}{2}}\lesssim B_0^{\frac 32}b_1^{L+\frac{3}{4}+\frac{\eta}{2}}.
\end{align*}
The last term in the right hand side of \eqref{eq:LqLamQt1} is estimated in the same way, by using that $|\Nc(q)|\lesssim\frac{|q|^2}{y^2}$ and $|\Hc(q)|\lesssim\frac{|q|}{y^2}$,
\begin{align*}
\left|\left<\Ls^L\big(\Hc(q) - \Nc(q)\big), \chi_{\frac{B_0}{4}}\Lambda Q\right> \right|& \lesssim \int |\Hc(q) \Ls^L(\chi_{\frac{B_0}{4}} \Lambda Q)| + \int |\Nc(q) \Ls^L(\chi_{\frac{B_0}{4}} \Lambda Q)|\\
& \lesssim \left(\int \frac{|\Hc(q)|^2}{1 + y^{4L}}\right)^\frac{1}{2}\left(\int (1 + y^{4L}) |\Ls^L(\chi_{\frac{B_0}{4}}\Lambda Q)|^2) \right)^\frac{1}{2}\\
&\quad  + \int \frac{|\Nc(q)|}{1 + y^{4L+2}}\left\| (1 + y^{4L+2}) |\Ls^L(\chi_{\frac{B_0}{4}}\Lambda Q)| \right\|_{L^\infty}\\
& \qquad \lesssim B_0^{\frac 32} \sqrt{\Es_{2L+2}} + B_0^{2L} \Es_{2L+2}\\
& \qquad \quad\lesssim B_0^{\frac 32} \sqrt{\Es_{2L+2}}.
\end{align*}
For the remaining term, we recall that $\Ls(\Lambda Q) = 0$, $\Ls^L T_k = 0$ for $1 \leq k \leq L-1$, $\Ls^L T_L = (-1)^L \Lambda Q$ and note that the support of $1-\chi_{B_1}$ and $\chi_{B_0}$ are disjoint, which we will use in the decomposition 
$$\Ls^L(T_k\chi_{B_1})=\Ls^L(T_k)+\Ls^L(-T_k(1-\chi_{B_1})).$$
 We then write from \eqref{def:Modhat} and \eqref{eq:Modt}, 
\begin{align*}
&\left|\left< \Ls^L\widehat{\textup{Mod}}(t),\chi_{\frac{B_0}{4}}\Lambda Q\right> - (-1)^LD_L^{b_1} \left[(b_L)_s + (2L - 2+C_{b_1})b_1b_L \right] \right|\\
&\quad \lesssim \sum_{k = 1}^{L-1} \left|(b_k)_s + (2k - 2+C_{b_1})b_1b_k - b_{k+1} \right| \left|\left<\sum_{j = k+1}^{L+2} \frac{\partial \tilde{S}_j}{\partial b_k}, \Ls^L (\chi_{\frac{B_0}{4}}\Lambda Q)\right>\right|\\
&\qquad + \left|\frac{\lambda_s}{\lambda} + b_1\right| \left|\left<\Lambda \tilde{\Theta}_b,\Ls^L( \chi_{\frac{B_0}{4}}\Lambda Q)\right>\right|.
\end{align*}
Recall that $T_k$ is admissible of degree $(k,k)$ and $S_k$ is homogeneous of degree $(k, k, k)$, we derive the round bounds for $y \leq B_0$:
$$|\Lambda \Theta_b| \lesssim \frac{b_1}{y}, \quad \sum_{j = k+1}^{L+2}\left|\frac{\partial S_j}{\partial b_k} \right| \leq \sum_{j = k+1}^{L+2}b_1^{(j-k)+\frac{1}{2}}y^{2j-3} \lesssim b_1^{\frac 52-L}.$$
Thus, from Lemma \ref{lemm:mod1}, we derive the bound
\begin{align*}
&\left|\frac{\lambda_s}{\lambda} + b_1\right| \left|\left<\Lambda \tilde{\Theta}_b,\Ls^L(\chi_{\frac{B_0}{4}}\Lambda Q)\right>\right|\\
& \quad  + \sum_{k = 1}^{L} \left|(b_k)_s + (2k - 2+C_{b_1})b_1b_k - b_{k+1} \right| \left|\left<\sum_{j = k+1}^{L+2} \frac{\partial \tilde{S}_j}{\partial b_k}, \Ls^L (\chi_{\frac{B_0}{4}}\Lambda Q)\right>\right|\\
& \qquad \lesssim \left( C(M)\sqrt{\Es_{2L+2}} +  b_1^{L + 1 + \frac{3}{4}+\frac{\eta}{2}}\right)B_0^{\frac 32}.
\end{align*}

The equation \eqref{eq:ODEbLimproved} follows by gathering all the above estimates into \eqref{eq:LqLamQt1}, dividing both sides of \eqref{eq:LqLamQt1} by $(-1)^L D_L^{b_1}$. This finishes the proof of Lemma \ref{lemm:mod2}.
\end{proof}

\subsection{Monotonicity.}

We derive in this subsection the main monotonocity formula for $\Es_{2k}$ for $ 2 \leq k \leq L+1$. We claim the following which is the heart of this paper:

\begin{proposition}[Lyapounov monotonicity for the high Sobolev norm] \label{prop:E2k} We have 
\begin{align}
&\frac{d}{dt}\left\{\frac{\Es_{2L+2}}{\lambda^{4L- 3}}\left[1 + \Oc\left(b_1^{\frac{\eta}{8}}\right) \right]\right\}\nonumber\\
&\qquad \qquad \leq \frac{b_1}{\lambda^{4L- 1}}\left[\frac{\Es_{2L+2}}{M^{\frac{3}{2}}} + b_1^{L+\frac{3}{4}+\frac{\eta}{4}}\sqrt{\Es_{2L+2}} + b_1^{2L+\frac 32+\frac{\eta}{2}}\right],\label{eq:Es2sLya}
\end{align}
and for $2\leq m \leq L$,
\begin{align}
&\frac{d}{dt}\left\{\frac{\Es_{2m}}{\lambda^{4m - 7}}\left[1 + \Oc(b_1) \right]\right\} \leq \frac{b_1}{\lambda^{4m - 5}}\left[b_1^{m-1+(\frac 12+2m)\frac{\eta}{4}}\sqrt{\Es_{2m}} + b_1^{2(m-1)+(\frac 12+2m)\frac{\eta}{2}}\right],\label{eq:Es2sLyam}
\end{align}
\end{proposition}

\begin{proof} The proof uses some ideas developed in \cite{RSapde2014} and \cite{MRRcmj2016}. Because the proof of \eqref{eq:Es2sLyam} follows exactly the same lines as for \eqref{eq:Es2sLya}, we only deal with the proof of \eqref{eq:Es2sLya}. Let us start the proof of \eqref{eq:Es2sLya}.

\noindent \textbf{Step 1: Suitable derivatives and energy identity.} For $k \in \mathbb{N}$, we define the suitable derivatives of $q$ and $v$ as follows:
\begin{equation}\label{eq:notq2k1}
q_{2k} = \Ls^kq, \quad q_{2k + 1} = \As \Ls^k q, \quad v_{2k} = \Ls^k_\lambda v, \quad v_{2k + 1} = \As_\lambda \Ls^k_\lambda v,
\end{equation}
where $q = q(y,s)$ and $v = v(r,t)$ satisfy \eqref{eq:qys} and \eqref{eq:vrt} respectively, the linearized operator $\Ls$ and $\Ls_\lambda$ are defined by \eqref{def:Lc} and \eqref{def:Llambda}, $\As$ and $\As^*$ are the first order operators defined by \eqref{def:As} and \eqref{def:Astar}, and
$$\As_\lambda f = -\partial_r f + \frac{V_\lambda}{r} f, \quad \As^*_\lambda f = \frac{1}{r^{6}}\partial_r(r^{6}f) + \frac{V_\lambda}{r} f,$$
with $V = \Lambda \log \Lambda Q$ admitting the asymptotic behavior as in \eqref{eq:asympV}.

With the notation \eqref{eq:notq2k1}, we note that 
$$q_{2k + 1} = \As q_{2k}, \quad q_{2k+2} = \As^*q_{2k + 1}, \quad v_{2k + 1} = \As_\lambda v_{2k}, \quad v_{2k+2} = \As^*_\lambda v_{2k + 1}.$$
Recall from Lemma \ref{lemm:factorL}, we have the following factorization:
$$\Ls  = \As^* \As, \quad \tilde\Ls = \As \As^*,  \quad \Ls_\lambda = \As^*_\lambda \As_\lambda, \quad \tilde{\Ls}_\lambda = \As_\lambda \As^*_\lambda,$$
where 
\begin{equation}\label{def:Lstil}
\tilde{\Ls} = -\partial_{yy} - \frac{6}{y}\py + \frac{\tilde{Z}}{y^2},
\end{equation}
and 
\begin{equation}\label{def:LstilLam}
\tilde{\Ls}_\lambda = -\partial_{rr} - \frac{6}{r}\partial_r + \frac{\tilde{Z}_\lambda}{r^2},
\end{equation}
with $\tilde{Z}$ expressed in terms of $V$ as in \eqref{def:ZtilbyV}.\\

We apply $\Ls_\lambda^{L}$ to \eqref{eq:vrt} and use the notation \eqref{eq:notq2k1} to derive 
\begin{equation}\label{eq:v2k2}
\pt v_{2L} + \Ls_\lambda v_{2L} = [\pt, \Ls^{L}_\lambda]v + \Ls_\lambda^{L}\left( \frac{1}{\lambda^2} \Fc_\lambda\right).
\end{equation}
Now apply $\As_\lambda$ to \eqref{eq:v2k2} yields
\begin{equation}\label{eq:v2k1}
\pt v_{2L+1} + \tilde \Ls_\lambda v_{2L+1} = \frac{\pt V_\lambda}{r} v_{2L} +\As_\lambda [\pt, \Ls^{L}_\lambda]v + \As_\lambda \Ls_\lambda^{L}\left( \frac{1}{\lambda^2} \Fc_\lambda\right), 
\end{equation}
Since $\Ls_\lambda = \frac{1}{\lambda^2}\Ls$, we then have 
$$\Ls^k_\lambda v = \frac{1}{\lambda^{2k}}\Ls^k q, \quad \text{hence,}\quad \int|\Ls^k_\lambda v|^2 = \frac{1}{\lambda^{4k - d}}\int |\Ls^k q|^2.$$
Using the definition \eqref{def:LstilLam} of $\tilde{\Ls}_\lambda$ and an integration by parts, we write
\begin{align*}
\frac{1}{2}\frac{d}{dt} \left(\frac{1}{\lambda^{4L-3}} \Es_{2L+2}\right) &= \frac{1}{2}\frac{d}{dt} \int |\Ls_\lambda^{L+1} v|^2 = \frac{1}{2}\frac{d}{dt}\int \tilde{\Ls}_\lambda v_{2L+1} v_{2L+1}\\
&= \int \tilde{\Ls}_\lambda v_{2L+1} \pt v_{2L+1} + \frac{1}{2}\int \frac{\pt (\tilde Z_\lambda)}{r^2}v^2_{2L+1}\\
&= \int \tilde{\Ls}_\lambda v_{2L+1} \pt v_{2L+1} + b_1\int \frac{(\Lambda \tilde{Z})_\lambda}{2\lambda^2r^2}v^2_{2L+1} - \left(\frac{\lambda_s}{\lambda} + b_1\right) \int \frac{(\Lambda \tilde{Z})_\lambda}{2\lambda^2r^2}v^2_{2L+1}.
\end{align*}
Using the definition \eqref{def:Astar} of $\As^*$ and an integration by parts together with the definition \eqref{def:ZtilbyV} of $\tilde{Z}$, we write
\begin{align*}
\int \frac{b_1(\Lambda V)_\lambda}{\lambda^2 r} v_{2L+1} \As^*_\lambda v_{2L+1} &= \frac{b_1}{\lambda^{4L-1}}\int \frac{\Lambda V}{y}q_{2L+1} \As^*q_{2L+1}\\
&= \frac{b_1}{\lambda^{4L-1}}\int \frac{\Lambda V(2V + 7) - \Lambda^2 V}{2y^2}q^2_{2L+1}\\
&= \frac{b_1}{\lambda^{4L-1}} \int \frac{\Lambda \tilde{Z}}{2y^2}q^2_{2L+1} = \int \frac{b_1(\Lambda \tilde{Z})_\lambda}{2\lambda^2r^2}v^2_{2L+1}.
\end{align*}
From \eqref{eq:v2k2}, we write
\begin{align*}
\frac{d}{dt}\int \frac{b_1(\Lambda V)_\lambda}{\lambda^2r}v_{2L+1} v_{2L} &= \int \frac{d}{dt}\left(\frac{b_1(\Lambda V)_\lambda}{\lambda^2r}\right)v_{2L+1} v_{2L}  + \int \frac{b_1(\Lambda V)_\lambda}{\lambda^2r}  v_{2L}\pt v_{2L+1}\\
& + \int \frac{b_1(\Lambda V)_\lambda}{\lambda^2r} v_{2L+1} \left[-\As^*_\lambda v_{2L+1} + [\pt, \Ls^{L}_\lambda]v + \Ls_\lambda^{L}\left( \frac{1}{\lambda^2} \Fc_\lambda\right)\right].
\end{align*}
Gathering all the above identities and using \eqref{eq:v2k1} yields the energy identity
\begin{align}
&\frac{1}{2}\frac{d}{dt}\left\{ \left(\frac{1}{\lambda^{4L-3}} \Es_{2L+2}\right) + 2 \int \frac{b_1(\Lambda V)_\lambda}{\lambda^2r}v_{2L+1} v_{2L} \right\} \label{eq:EnerID}\\
&= - \int |\tilde\Ls_\lambda v_{2L+1}|^2 - \left(\frac{\lambda_s}{\lambda} + b_1\right) \int \frac{(\Lambda \tilde{Z})_\lambda}{2\lambda^2r^2}v^2_{2L+1} - \int \frac{b_1(\Lambda V)_\lambda}{\lambda^2r} v_{2L} \tilde{\Ls}_\lambda v_{2L+1}\nonumber\\
&\quad + \int \frac{d}{dt}\left(\frac{b_1(\Lambda V)_\lambda}{\lambda^2r}\right)v_{2L+1} v_{2L} + \int \frac{b_1(\Lambda V)_\lambda}{\lambda^2r} v_{2L+1} \left[[\pt, \Ls^{L}_\lambda]v + \Ls_\lambda^{L}\left( \frac{1}{\lambda^2} \Fc_\lambda\right)\right]\nonumber\\
& \quad + \int \left(\tilde{\Ls}_\lambda v_{2L+1} +  \frac{b_1(\Lambda V)_\lambda}{\lambda^2r}  v_{2L}\right)\left[\frac{\pt V_\lambda}{r} v_{2L} +\As_\lambda [\pt, \Ls^{L}_\lambda]v + \As_\lambda \Ls_\lambda^{L}\left( \frac{1}{\lambda^2} \Fc_\lambda\right) \right].\nonumber
\end{align}
We now estimate all terms in \eqref{eq:EnerID}. The proof uses the coercivity estimate given in Lemma \ref{lemm:coeLk}. In particular, we shall apply Lemma \ref{lemm:coeLk} with $k = L$ to have the estimate
\begin{equation}\label{eq:Enercontrol}
\Es_{2L+2} \gtrsim \int \frac{|q_{2L+1}|^2}{y^2} + \sum_{m = 0}^{L} \int \frac{|q_{2m}|^2}{y^4(1 + y^{4(L- m)})} + \sum_{m = 0}^{L-1}\int \frac{|q_{2m +1}|^2}{y^6(1 + y^{4(L-1-m)})}.
\end{equation}
 
\noindent \textbf{Step 2: Control of the lower order quadratic terms.} Let us start with the second term in the left hand side of \eqref{eq:EnerID}. From \eqref{eq:asympV} and \eqref{def:ZtilbyV}, we have the round bound
\begin{equation}\label{eq:estLamZV}
|\Lambda \tilde{Z}(y)| + |\Lambda V(y)| \lesssim \frac{y}{1 + y^2}, \quad \forall y \in [0, +\infty).
\end{equation}
By making a change of variables and using the Cauchy-Schwartz inequality, we obtain that
\begin{align*}
\left|\int \frac{b_1(\Lambda V)_\lambda}{\lambda^2r}v_{2L+1} v_{2L}\right| &= \left|\frac{b_1}{\lambda^{4L-3}}\int \frac{\Lambda V}{y}q_{2L+1}q_{2L}\right|\\
&\lesssim \frac{b_1}{\lambda^{4L-3}} \left(\int \frac{|q_{2L+1}|^2}{y^2}\right)^\frac 12 \left(\int \frac{|q_{2L}|^2}{1 + y^2}\right)^\frac 12.
\end{align*}
From Lemma \ref{lemm:interbounds} together with \eqref{eq:Enercontrol} we have
$$\left(\int \frac{|q_{2L+1}|^2}{y^2}\right)^\frac 12 \left(\int \frac{|q_{2L}|^2}{1 + y^2}\right)^\frac 12\lesssim (\Es_{2L})^{\frac 14}(\Es_{2L+2})^{\frac 34},$$
which implies

\begin{align*}
\left|\int \frac{b_1(\Lambda V)_\lambda}{\lambda^2r}v_{2L+1} v_{2L}\right| &\lesssim \frac{b_1}{\lambda^{4L-3}} (\Es_{2L})^{\frac 14}(\Es_{2L+2})^{\frac 34}.
\end{align*}
By using the boostrap bounds we get
\begin{align*}
\left|\int \frac{b_1(\Lambda V)_\lambda}{\lambda^2r}v_{2L+1} v_{2L}\right| &\lesssim \frac{b_1^{\frac{1}{2}}}{\lambda^{4L-3}}\Es_{2L+2}.
\end{align*}

Using \eqref{eq:estLamZV}, \eqref{eq:ODEbkl},\eqref{eq:Enercontrol} and the bootstrap bound \eqref{bootstrapboundE2L2}, we deduce
\begin{align*}
\left|\left(\frac{\lambda_s}{\lambda} + b_1\right)\int \frac{(\Lambda \tilde{Z})_\lambda}{\lambda^2 r}v^2_{2L+1} \right| &= \left|\left(\frac{\lambda_s}{\lambda} + b_1\right) \frac{1}{\lambda^{4L-1}}\int \frac{\Lambda \tilde{Z}}{y^2}q^2_{2L+1} \right|\\
& \lesssim \frac{(b_1^{L+1 + \frac{3}{4}+\frac{\eta}{2}}+\sqrt{\Es_{2L+2}})}{\lambda^{4L-1}}\int \frac{q^2_{2L+1}}{y^2} \lesssim \frac{b_1^{2}}{\lambda^{4L-1}}\Es_{2L+2}.
\end{align*}
For the third term in the right hand side of \eqref{eq:EnerID}, we write
\begin{align*}
\left|\int \frac{b_1(\Lambda V)_\lambda}{\lambda^2r} v_{2L} \tilde{\Ls}_\lambda v_{2L+1}\right| &\leq \frac{1}{4}\int |\tilde{\Ls}_\lambda v_{2L+1}|^2 + 4\int \left(\frac{b_1(\Lambda V)_\lambda}{\lambda^2r}\right)^2 v^2_{2L}\\
&= \frac{1}{4}\int |\tilde{\Ls}_\lambda v_{2L+1}|^2 + \frac{4b_1^2}{\lambda^{4L-1}}\int \frac{|\Lambda V|^2}{y^2}q^2_{2L}\\
& \leq \frac{1}{4}\int |\tilde{\Ls}_\lambda v_{2L+1}|^2 + \frac{Cb_1^2}{\lambda^{4L-1}} \Es_{2L+2}.
\end{align*}
A direct computation yields the round bound
$$\left|\frac{d}{dt}\left(\frac{b_1(\Lambda V)_\lambda}{\lambda^2}\right)\right| \lesssim \frac{b_1^{\frac 52}}{\lambda^4}(|\Lambda V| + |\Lambda^2V|).$$
Thus, we use \eqref{eq:estLamZV}, the Cauchy-Schwartz inequality, \eqref{eq:Enercontrol}, Lemma \ref{lemm:interbounds} and bootstrap bound \eqref{bootstrapboundE2L2} to estimate
\begin{align*}
\left|\int \frac{d}{dt}\left(\frac{b_1(\Lambda V)_\lambda}{\lambda^2r}\right)v_{2L+1} v_{2L} \right| &\lesssim \frac{b_1^{\frac 52}}{\lambda^{4L-1}}\int \frac{|\Lambda V| + |\Lambda^2V|}{y}|q_{2L+1}q_{2L}|\\
& \lesssim \frac{b_1^{\frac 52}}{\lambda^{4L-1}}\left(\int \frac{q^2_{2L+1}}{y^2}\right)^\frac 12 \left(\int \frac{q^2_{2L}}{1 + y^2}\right)^\frac 12\\
& \lesssim \frac{b_1^{\frac 52}}{\lambda^{4L-1}}(\Es_{2L+2})^{\frac 34}(\Es_{2L})^{\frac 14}\lesssim\frac{b_1^{2}}{\lambda^{4L-1}}\Es_{2L+2}.
\end{align*}
Similarly, we have 
\begin{align*}
&\left|\int \left(\tilde{\Ls}_\lambda v_{2L+1} +  \frac{b_1(\Lambda V)_\lambda}{\lambda^2r}  v_{2L}\right)\frac{\pt V_\lambda}{r} v_{2L}\right|\\
&\qquad \leq \frac{1}{4}\int |\tilde{\Ls}_\lambda v_{2L+1}|^2 + \frac{Cb_1^2}{\lambda^{4L-1}} \int \frac{|\Lambda V|^2}{y^2}q^2_{2L}\\
&\qquad \leq \frac{1}{4}\int |\tilde{\Ls}_\lambda v_{2L+1}|^2 + \frac{Cb_1^2}{\lambda^{4L-1}}\Es_{2L+2},
\end{align*}
and 
\begin{align*}
& \left|\int \frac{b_1(\Lambda V)_\lambda}{\lambda^2r} v_{2L+1}[\pt, \Ls^{L}_\lambda]v \right| + \left|\int \left(\tilde{\Ls}_\lambda v_{2L+1} +  \frac{b_1(\Lambda V)_\lambda}{\lambda^2r}  v_{2L}\right)\As_\lambda [\pt, \Ls^{L}_\lambda]v \right|\\
&\leq \frac{1}{4}\int |\tilde{\Ls}_\lambda v_{2L+1}|^2 + C\left(\frac{b_1^2}{\lambda^{4L- 4}}\Es_{2L+2} + \int \frac{\big|[\pt, \Ls_\lambda^{L}]v\big|^2}{\lambda^2(1 + y^2)} + \int \big|\As_\lambda[\pt, \Ls_\lambda^{L}]v\big|^2\right).
\end{align*}
We claim the bound 
\begin{equation}\label{est:comtor}
\int \frac{\big|[\pt, \Ls_\lambda^{L}]v\big|^2}{\lambda^2(1 + y^2)} + \int \big|\As_\lambda[\pt, \Ls_\lambda^{L}]v\big|^2 \lesssim \frac{b_1^2}{\lambda^{4L-1}}\Es_{2L+2},
\end{equation}
whose proof is left to Appendix \ref{ap:EstComm}. 

The collection of all the above estimates to \eqref{eq:EnerID} yields 
\begin{align}
\frac{1}{2}\frac{d}{dt}\left\{ \frac{\Es_{2L+2}}{\lambda^{4L-1}} \Big[1 + \Oc(b_1^{\frac 12})\Big] \right\} &\leq -\frac{1}{4} \int |\tilde\Ls_\lambda v_{2L+1}|^2 + \frac{Cb_1^2}{\lambda^{4L-1}}\Es_{2L+2}\nonumber\\
& \quad + \int \frac{b_1(\Lambda V)_\lambda}{\lambda^2r} v_{2L+1} \Ls_\lambda^{L}\left( \frac{1}{\lambda^2} \Fc_\lambda\right)\nonumber\\
& \quad +\int \frac{b_1(\Lambda V)_\lambda}{\lambda^2r}  v_{2L} \As_\lambda\Ls_\lambda^{L}\left( \frac{1}{\lambda^2} \Fc_\lambda\right)\nonumber\\
&\quad + \int \tilde{\Ls}_\lambda v_{2L+1}\As_\lambda \Ls_\lambda^{L}\left( \frac{1}{\lambda^2} \Fc_\lambda\right).\label{eq:EnerID1}
\end{align}

\noindent \textbf{Step 3: Further use of dissipation.} We aim at estimating all terms in the right hand side of \eqref{eq:EnerID1}. From \eqref{eq:estLamZV}, \eqref{eq:Enercontrol} and the Cauchy-Schwartz inequality, we write 
\begin{align*}
\left|\int \frac{b_1(\Lambda V)_\lambda}{\lambda^2r} v_{2L+1} \Ls_\lambda^{L}\left( \frac{1}{\lambda^2} \Fc_\lambda\right) \right|&= \left|\frac{b_1}{\lambda^{4L-1}}\int \frac{\Lambda V}{y}q_{2L+1}\Ls^{L}\Fc \right|\\
& \quad \lesssim \frac{b_1}{\lambda^{4L-1}} \left(\int q_{2L+1}^2 \right)^\frac{1}{2}\left(\int \frac{|\Ls^{L}\Fc|^2}{1 + y^4} \right)^\frac{1}{2}\\
& \quad \lesssim \frac{b_1}{\lambda^{4L-1}} (\Es_{2L})^{\frac 14}(\Es_{2L+2})^{\frac 14}\left(\int \frac{|\Ls^{L}\Fc|^2}{1 + y^4} \right)^\frac{1}{2}.
\end{align*}
Similarly, we have 
\begin{align*}
\left|\int \frac{b_1(\Lambda V)_\lambda}{\lambda^2r} v_{2L} \As_\lambda\Ls_\lambda^{L}\left( \frac{1}{\lambda^2} \Fc_\lambda\right) \right|&= \left|\frac{b_1}{\lambda^{4L-1}}\int \frac{\Lambda V}{y}q_{2L}\As\Ls^{L}\Fc \right|\\
& \quad \lesssim \frac{b_1}{\lambda^{4L-1}} \left(\int \frac{q_{2L}^2}{1+y^2} \right)^\frac{1}{2}\left(\int \frac{|\As\Ls^{L}\Fc|^2}{1 + y^2} \right)^\frac{1}{2}\\
& \quad \lesssim \frac{b_1}{\lambda^{4L-1}} (\Es_{2L})^{\frac 14}(\Es_{2L+2})^{\frac 14}\left(\int \frac{|\As\Ls^{L}\Fc|^2}{1 + y^2} \right)^\frac{1}{2}.
\end{align*}
For the last term in \eqref{eq:EnerID1}, let us introduce the function
\begin{equation}\label{def:xiL}
\xi_L = \frac{\left<\Ls^Lq, \chi_{B_0}\Lambda Q\right>}{\left<\Lambda Q+(-1)^L\Ls^L(\sum_{k=0}^2\frac{\p S_{L+k}}{\p{b_L}}), \chi_{B_0}\Lambda Q\right>}\chi_{B_1}\left(T_L+\sum_{k=0}^2\frac{\p S_{L+k}}{\p b_L}\right),
\end{equation}
and the decomposition 
\begin{equation}\label{eq:decomF}
\Fc = \partial_s \xi_L + \Fc_0 + \Fc_1, \quad \Fc_0 = - \tilde{\Psi}_b - \underbrace{(\widehat{\text{Mod}} + \partial_s \xi_L)}_{\equiv \widetilde{\text{Mod}}}, \quad \Fc_1 = \Hc(q) - \Nc(q),
\end{equation}
where $\tilde \Psi_b$ is referred to \eqref{def:Psibtilde}, $\widehat{\text{Mod}}$, $\Hc(q)$ and $\Nc(q)$ are defined as in \eqref{def:Modhat} \eqref{def:Lq} and \eqref{def:Nq} respectively.
Actually, we introduced the decomposition \eqref{eq:decomF} and $\xi_L$ to take advantage of the improved bound obtained in Lemma  \ref{lemm:mod2}.
We now write
\begin{align*}
&\int \tilde{\Ls}_\lambda v_{2L+1}\As_\lambda \Ls_\lambda^{L}\left( \frac{1}{\lambda^2} \Fc_\lambda\right)\\
& \qquad = \frac{1}{\lambda^{4L-1}}\left(\int \As^* q_{2L+1}\Ls^{L+1}(\partial_s \xi_L) + \int \As^* q_{2L+1}\Ls^{L+1} \Fc_0 + \int \tilde{\Ls}q_{2L+1} \As \Ls^{L}\Fc_1\right)\\
& \qquad \leq \frac{1}{\lambda^{4L-1}}\int \Ls^{L+1} q \Ls^{L+1}(\partial_s \xi_L) + \frac{C}{\lambda^{4L-1}} \left(\int |\Ls^{L+1} q|^2\right)^\frac{1}{2}\left(\int |\Ls^{L+1} \Fc_0| \right)^\frac 12\\
&\qquad \qquad \qquad \qquad \qquad \qquad \qquad  \qquad + \frac{1}{8}\int |\tilde{\Ls}_\lambda v_{2L+1}|^2 + \frac{C}{\lambda^{4L-1}}\int |\As \Ls^{L}\Fc_1|^2\\
&\qquad = \frac{1}{\lambda^{4L-1}}\int \Ls^{L+1} q \Ls^{L+1}(\partial_s \xi_L) + \frac{1}{8}\int |\tilde{\Ls}_\lambda v_{2L+1}|^2 \\
&\qquad \qquad+ \frac{C}{\lambda^{4L-1}}\left( \sqrt{\Es_{2L+2}}\left\|\Ls^{L+1} \Fc_0 \right\|_{L^2} + \left\|\As \Ls^{L}\Fc_1 \right\|_{L^2}^2 \right).
\end{align*}
Injecting all these bounds into \eqref{eq:EnerID1} yields
\begin{align}
&\frac{1}{2}\frac{d}{dt}\left\{ \frac{\Es_{2L+2}}{\lambda^{4L-3}}\Big[1  + \Oc(b_1^{\frac 12})\Big] \right\}\nonumber\\
&\qquad \leq -\frac{1}{8} \int |\tilde\Ls_\lambda v_{2L+1}|^2 + \frac{Cb_1^2}{\lambda^{4L-1}}\Es_{2L+2} + \frac{1}{\lambda^{4L-1}}\int \Ls^{L+1} q \Ls^{L+1}(\partial_s \xi_L)\nonumber\\
& \qquad \quad + \frac{b_1}{\lambda^{4L-1}} (\Es_{2L})^{\frac 14}(\Es_{2L+2})^{\frac 14}\left[\left(\int \frac{|\As\Ls^{L}\Fc|^2}{1 + y^2} \right)^\frac{1}{2} + \left(\int \frac{|\Ls^{L}\Fc|^2}{1 + y^4} \right)^\frac{1}{2} \right]\nonumber\\
&\qquad  \quad + \frac{C}{\lambda^{4L-1}}\left( \sqrt{\Es_{2L+2}}\left\|\Ls^{L+1} \Fc_0 \right\|_{L^2} + \left\|\As \Ls^{L}\Fc_1 \right\|_{L^2}^2 \right). \label{eq:EnerID2}
\end{align}
\noindent \textbf{Step 4: Estimates for $\tilde \Psi_b$ term.} Recall from \eqref{eq:estPsibLarge2} that we already have the following estimate for $\tilde{\Psi}_b$:
\begin{equation}\label{eq:PsibtilE2k}
\left\|\Ls^{L+1} \tilde\Psi_b \right\|_{L^2} + \left(\int \frac{|\As\Ls^{L}\tilde\Psi_b|^2}{1 + y^2} \right)^\frac{1}{2} + \left(\int \frac{|\Ls^{L}\tilde\Psi_b|^2}{1 + y^4} \right)^\frac{1}{2} \lesssim b_1^{L + 1 + \frac{3}{4}+\frac{\eta}{4}}.
\end{equation}
\noindent \textbf{Step 5: Estimates for $\widetilde{\text{Mod}}$ term.} We claim the following:
\begin{align}
\left(\int \frac{|\Ls^{L}\widetilde{\text{Mod}}|^2}{1 + y^4} \right)^\frac{1}{2} + \left(\int \frac{|\As\Ls^{L}\widetilde{\text{Mod}}|^2}{1 + y^2} \right)^\frac{1}{2}+\left(\int \left|\Ls^{L+1} \widetilde{\text{Mod}}\right|^2 \right)^\frac{1}{2}\nonumber\\
\lesssim b_1^{1+\frac \eta4}\left(\frac{\sqrt{\Es_{2L+2}}}{M^{\frac{3}{2}}} + b_1^{L+ \frac{3}{4}+\frac{\eta}{4}}\right),\label{eq:ModtilE2k}
\end{align}
where
\begin{equation*}
\widetilde{\text{Mod}} = \widehat{\text{Mod}} + \partial_s \xi_L.
\end{equation*}
\noindent Let us prove \eqref{eq:ModtilE2k}. We only deal with the first term since the second term is estimated similarly.We write
\begin{align*}
\widetilde{\text{Mod}} &= - \left(\frac{\lambda_s}{\lambda} + b_1\right)\Lambda \tilde{Q}_b + \sum_{i = 1}^{L-1}\big[(b_i)_s + (2i - 2+C_{b_1})b_1 b_i - b_{i + 1}\big]\tilde T_i\\
& \quad + \sum_{i = 1}^{L-1}\big[(b_i)_s + (2i - 2+C_{b_1})b_1 b_i - b_{i + 1}\big] \chi_{B_1}\sum_{j = i+1}^{L+2}\frac{\partial S_j}{\partial b_i}\\
&\quad + \left[(b_L)_s + (2i - 2+C_{b_1})b_1 b_L + \frac{d}{ds}\left\{\frac{\left<\Ls^Lq, \chi_{B_0}\Lambda Q\right>}{D_{L}^{b_1}} \right\} \right]\chi_{B_1}\left( T_L+\sum_{k=0}^2\frac{\p S_{L+k}}{\p b_L}\right)\nonumber\\
& + \frac{\left<\Ls^Lq, \chi_{B_0}\Lambda Q\right>}{D_L^{b_1}}\left[\partial_s\chi_{B_1}\left( T_L+\sum_{k=0}^2\frac{\p S_{L+k}}{\p b_L}\right)+\chi_{b_1}\p_s\left(\sum_{k=0}^2\frac{\p S_{L+k}}{\p b_L}\right)\right], 
\end{align*}
where $\tilde{Q}_b$ is defined as in \eqref{def:Qbtil} and we know from Lemma \ref{lemm:GenLk} that $T_i$ is admissible of degree $(i,i)$ and from Proposition \ref{prop:1} that $S_j$ is homogeneous of degree $(j, j, j)$. \\
Since $|b_j|\lesssim b_1^{j+\frac{1}{2}}$ and $\Ls \Lambda Q = 0$, we use Lemma \ref{lemm:actionLL} to estimate
\begin{align*}
\int \frac{|\Ls^{L}\Lambda \tilde{Q}_b|^2}{1 + y^4} &\lesssim \sum_{i = 1}^L b_i^2\int \frac{|\Ls^{L} \Lambda \tilde{T}_i|^2}{1 + y^4} + \sum_{i = 2}^{L+2} \int \frac{|\Ls^{L} \Lambda \tilde{S}_i|^2}{1 + y^4}\\
&\quad\lesssim \sum_{i = 1}^L b_1^{2i+1}\int_{y \leq 2B_1}\frac{y^{6}dy}{1 + y^{4(L+1 - i) + 4}} + \sum_{i = 2}^{L+2}b_1^{2i+1}\int_{y \leq 2B_1} \frac{y^{6}dy}{1 + y^{4(L+1 - i) + 6}} \lesssim b_1^{\frac 72}.
\end{align*}
Using the cancellation $\Ls^{L+1}T_i = 0$ for $1 \leq i \leq L$ and the admissibility of $T_i$, we estimate
$$\sum_{i = 1}^{L-1} \int \frac{|\Ls^{L} (\chi_{B_1}T_i)|^2}{1 + y^4} \lesssim \sum_{i = 1}^{L-1} \int_{B_1 \leq y \leq 2B_1 } y^{4(i - L-1) -4 + 6}dy \lesssim b_1^{\frac{5}{2}}.$$
Using the homogeneity of $S_j$, we estimate for $1 \leq i \leq L$,
$$\sum_{i=1}^{L-1}\sum_{j = i + 1}^{L+2} \int \frac{1}{1 + y^4}\left|\Ls^{L}\left(\chi_{B_1}\frac{\partial S_j}{\partial b_i}\right)\right|^2 \lesssim \sum_{i=1}^{L-1}\sum_{j = i + 1}^{L+2}b_1^{2(j-i)+1}\int_{B_1 \leq y \leq 2B_1}y^{4(j-L)-6 - 4}y^{6}dy \lesssim b_1^{\frac 72-\frac{5\eta}{2}}.$$
The collection of the above bounds together yields the estimate
\begin{align*}
\int |\Ls^{L+1} \Lambda \tilde{Q}_b|^2 + \sum_{i = 1}^{L-1}\int |\Ls^{L+1} \tilde{T}_i|^2 + \sum_{i = 1}^{L-1}\sum_{j = i + 1}^{L+2}\int \left|\Ls^{L+1} \left(\chi_{B_1}\frac{\partial S_j}{\partial b_i}\right)\right|^2 \lesssim b_1^{\frac 52},
\end{align*}
and
\begin{align}\label{est:LkTL}
\int \left|\Ls^{L+1}\left(\chi_{B_1}\left( T_L+\sum_{k=0}^2\frac{\p S_{L+k}}{\p b_L}\right)\right)\right|^2\lesssim b_1^{\frac{1+\eta}{2}}.
\end{align}
From \eqref{est:LQ2chiB0} and \eqref{est:LsLqchiB0LQ}, we have the bound
\begin{equation}\label{est:CLxiL}
\left|\frac{\left<\Ls^Lq, \chi_{B_0}\Lambda Q\right>}{D_L^{b_1}}\right| \lesssim \sqrt{B_0}\sqrt{\Es_{2L+2}} = b_1^{-\frac{1}{4}}\sqrt{\Es_{2L+2}}.
\end{equation}
We also have 
$$\int \left|\Ls^{L+1} \left(\partial_s \left[\chi_{B_1} \left( T_L+\sum_{k=0}^2\frac{\p S_{L+k}}{\p b_L}\right)\right]\right)\right|^2 \lesssim b_1^{\frac 72}.$$
The collection of the above bounds together with Lemmas \ref{lemm:mod1} and \ref{lemm:mod2} yields
\begin{align*}
\left(\int |\Ls^{L+1} \widetilde{\text{Mod}}|^2 \right)^\frac{1}{2} &\lesssim b_1^{1+\frac \eta4}\left(\frac{\sqrt{\Es_{2L+2}}}{M^{\frac{3}{2}}} + b_1^{L+\frac{3}{4}+\frac{\eta}{4}}\right),
\end{align*}
which is the conclusion of \eqref{eq:ModtilE2k}.\\

Injecting the estimates \eqref{eq:PsibtilE2k}, \eqref{eq:ModtilE2k} into \eqref{eq:EnerID2}, we arrive at 
\begin{align}
&\frac{1}{2}\frac{d}{dt}\left\{ \frac{\Es_{2L+2}}{\lambda^{4L-3}}\Big[1 + \Oc(b_1^{\frac 12})\Big] \right\}\nonumber\\
&\qquad \leq -\frac{1}{8} \int |\tilde\Ls_\lambda v_{2L+1}|^2 + \frac{b_1}{\lambda^{4L-1}}\left(\frac{\Es_{2L+2}}{M^{\frac{3}{2}}} + b_1^{\frac 14}\Es_{2L+2} + b_1^{L+\frac{3}{4}}\sqrt{\Es_{2L+2}}\right) \nonumber\\
&\qquad  \quad + \frac{\sqrt{b_1}\sqrt{\Es_{2L+2}}}{\lambda^{4L-1}}\left[\left(\int \frac{|\As\Ls^{L}\Fc_1|^2}{1 + y^2} \right)^\frac{1}{2} + \left(\int \frac{|\Ls^{L}\Fc_1|^2}{1 + y^4} \right)^\frac{1}{2} \right]\nonumber\\
&\qquad \qquad + \frac{1}{\lambda^{4L-1}}\left\|\As \Ls^{L}\Fc_1 \right\|_{L^2}^2 + \frac{1}{\lambda^{4L-1}}\int \Ls^{L+1} q \Ls^{L+1}(\partial_s \xi_L). \label{eq:EnerID3}
\end{align}

\medskip

\noindent \textbf{Step 6: Estimates for the linear small term $\Hc(q)$.} We claim the following 
\begin{equation}\label{est:HqE2k}
\int |\As \Ls^{L}\Hc(q)|^2 + \int  \frac{|\As \Ls^{L}\Hc(q)|^2}{1 + y^2} + \frac{| \Ls^{L}\Hc(q)|^2}{1 + y^4} \lesssim b_1^2 \Es_{2L+2}. 
\end{equation}
We only deal with the estimate for the first term because the last two terms are estimated similarly. Let us rewrite from \eqref{def:Lq} the definition of $\Hc(q)$, 
$$\Hc(q) = \Phi q \quad \text{with} \quad \Phi = \frac{6}{y^2}\left[\cos(2Q) - \cos(2Q + 2\tilde{\Theta}_b)\right],$$
where 
$$\tilde{\Theta}_b = \sum_{i = 1}^Lb_i \tilde{T}_i + \sum_{i = 2}^{L+2}\tilde{S}_i(b,y).$$
From the asymptotic behavior of $Q$ given in \eqref{eq:asymQ}, the admissibility of $T_i$ and the homogeneity of $S_i$, we deduce that $\Phi$ is a regular function both at the origin and at infinity. We then apply the Leibniz rule \eqref{eq:LeibnizALk} with $k = L$ and $\phi = \Phi$ to write 
$$\As \Ls^{L} \Hc(q) = \sum_{m = 0}^{L}\Big[ q_{2m +1}\Phi_{2L+1, 2m + 1} + q_{2m}\Phi_{2L+1, 2m}\Big],$$
where $\Phi_{2L+1, i}$ with $0 \leq i \leq 2L+1$ are defined by the recurrence relation given in Lemma \ref{lemm:LeibnizLk}. In particular, we have the following estimate
$$|\Phi_{k, i}| \lesssim \frac{b_1}{1 + y^{2 + (k-i)}}\lesssim \frac{b_1}{1 + y^{1 + k - i}}, \quad \forall k \geq 1, \;\; 0 \leq i \leq k.$$
Hence, we estimate from \eqref{eq:Enercontrol},
\begin{align*}
\int |\As \Ls^{L} \Hc(q)|^2 &\lesssim \sum_{m = 0}^{L}\left[\int|q_{2m + 1} \Phi_{2L+1,2m + 1}|^2 + \int |q_{2m} \Phi_{2L+1,2m}|^2\right]\\
&\quad \lesssim b_1^2\sum_{m = 0}^{L}\left[\int \frac{|q_{2m + 1}|^2}{1 + y^{2 + 2(2L+1 - 2m - 1)}} +  \int \frac{|q_{2m}|^2}{1 + y^{2 + 2(2L+1 - 2m)}} \right]\\
&\qquad \lesssim b_1^2\sum_{m = 0}^{L}\left[\int \frac{|q_{2m + 1}|^2}{1 + y^{2 + 4(L - m )}} +  \int \frac{|q_{2m}|^2}{1 + y^{4 + 4(L - m)}} \right]\\
&\qquad \quad\lesssim b_1^2 \Es_{2L+2}.
\end{align*}
This concludes the proof of \eqref{est:HqE2k}.\\

\noindent \textbf{Step 7: Estimates for the nonlinear term $\Nc(q)$.} This is the most delicate point in the proof of \eqref{eq:Es2sLya}.  We claim the following 
\begin{equation}\label{est:NqE2k1}
\int |\As \Ls^{L}\Nc(q)|^2  \leq b_1^{2L+ \frac 72+\frac \eta2}, 
\end{equation}
\begin{equation}\label{est:NqE2k}
\int  \frac{|\As \Ls^{L}\Nc(q)|^2}{1 + y^2} + \int \frac{| \Ls^{L}\Nc(q)|^2}{1 + y^4} \leq b_1^{2L+4}, 
\end{equation}
We only deal with the proof of \eqref{est:NqE2k1} since the same proof holds for \eqref{est:NqE2k}.\\

\noindent - \textit{Control for $y < 1$.} Let us rewrite from \eqref{def:Nq} the definition of $\Nc(q)$,
$$\Nc(q) = \frac{q^2}{y} \Phi \quad \text{with} \quad \Phi = \left[-\frac{6}{y}\int_0^1 (1 -\tau)\sin(2\tilde{Q}_b + 2\tau q)d\tau\right].$$
From \eqref{eq:expandqat0} and the admissibility of $T_i$, we write
\begin{equation}\label{eq:expanq2y}
\frac{q^2}{y} = \frac{1}{y}\left(\sum_{i = 0}^{L+1} c_i T_i(y) + r_q(y)\right)^2 = \sum_{i = 0}^{L}\tilde{c}_i y^{2i + 1} + \tilde{r}_q \quad \text{for}\;\; y < 1,
\end{equation}
where 
\begin{equation*}
|\tilde c_i|\lesssim \Es_{2L+2}, \quad |\partial^j_y \tilde r_q(y)| \lesssim y^{2L - \frac 72 - j} |\ln y|^{L+1} \Es_{2L+2}, \quad 0\leq j \leq 2L+1, \;\; y < 1.
\end{equation*}
Let $\tau \in [0,1]$ and
$$v_\tau = \tilde{Q}_b + \tau q,$$
we obtain from Proposition \ref{prop:1} and the expansion \eqref{eq:expandqat0}, 
$$v_\tau = \sum_{i = 0}^{L}\hat c_i y^{2i + 1} + \hat r_q,$$
with 
$$|\hat c_i|\lesssim 1, \quad |\py^j \hat r_q| \lesssim y^{2L- \frac 72 - j}|\ln y|^{L+1}, \;\; 0 \leq j \leq 2L+1, \;\; y <1.$$
Together with the Taylor expansion of $\sin(x)$ at $x=0$, we write
\begin{equation}\label{eq:expPhiy}
\Phi(q) = \sum_{i = 0}^{L}\bar c_i y^{2i} + \bar r_q,
\end{equation}
with 
$$|\bar c_i| \lesssim 1, \quad |\py^j \bar r_q| \lesssim y^{2L - \frac{9}{2}  - j}|\ln y|^{L+1}, \quad 0 \leq j \leq 2L+1, \quad y < 1.$$
From \eqref{eq:expanq2y} and \eqref{eq:expPhiy}, we have the expansion of $\Nc$ near the origin,
$$\Nc(q) = \sum_{i = 0}^{L}\hat{\tilde{c}}_i y^{2i + 1} + \hat{\tilde{r}}_q,$$
with 
$$|\hat{\tilde{c}}_i| \lesssim \Es_{2L+2}, \quad |\py^j \hat{\tilde{r}}_q| \lesssim y^{2L - \frac 72 -j}|\ln y|^{L+1} \Es_{2L+2}, \quad 0 \leq j \leq 2L+1, \;\; y < 1.$$
From the definition of $\As$ and $\As^*$ (see \eqref{def:As} and \eqref{def:Astar}), one can check that for $y < 1$,
$$|\As \Ls^{L} \hat{\tilde{r}}_q | \lesssim \sum_{i = 0}^{2L+1} \frac{\py^i \hat{\tilde{r}}_q}{y^{2L+1 - i}} \lesssim \Es_{2L+2} \sum_{i = 0}^{2L+1}\frac{y^{2L - \frac{7}{2} - i}|\ln y|^L+1}{y^{2L+1 - i}} \lesssim y^{-\frac{5}{2}}|\ln y|^{L+1}\Es_{2L+2}.$$
Note from the asymptotic behavior \eqref{eq:asympV} of $V$ that $\As(y) = \Oc(y^2)$ for $y < 1$, which  implies 
$$\left|\As \Ls^{L} \left(\Nc(q)-\hat{\tilde{r}}_q\right)\right| \lesssim \sum_{i = 0}^{L}|\hat{\tilde{c}}_i| y^2 \lesssim y^2\Es_{2L+2}.$$
We then conclude 
$$\int_{y < 1} |\As \Ls^{L}\Nc(q)|^2 \lesssim \Es_{2L+2}^2\int_{y < 1}y|\ln y|^{2L+2} dy \lesssim \Es_{2L+2}^2 \lesssim b_1^{2L + 4}.$$

\medskip

\noindent - \textit{Control for $y \geq 1$.} Let us rewrite from the definition of $\Nc(q)$,
\begin{equation}\label{def:Zpsi}
\Nc(q) = Z^2 \psi, \quad Z = \frac{q}{y}, \quad \psi = -6\int_0^1(1 - \tau)\sin(2\tilde{Q}_b + 2\tau q) d\tau.
\end{equation}
Note from the definitions of $\As$ and $\As^*$ that 
$$\forall k \in \mathbb{N}, \quad |\As \Ls^{k} f| \lesssim \sum_{i = 0}^{2k + 1} \frac{|\py^i f|}{y^{2k + 1 - i}},$$
from which and the Leibniz rule, we write 
\begin{align*}
\int_{y \geq 1}|\As \Ls^{L}\Nc(q)|^2 &\lesssim \sum_{k = 0}^{2L+1}\int_{y \geq 1}\frac{|\py^k\Nc(q)|^2}{y^{4L+1 - 2k - 2}}\\
&\lesssim \sum_{k = 0}^{2L+1}\sum_{i = 0}^k \int_{y \geq 1} \frac{|\py^i Z^2|^2|\py^{k - i}\psi|^2}{y^{4L+1 - 2k - 2}}\\
&\lesssim \sum_{k = 0}^{2L+1}\sum_{i = 0}^k \sum_{m = 0}^i \int_{y \geq 1} \frac{|\py^m Z|^2 |\py^{i - m}Z|^2 |\py^{k - i}\psi|^2}{y^{4L+1 - 2k - 2}}.
\end{align*}

We aim at using the pointwise estimate \eqref{eq:pointwise_yg1} to prove that for $0 \leq k \leq 2L+1$, $0 \leq i \leq k$ and $0 \leq m \leq i$,
\begin{equation}\label{def:Akmi}
A_{k,i,m} := \int_{y \geq 1} \frac{|\py^m Z|^2 |\py^{i - m}Z|^2 |\py^{k - i}\psi|^2}{y^{4L+2 - 2k}} \leq b_1^{2L + \frac 72+\frac \eta2},
\end{equation}
which concludes the proof of \eqref{est:NqE2k1}.

To prove \eqref{def:Akmi}, we have to consider 3 cases:\\

\noindent - \underline{\textbf{The initial case: $k = 0$.}} Since $0 \leq m \leq i \leq k$, then $k = i = m = 0$. Although this is the simplest case, it gives us a basic idea about how to handle the other cases. From \eqref{def:Zpsi}, it is obvious to see that $|\psi|$ is uniformly bounded.  We write 
\begin{align*}
A_{0,0,0}  = \int_{y \geq 1} \frac{|q|^4|\psi|^2}{y^{4L+6}} y^{6}dy  \lesssim \int_{1 \leq y \leq B_0} \frac{|q|^4}{y^{4L}} dy + \int_{y \geq B_0} \frac{|q|^4}{y^{4L}}dy.
\end{align*}
Using \eqref{eq:pointwise_yg1}, and the bootstrap assumption, we estimate 
\begin{align*}
&\int_{1 \leq y \leq B_0} \frac{|q|^4}{y^{4L}}dy\\
& \quad \lesssim \left\|\frac{y^{5}|q|^2}{y^{2(2L+1)}}\right\|_{L^\infty(y > 1)}\left\|\frac{y^{5}|q|^2}{y^{2(2\ell + 3)}}\right\|_{L^\infty(y > 1)}\int_{1 \leq y \leq B_0}y^{4\ell-2}dy\\
& \qquad \lesssim \Es_{2L+2}\Es_{2\ell + 2} B_0^{4\ell-1}\\
& \quad \qquad \lesssim b_1^{2L + \frac 72 +\frac \eta2},
\end{align*}
for all $1\leq\ell\leq L$.\\
For the integral on the domain $y \geq B_0$, let us write
\begin{align*}
&\int_{y \geq B_0} \frac{|q|^4}{y^{4L}}dy\\
&\quad \lesssim \left\|\frac{y^{5}|q|^2}{y^{2(2L+1  - 2\ell)}}\right\|_{L^\infty(y > 1)}\left\|\frac{y^{5}|q|^2}{y^{2(2\ell +1)}}\right\|_{L^\infty(y > 1)}\int_{y \geq B_0} \frac{dy}{y^{6}}\\
& \qquad \lesssim \Es_{2L+2 - 2\ell}\Es_{2\ell +2} B_0^{-5}\lesssim b_1^{2L + \frac{11}{2}}.
\end{align*}
This concludes the proof of \eqref{def:Akmi} when $k = i = m = 0$. \\

\noindent  - \underline{\textbf{Case II: $k \geq 1$ and $k = i$}}. We first use the Leibniz rule to write
\begin{equation}\label{eq:pylZexp}
\forall l \in \mathbb{N}, \quad |\py^l Z|^2 \lesssim \sum_{j = 0}^l \frac{|\py^j q|^2}{y^{2 + 2l - 2j}},
\end{equation}
from which, 
\begin{align*}
A_{k,k,m} &\lesssim \sum_{j = 0}^m \sum_{l = 0}^{k-m} \int_{y \geq 1} \frac{|\py^j q|^2 |\py^l q|^2}{y^{4 L - 2j - 2l + 6}}y^{6}dy.
\end{align*}
We claim that for all $(j,l) \in \mathbb{N}^2$ and $1 \leq j + l \leq 2L+1$,
\begin{equation}\label{est:Bjl}
B_{j,l,0}:= \int_{y \geq 1} \frac{|\py^j q|^2 |\py^l q|^2}{y^{4 L - 2j - 2l + 6}}y^{6}dy \lesssim b_1^{2L +\frac 72},
\end{equation}
which immediately follows \eqref{def:Akmi} for the case when $k = i$. 

To prove \eqref{est:Bjl}, we proceed as for the case $k = 0$ by splitting the integral in two parts as follows:
\begin{align*}
B_{j,l,0} &= \int_{1 \leq y \leq B_0} \frac{\big(y^{5}|\py^j q|^2\big)\big(y^{5}|\py^l q|^2\big)}{y^{4L - 2j - 2l + 10}}dy\\
&\quad  + \int_{y \geq B_0} \frac{\big(y^{5}|\py^j q|^2\big)\big(y^{5}|\py^l q|^2\big)}{y^{4L+4 - 2j - 2l}}\frac{dy}{y^{6}}\\
&\qquad \lesssim \left\|\frac{\big(y^{5}|\py^j q|^2\big)\big(y^{5}|\py^l q|^2\big)}{y^{4L - 2j - 2l + 10}} \right\|_{L^\infty(y \geq 1)}b_1^{-\frac 12}\\
&\qquad \quad + \left\|\frac{\big(y^{5}|\py^j q|^2\big)\big(y^{5}|\py^l q|^2\big)}{y^{4L+4 - 2j - 2l}} \right\|_{L^\infty(y \geq 1)}b_1^{\frac{5}{2}}\\
& \qquad \qquad = \left\|\frac{\big(y^{5}|\py^j q|^2\big)\big(y^{5}|\py^l q|^2\big)}{y^{2J_1 - 2j + 2J_2 - 2l}} \right\|_{L^\infty(y \geq 1)}b_1^{-\frac{1}{2}}\\
&\qquad \qquad \quad + \left\|\frac{\big(y^{5}|\py^j q|^2\big)\big(y^{5}|\py^l q|^2\big)}{y^{2J_3 - 2j + 2J_4 - 2l}} \right\|_{L^\infty(y \geq 1)}b_1^{\frac{5}{2}}\\
& \qquad \qquad \qquad := B_{j,l,0,J_1, J_2} b_1^{-\frac{1}{2}} + B_{j,l,0, J_3, J_4}b_1^{\frac{5}{2}},
\end{align*}
where $J_n (n = 1,2,3,4)$ satisfy
$$J_1 + J_2 = 2L+5, \quad J_3 + J_4 = 2L+2.$$
We now estimate $B_{j,l,0,J_1, J_2}$.\\
- If $l$ is even, we take 
$$J_2 = \left\{ \begin{array}{ll}
l + 2 &\quad \text{if}\quad l \leq 2L-2,\\
l &\quad \text{if}\quad l = 2L.
\end{array} \right.$$
This gives
$$4 \leq J_2 \leq 2L, \quad 5 \leq J_1 = 2L+5 - J_2 \leq 2L+1.$$
Using \eqref{eq:pointwise_yg1}, we have the estimate
\begin{align*}
 B_{j,l,0,J_1, J_2}& \lesssim \left\|\frac{y^{5}|\py^j q|^2}{y^{2J_1 - 2j}}\right\|_{L^\infty(y \geq 1)} \left\|\frac{y^{5}|\py^l q|^2}{y^{2J_2 - 2l}} \right\|_{L^\infty(y \geq 1)}\\
 &\quad  \lesssim \Es_{J_1 + 1} \sqrt{\Es_{J_2} \Es_{J_2 + 2}}.
\end{align*}
- If $l$ is odd, we simply take $J_2 = l + 1$ which gives
 $$4 \leq J_2 \leq 2L, \quad 5 \leq J_1 \leq 2L+1.$$
Hence,
\begin{align*}
 B_{j,l,0,J_1, J_2}&\lesssim \Es_{J_1 + 1} \sqrt{\Es_{J_2} \Es_{J_2 + 2}}.
\end{align*}
Recall from the bootstrap assumption that for all even integer $m$ in the range $4 \leq m \leq 2L+2$,
\begin{equation}\label{est:Esmeven}
\Es_{m} \lesssim b_1^{m-1+(\frac 12+2m)\frac{\eta}{4}}.
\end{equation}
Hence, we obtain 
$$B_{j,l,0,J_1, J_2} \lesssim b_1^{2L+\frac 72+\frac \eta2}.$$
Similarly, one prove that 
$$B_{j,l,0,J_3, J_4} \lesssim b_1^{2L+\frac 72+\frac \eta2} \quad \text{for} \quad J_3 + J_4 = 2L+ 2.$$
Therefore, 
\begin{align*}
B_{j,l,0} & \lesssim b_1^{2L + \frac 72+\frac \eta2}.
\end{align*}
This concludes the proof of \eqref{est:Bjl} as well as \eqref{def:Akmi} when $k = i$. \\

\noindent  - \underline{\textbf{Case III: $k \geq 1$ and $k - i \geq 1$.}} Let us write from \eqref{def:Akmi} and \eqref{eq:pylZexp}, 
\begin{equation}\label{eq:Akmi11}
A_{k,m,i} \lesssim \sum_{j = 0}^{m}\sum_{l = 0}^{i - m} \int_{y \geq 1} \frac{|\py^j q|^2 |\py^l q|^2}{y^{4L+6 - 2j - 2l}} \frac{|\py^{k - i} \psi|^2}{y^{-2(k - i)}}.
\end{equation}
At this stage, we need to precise the decay of $|\py^n \psi|$ to a    chive the bound \eqref{def:Akmi}. To do so, let us recall that $T_i$ is admissible of degree $(i,i)$ (see Lemma \ref{lemm:GenLk}) and $S_i$ is homogeneous of degree $(i,i,i)$ (see Proposition \ref{prop:1}). Together with \eqref{eq:asymQ}, we estimate for all $j\geq1$ and $y\geq 1$
$$|\py^j \tilde{Q}_b | \lesssim \frac{1}{y^{2 + j}} + \sum_{l = 1}^{L}b_1^{l+\frac{1}{2}} y^{2l-2-j} \mathbf{1}_{\{y \leq 2B_1\}} \lesssim \frac{1}{y^{2+j}}\mathbf{1}_{\{y \leq B_0\}}+\frac{b_1^{\frac{j+1}{2}+\frac{3\eta}{4}}}{y^{\frac{1}{2}}}\mathbf{1}_{\{B_0\leq y \leq 2B_1\}}.$$
Let $\tau \in [0,1]$ and $v_\tau = \tilde{Q}_b + \tau q$. We use the Faa di Bruno formula to write\\ 
\begin{align*}
\forall 1\leq n \leq 2L+1, \quad |\py^n \psi|^2 &\lesssim \int_0^1 \sum_{m^* = n} |\partial_{v_\tau}^{m_1 + \cdots+m_n}\sin(v_\tau)|^2  \prod_{i = 1}^n |\py^i \tilde{Q}_b + \py^iq|^{2m_i} d\tau \nonumber\\
&\quad \lesssim \sum_{m^* = n}  \prod_{i = 1}^n \left(|\py^i \tilde{Q}_b|^2 +  |\py^iq|^2\right)^{m_i}, \quad m^* = \sum_{i = 1}^n im_i.
\end{align*}
We split the estimation of the nonlinear term when $y\geq1$ into 2 steps, $y\leq B_0$ and $y\geq B_0$.
For $1 \leq y \leq B_0$, we use \eqref{eq:pointwise_yg1} and the bootstrap assumption \eqref{bootstrapboundE2L2} to estimate 
\begin{equation*}
|\py^iq|^2 = y^{4L+2 - 2i}\left|\frac{\py^i q}{y^{2L+1 - i}}\right|^2\lesssim y^{4L-3 - 2i}\Es_{2L+2}\lesssim \frac{b_1^{1+\frac \eta2}}{y^{4+2i}},
\end{equation*}
 from which, we have
\begin{equation}\label{est:psin1B0}
|\py^n \psi|^2 \lesssim \sum_{m^* = n}  \prod_{i = 1}^n \frac{1}{y^{2im_i+4m_i}} \lesssim \frac{1}{y^{2n+4}}, \quad \forall 1 \leq y \leq B_0.
\end{equation}
For $y \geq B_0$, by using again \eqref{eq:pointwise_yg1} and the bootstrap assumption \eqref{bootstrapboundE2k} we deduce that
$$|\py^i q|^2\lesssim y^4\Big|\frac{\py^i q}{y^2}\Big|^2\lesssim \frac{b_1^{i+2+\frac{5\eta}{2}}}{y}.$$
Hence, for all $y\geq B_0$ we get
\begin{align}\label{est:psinB0g}
|\py^n \psi|^2 &\lesssim \sum_{m^* = n}  \prod_{i = 1}^n \left(|\py^i \tilde{Q}_b|^2 +  |\py^iq|^2\right)^{m_i}, \quad m^* = \sum_{i = 1}^n im_i\nonumber\\
&\lesssim \sum_{m^* = n}  \prod_{i = 1}^{n}\left(\frac{b_1^{i+1+\frac{3\eta}{4}}}{y} +  \frac{b_1^{i+2+\frac {5\eta}{2}}}{y}\right)^{m_i}\lesssim \frac{b_1^{n+1}}{y}.
\end{align}
Injecting \eqref{est:psin1B0}, and \eqref{est:psinB0g} into \eqref{eq:Akmi11}, we obtain that for all $1\leq k-i\leq 2L+1$
\begin{align*}
A_{k,i,m} &\lesssim \sum_{j = 0}^{m}\sum_{l = 0}^{i - m} \left(\int_{1 \leq y \leq B_0} \frac{|\py^j q|^2 |\py^l q|^2}{y^{4L+4 - 2j - 2l + 6}} +b_1^{k-i+1} \int_{ y \geq B_0} \frac{|\py^j q|^2 |\py^l q|^2}{y^{4L+1 - 2j - 2l + 6 -2(k-i)}}\right).
\end{align*}
Arguing as for the proof of \eqref{est:Bjl}, we end up with 
\begin{align*}
A_{k,i,m} &\lesssim b_1^{2L+\frac{7}{2}+\frac{\eta}{2}}.
\end{align*}
This finishes the proof of \eqref{def:Akmi} as well as \eqref{est:NqE2k1}. Since the proof of \eqref{est:NqE2k} follows exactly the same lines as for the proof of \eqref{est:NqE2k1}, we omit its proof here. 

\medskip

Inserting \eqref{est:HqE2k}, \eqref{est:NqE2k1} and \eqref{est:NqE2k} into \eqref{eq:EnerID3} and recalling from \eqref{bootstrapboundE2L2} that $\Es_{2L+2} \leq K b_1^{2L+\frac 32+\frac{\eta}{10}}$, we arrive at
\begin{align}
\frac{1}{2}\frac{d}{dt}\left\{ \frac{\Es_{2L+2}}{\lambda^{4L-3}}\Big[1 + \Oc(b_1)\Big] \right\}
&\lesssim \frac{b_1}{\lambda^{4L-1}}\left(\frac{\Es_{2L+2}}{M^{\frac{3}{2}}} + b_1^{L+\frac{3}{4}+\frac{\eta}{4}}\sqrt{\Es_{2L+2}} + b_1^{2L+\frac{3}{2}+\frac{\eta}{2}}\right)\nonumber\\
&\qquad + \frac{1}{\lambda^{4L-1}}\int \Ls^{L+1} q \Ls^{L+1}(\partial_s \xi_L). \label{eq:EnerID4}
\end{align}

\noindent \textbf{Step 8: Time oscillations.} In this step, we want to find the contribution of the last term in \eqref{eq:EnerID4} to the estimate \eqref{eq:Es2sLya}. Let us write
\begin{align}
\frac{1}{\lambda^{4L-1}}\int \Ls^{L+1} q \Ls^{L+1}(\partial_s \xi_L) &= \frac{d}{ds} \left\{\frac{1}{\lambda^{4L-1}}\left[\int \Ls^{L+1} q \Ls^{L+1} \xi_L - \frac{1}{2}\int |\Ls^{L+1} \xi_L|^2 \right]\right \}\nonumber\\
&\quad +\frac{4L-1}{\lambda^{4L-1}} \frac{\lambda_s}{\lambda} \left[\int \Ls^{L+1} q \Ls^{L+1} \xi_L + \frac{1}{2}\int |\Ls^{L+1} \xi_L|^2\right]\nonumber\\
&\qquad  - \frac{1}{\lambda^{4L-1}}\int \Ls^{L+1}(\ps q - \ps \xi_L)  \Ls^{L+1} \xi_L.\label{id:Aextract}
\end{align}
From \eqref{est:LkTL} and \eqref{est:CLxiL}, we have
\begin{equation}\label{est:LkxiL2}
\int|\Ls^{L+1} \xi_L|^2 \lesssim b_1^{\frac{\eta}{2}}\Es_{2L+2}.
\end{equation} 
This follows
\begin{align*}
\left|\int \Ls^{L+1} q \Ls^{L+1} \xi_L \right| &\lesssim \left(\int |\Ls^{L+1} q|^2\right)^\frac{1}{2}\left(\int |\Ls^{L+1} \xi_L|^2\right)^\frac{1}{2}\\
 &\quad \lesssim b_1^{\frac{\eta}{4}}\Es_{2L+2}.
\end{align*}
Since $\frac{dt}{ds} = \lambda^2$, we then write 
\begin{equation}\label{est:A1}
\frac{d}{ds} \left\{\frac{1}{\lambda^{4L-1}}\Big[\int \Ls^{L+1} q \Ls^{L+1} \xi_L - \frac{1}{2}\int |\Ls^{L+1} \xi_L|^2\Big]\right\} = \frac{d}{dt}\left(\frac{\Es_{2L+2}}{\lambda^{4L-3}}\mathcal{O}(b_1^{\frac{\eta}{4}}) \right).
\end{equation}
Noting from \eqref{eq:ODEbkl} that $\left|\frac{\lambda_s}{\lambda}\right| \lesssim b_1$, this gives
\begin{equation}\label{est:A2}
\left|\frac{\lambda_s}{\lambda} \left[\int \Ls^{L+1} q \Ls^{L+1} \xi_L + \frac{1}{2}\int |\Ls^{L+1} \xi_L|^2\right]\right| \lesssim b_1 b_1^{\frac{\eta}{4}}\Es_{2L+2}.
\end{equation}
For the last term in \eqref{id:Aextract}, we use equation \eqref{eq:qys} and the decomposition \eqref{eq:decomF} to write
\begin{align}
\int \Ls^{L+1}(\ps q - \ps \xi_L)  \Ls^{L+1} \xi_L &= \left[-\int \Ls^{L+1} q \Ls^{L+2}\xi_L + \frac{\lambda_s}{\lambda}\int \Lambda q \Ls^{2L+2}\xi_L\right] \nonumber\\
&+ \int \Ls^{L+1}\big[-\tilde \Psi_b - \widetilde{Mod} + \Hc(q) + \Nc(q)\big] \Ls^{L+1}\xi_L.\label{eq:A3exp}
\end{align}
Using \eqref{est:CLxiL}, the admissibility of $T_L$ and $S_j$ and the fact that $\Ls^k T_i = 0$ if $i < k$, we estimate 
\begin{align*}
\int |\Ls^{L+2}\xi_L|^2 & \lesssim \left|\frac{\left<\Ls^Lq, \chi_{B_0}\Lambda Q\right>}{D_L^{b_1}} \right|^2\int\left|\Ls^{L+2}\left[(1 - \chi_{B_1})T_L+\chi_{B_1}\sum_{k=0}^2\frac{\p S_{L+k}}{\p b_L}\right]\right| ^2\\
&\lesssim b_1^{- \frac{1}{2}}\Es_{2L+2}\int_{y \geq B_1}y^{2(2L - 2 - 2(L+2))} y^{6}dy\\
&\quad \lesssim b_1^{-\frac{1}{2}}\Es_{2L+2}b_1^{\frac 52(1+\eta)} \lesssim b_1^{2+\frac{5\eta}{2}}\Es_{2L+2}.
\end{align*}
from which we obtain
\begin{equation*}
\left|\int \Ls^{L+1} q \Ls^{L+2}\xi_L\right| \lesssim b_1 b_1^{\frac{5\eta}{4}}\Es_{2L+2},
\end{equation*}
Similarly, we have the estimate
\begin{align*}
\int (1 + y^{4(L+1)})|\Ls^{2L+2}\xi_L|^2 &\lesssim b_1^{\frac{\eta}{2}}\Es_{2L+2}, 
\end{align*}
hence, using \eqref{est:Es2K1} and \eqref{eq:ODEbkl}, we get
\begin{equation*}
\left|\frac{\lambda_s}{\lambda}\int \Lambda q \Ls^{2L+2}\xi_L \right| \lesssim b_1 \left(\int \frac{|\py q|^2}{1 + y^{4L+2}} \right)^\frac{1}{2}\left(\int (1 + y^{4(L+1)})|\Ls^{2L+2}\xi_L|^2  \right)^\frac{1}{2} \lesssim b_1 b_1^{\frac{\eta}{4}}\Es_{2L+2}.
\end{equation*}
From \eqref{est:LkxiL2}, \eqref{eq:PsibtilE2k} and \eqref{eq:ModtilE2k}, we have
\begin{align*}
\left|\int \Ls^{L+1} (\tilde{\Psi}_b + \widetilde{Mod}) \Ls^{L+1} \xi_L\right| &\lesssim \left(\int |\Ls^{L+1} \xi_L|^2 \right)^\frac{1}{2}\left(\int|\Ls^{L+1} (\tilde \Psi_b + \widetilde{Mod})|^2 \right)^\frac{1}{2}\\
&\quad \lesssim b_1b_1^{\frac{1+\eta}{4}}\Es_{2L+2} + b_1^{\frac{5}{4}} b_1^{L+1 + \frac{3}{4}+\frac{\eta}{4}}\sqrt{\Es_{2L+2}}.
\end{align*}
In the same manner, we have the estimate
\begin{align*}
\int (1 + y^{4})|\Ls^{L+2}\xi_L|^2 &\lesssim b_1^{-\frac{1}{2}}\Es_{2L+2}\int_{y \geq B_1}y^{4}y^{2(2L - \gamma - 2(L+2))}y^{6}dy\lesssim b_1^{\frac{\eta}{2}}\Es_{2L+2}, 
\end{align*}
from which together with \eqref{est:HqE2k} and \eqref{est:NqE2k}, we get the bound
\begin{align*}
\left|\int \Ls^{L} (\Hc(q) + \Nc(q)) \Ls^{L+2}\xi_L \right|& \lesssim \left(\int \frac{|\Ls^{L}(\Hc(q) + \Nc(q))|^2}{1 + y^4} \right)^\frac{1}{2}\left(\int (1 + y^4)|\Ls^{L+2}\xi_L|^2\right)^\frac{1}{2}\\
&\quad \lesssim b_1 b_1^{\frac{2}{8}}\Es_{2L+2} + b_1 b_1^{L + \frac{3}{4}+\frac{1}{8}}\sqrt{\Es_{2L+2}}.
\end{align*}
Collecting these final bounds into \eqref{eq:A3exp} yields 
\begin{equation}\label{est:A3}
\left|\int \Ls^{L+1}(\ps q - \ps \xi_L)  \Ls^{L+1} \xi_L\right| \lesssim b_1 b_1^{\frac{\eta}{4}}\Es_{2L+2} + b_1^{1+\frac{\eta}{4}} b_1^{L + \frac{3}{4}+\frac{\eta}{4}}\sqrt{\Es_{2L+2}}.
\end{equation}

\medskip

\noindent Substituting \eqref{id:Aextract}, \eqref{est:A1}, \eqref{est:A2} and \eqref{est:A3} into \eqref{eq:EnerID4} concludes the proof of \eqref{eq:Es2sLya} as well as Proposition \ref{prop:E2k}.
\end{proof}

\subsection{Proof of Proposition \ref{proposition:Bootstrap}.}
We give the proof of Proposition \ref{proposition:Bootstrap} in this subsection in order to complete the proof of Theorem \ref{Theo:1}. 
Before going to the proof of \eqref{bootstrap_b}, \eqref{bootstrapboundE2k}, and \eqref{bootstrapboundE2L2} let us compute explicitly the scaling parameter $\lambda$. To do so, let us note from the bootstrap assumption on $b_2\lesssim b^{\frac{5}{2}+\frac{\eta}{10}}$ that
$$(b_1)_s+C_{b_1}b_1^2=O(b^{\frac{5}{2}+\frac{\eta}{10}}),$$
with $C_{b_1}=\frac{6a_0}{a_1}\sqrt{b_1}+O(b_1)$ where $a_0>0$ and $a_1>0$ are coming from the asymptotic behaviour of $Q$ \eqref{eq:asymQ}.
Hence, we deduce that
\begin{align}\label{b1est}
b_1=\frac{C}{s^{\frac{2}{3}}}+O\Big(\frac{1}{s^{\frac{2}{3}+\frac{2\eta}{30}}}\Big).
\end{align}
Using \eqref{b1est} and \eqref{eq:ODEbkl} yields
\begin{equation}\label{eq:lam10}
-\frac{\lambda_s}{\lambda} = \frac{C}{s^{\frac{2}{3}}} + \Oc\left(\frac{1}{s^{\frac{2}{3}+\frac{2\eta}{30}}}\right),
\end{equation}
from which we get,
\begin{equation}\label{eq:Lamdas}
\lambda(s) =C(s_0)e^{-Cs^{\frac{1}{3}}+O(s^{\frac{1}{3}-\frac{2\eta}{30}})}.
\end{equation}
We start first by closing  the booststrap bound \eqref{bootstrap_b}.\\

\noindent - \textit{Control of the modes $b_k$'s.} We close the control of the  modes $(b_{1}, \cdots, b_L)$. We first treat the case when $k = L$. Let 
$$\tilde{b}_L = b_L + \frac{\left<\Ls^Lq,\chi_{B_0}\Lambda Q \right>}{D_L^{b_1}},$$
then from \eqref{est:CLxiL} and \eqref{bootstrapboundE2L2}, 
$$|\tilde b_L - b_L| \lesssim b_1^{-\frac{1}{4}}\sqrt{\Es_{2L+2}} \lesssim b_1^{L+\frac{1}{2}+\frac{\eta}{4}},$$
and hence from the improved modulation equation \eqref{eq:ODEbLimproved}, 
\begin{align*}
|(\tilde{b}_L)_s + (2L - 2+C_{b_1})b_1\tilde{b}_L|&\lesssim b_1|\tilde b_L - b_L| + \frac{1}{B_0^{\frac{3}{2}}}\left[C(M)\sqrt{\Es_{2L+2}} + b_1^{L+ \frac{3}{4}+\frac{\eta}{4}}\right]\\
&\quad \lesssim b_1^{L + 1 +\frac{1}{2}+\frac{\eta}{4}}.
\end{align*}
This follows
$$\left|\frac{d}{ds} \left\{\frac{\tilde b_L}{\lambda^{2L - 2+C_{b_1}}} \right\}\right| \lesssim \frac{b_1^{L + 1 + \frac{1}{2}+\frac{\eta}{4}}}{\lambda^{2L - 2+C_{b_1}}}.$$
Integrating this identity in time from $s_0$ and recalling that $\lambda(s_0) = 1$ yields
\begin{align*}
\tilde{b}_L(s) \lesssim C\lambda(s)^{2L - 2 + C_{b_1}}\left(\tilde{b}_L(s_0) + \underbrace{\int_{s_0}^s\frac{b_1(\tau)^{L + 1 +\frac{1}{2}+\frac{\eta}{4}}}{\lambda(\tau)^{2L - 2 +C_{b_1}}} d\tau}_{:=F(s,s_0)}\right).
\end{align*}
Using \eqref{est:CLxiL}, $b_1(s) \sim \frac{C}{s^{\frac{2}{3}}}$, the initial bounds \eqref{eq:initbk} and \eqref{eq:intialbounE2m} together with \eqref{eq:Lamdas} , we estimate for $s_0$ large enough
\begin{align}\label{estFss_0}
F(s,s_0)=\int_{s_0}^s\frac{b_1(\tau)^{L + 1 +\frac{1}{2}+\frac{\eta}{4}}}{\lambda(\tau)^{2L - 2 +C_{b_1}}} d\tau=C(s_0)+\frac{C_Fb_1^{L+\frac{1}{2}+\frac{\eta}{4}}}{\lambda(s)^{2L - 2 +C_{b_1}}}+O\Big(\frac{b_1^{L+1+\frac{\eta}{4}}}{\lambda(s)^{2L - 2 +C_{b_1}}}\Big),
\end{align}
where $C_F>0$.
It follows,
$$b_L(s) \lesssim |\tilde{b}_L(s)| + |\tilde{b}_L(s) - b_L(s)| \lesssim b_1^{L+\frac{1}{2}+\frac{\eta}{4}},$$
which concludes the proof of \eqref{bootstrap_b} for $k = L$.
Now we will propagate this improvement that we found for the bound of $b_L$ to all $b_k$ for all $2 \leq k \leq L - 1$. To do so we do a descending induction where the initialization is for $k=L$. Let assume  the bound 
$$|b_k|\lesssim  b_1^{k+\frac{1}{2} + \frac{\eta}{4}},$$
for $k+1$ and let's prove it for $k$. 
Indeed, from \eqref{eq:ODEbkl} and the induction bound, we have 
$$\left|(b_k)_s - (2k - 2+C_{b_1})\frac{\lambda_s}{\lambda}b_k\right| \lesssim b_1^{L + 1} + |b_{k + 1}| \lesssim b_1^{k + 1+\frac 12 +\frac{\eta}{4}},$$
which follows
$$\left|\frac{d}{ds}\left\{\frac{b_k}{\lambda^{2k - 2+C_{b_1}}}\right\} \right| \lesssim \frac{b_1^{k + 1 +\frac 12 +\frac{\eta}{4}}}{\lambda^{2k - 2+C_{b_1}}}.$$
Integrating this identity in time as for the case $k = L$, we end-up with 
\begin{align*}
b_k(s) &\lesssim C\lambda(s)^{2k - 2+C_{b_1}}\left(b_k(s_0) + \int_{s_0}^s\frac{b_1(\tau)^{k + 1 +\frac 12 +\frac{\eta}{4}}}{\lambda(\tau)^{2k - 2+C_{b_1}}} d\tau\right)\\
&\quad \lesssim b_1^{k+\frac 12 +\frac{\eta}{4}},
\end{align*}
where we used the initial bound \eqref{eq:initbk} and \eqref{estFss_0}. This concludes the proof of \eqref{bootstrap_b}.\\
\noindent - \textit{Improved control of $\Es_{2L+2}$}: We aim at using \eqref{eq:Es2sLya} to derive the improved bound on \eqref{bootstrapboundE2L2}. To do so, we inject the bootstrap bound of $\Es_{2L+2}$ into the monotonicity formula \eqref{eq:Es2sLya} and integrate in time by using $\lambda(s_0) = 1$: for all $s \in [s_0, s_1)$,
\begin{align*}
\Es_{2L+2}(s) &\leq C \lambda(s)^{4L-3}\left[\Es_{2L+2}(s_0) + \left(\frac{K}{M^{\frac{3}{2}}} + \sqrt K +  1\right)\int_{s_0}^s\frac{b_1^{2L+1+\frac 32+\frac{\eta}{10}}}{\lambda(\tau)^{4L-3}}d\tau\right].
\end{align*}
Using \eqref{estFss_0} and the initial bound \eqref{eq:intialbounE2m} , we get
$$\Es_{2L+2}(s) \leq C\left(\frac{K}{M^{\frac{3}{2}}} + \sqrt K + 1\right)b_1^{2L + \frac{3}{2}+\frac{\eta}{10}} \leq \frac{K_1}{2}b_1^{2L + \frac{3}{2}+\frac{\eta}{10}}, $$
for $K_1 = K(M)$ large enough. This concludes the proof of \eqref{bootstrapboundE2L2}.\\

\noindent - \textit{Improved control of $\Es_{2m}$.} We can improve the control of $\Es_{2m}$ by using the monotonicity formula \eqref{eq:Es2sLyam}.From the boostrap bound of $\Es_{2m}$ and the fact that $b_1(s) \sim \frac{1}{s^{\frac 23}}$, we integrate \eqref{eq:Es2sLyam} in time $s$ by using $\lambda(s_0) = 1$ to find that
\begin{align*}
\Es_{2m}(s) &\leq C \lambda(s)^{4m - 7}\left[\Es_{2m}(s_0)  + (1+\sqrt{K})\underbrace{\int_{s_0}^s \frac{ b_1^{2(m-1)+\frac{\eta}{10}+1}}{\lambda(\tau)^{4m - 7}}d\tau}_{:=G(s,s_0)}\right]
\end{align*}
Using the initial bound \eqref{eq:intialbounE2m} and \eqref{b1est} we obtain
$$G(s,s_0)\lesssim b_1^{2(m-1)+\frac{\eta}{10}},$$
which implies,
$$\Es_{2m}(s) \leq C(1 + \sqrt{K})b_1^{2(m-1)+\frac{\eta}{10}}\leq \frac{K}{2}b_1^{2(m-1)+\frac{\eta}{10}},$$
for $K$ large, which closes the bootstrap bound \eqref{bootstrapboundE2k} for $ 2 \leq m \leq L$.

\appendix
\section{Coercivity of the adapted norms.}
We give in this section the coercivity estimates for the operator $\Ls$ as well as the iterates of $\Ls$  under some suitable orthogonality condition. We first recall the standard Hardy type inequalities for the class of radially symmetric functions, 
$$\Dc_{rad} = \{f \in \Cc_c^\infty(\Rd) \; \text{with radial symmetry}\}.$$
For simplicity, we write
$$\int f := \int_0^{+\infty} f(y)y^{d-1}dy.$$
and
$$D^k = \left\{\begin{array}{ll} \Delta^m &\quad \text{if}\;\; k = 2m,\\
 \partial_y \Delta^m &\quad \text{if} \;\; k= 2m + 1.
\end{array}\right.$$

We have the following:
\begin{lemma}[Hardy type inequalities]\label{lemm:Hardy} Let $d \geq 7$ and $f \in \Dc_{rad}$, then\\
$(i)$ (Hardy near the origin)
\begin{equation*}
\int_{0}^1 \frac{|\py f|^2}{y^{2i}} \geq \frac{(d - 2 - 2i)^2}{4}\int_{0}^1 \frac{f^2}{y^{2+2i}} - C(d)f^2(1), \quad i = 0, 1, 2.
\end{equation*}
$(ii)$ (Hardy away from the origin for the non-critical exponent) Let $\alpha > 0, \alpha \ne \frac{d-2}{2}$, then
\begin{align}
&\int_1^{+\infty} \frac{|\py f|^2}{y^{2\alpha}} \geq \left(\frac{d - (2\alpha + 2)}{2}\right)^2\int_1^{+\infty} \frac{f^2}{y^{2 + 2\alpha}} - C(\alpha,d)f^2(1), \label{eq:HarA2} \\
&\int_1^{+\infty} \frac{|\py f|^2}{y^{2\alpha}} \geq \left(\frac{d - (2\alpha + 2)}{2}\right)^2\left\|\frac {f}{y^{\alpha+ 1 - d/2}} \right\|^2_{L^\infty(y \geq 1)}- C(\alpha,d)f^2(1). \label{eq:HarA3}
\end{align}
$(iii)$ (Hardy away from the origin for the critical exponent) Let $\alpha = \frac{d-2}{2}$, then
$$
\int_1^{+\infty} \frac{|\py f|^2}{y^{2\alpha}} \geq \frac 14 \int_1^{+\infty} \frac{f^2}{y^{2 + 2\alpha} (1 + \log y)^2} - C(d)f^2(1).$$
$(iv)$ (General weighted Hardy) For any $\mu > 0$, $k \geq 2$ be an integer and $1 \leq j \leq k-1$,
$$\int \frac{|D^jf|^2}{1 + y^{\mu + 2(k - j)}} \lesssim_{j,\mu} \int \frac{|D^k f|^2}{1 + y^\mu} + \int \frac{f^2}{1 + y^{\mu + 2k}}.$$
\end{lemma}
\begin{proof} See Lemma B.1 in \cite{MRRcmj2016}.
\end{proof}

From the Hardy type inequalities, we derive the following coercivity of $\As^*$: 
\begin{lemma}[Weight coercivity of $\As^*$]\label{lemm:coerAst} Let $\alpha \geq 0$, there exists $c_\alpha > 0$ such that for all $f \in \Dc_{rad}$ with 
$$i = 0, 1,2, \quad \int \frac{|\py f|^2}{y^{2i}(1 + y^{2\alpha})} + \int \frac{f^2}{y^{2i + 2}(1 + y^{2\alpha})} < +\infty,
$$
then 
\begin{equation}\label{eq:coerAst}
i = 0, 1,2, \quad \int \frac{|\As^* f|^2}{y^{2i}(1 + y^{2\alpha})} \geq c_\alpha\left(\int \frac{|\py f|^2}{y^{2i}(1 + y^{2\alpha})} + \int \frac{f^2}{y^{2i + 2}(1 + y^{2\alpha})}\right).
\end{equation}
\end{lemma}
\begin{proof} 
See Lemma A.2 in \cite{IGN16}.
\end{proof}

We also need the following subcoercivity of $\As$. 
\begin{lemma}[Weight coercivity of $\As$] \label{lemm:subcoerA} Let $p \geq 0$ and $i = 0, 1, 2$ such that $|2p + 2i - 1| \ne 0$. 
For all $f \in \mathcal{D}_{rad}$ with 
$$\int \frac{|\py f|^2}{y^{2i}(1 + y^{2p})}  + \int \frac{f^2}{y^{2i + 2}(1 + y^{2p})} < +\infty,$$
we have
\begin{align}
\int \frac{|\As f|^2}{y^{2i}(1 + y^{2p})} \gtrsim \int \frac{|\py f|^2}{y^{2i}(1 + y^{2p})} & + \int \frac{f^2}{y^{2i + 2}(1 + y^{2p})} \nonumber\\
&\quad - \left[f^2(1) + \int \frac{f^2}{1 +y^{2i + 2p + 4}} \right].\label{eq:subcoerA}
\end{align}
Assume in addition that 
\begin{equation}\label{cond:A3}
\left<f, \Phi_M\right> = 0 \quad \text{if}\quad 2i + 2p > 1,
\end{equation}
where $\Phi_M$ is defined in \eqref{def:PhiM}, we have 
\begin{equation}\label{eq:coerA}
\int \frac{|\As f|^2}{y^{2i}(1 + y^{2p})} \gtrsim \int \frac{|\py f|^2}{y^{2i}(1 + y^{2p})} + \int \frac{f^2}{y^{2i + 2}(1 + y^{2p})}.
\end{equation}
\end{lemma}
\begin{proof} 
See Lemma A.3 in \cite{IGN16}.
\end{proof}

From the coercivities of $\As$ and $\As^*$, we claim the following coercivity for $\Ls$:
\begin{lemma}[Weighted coercivity of $\Ls$ under a suitable orthogonality condition] \label{lemm:coerL}  Let $k \in \mathbb{N}$, $i = 0, 1, 2$ and $M = M(k)$ large enough, then there exists $c_{M,k} > 0$ such that for all $f \in \Dc_{rad}$ satisfying
$$ \int \frac{|\As f|^2}{y^{2i}(1 + y^{2k + 2})} + \int \frac{|f|^2}{y^{2i + 2}(1 + y^{2k + 2})} < +\infty,$$
and the orthogonality 
$$\left<f, \Phi_M\right> = 0 \quad \text{if}\;\; 2i + 2k > -1,$$ 
where $\Phi_M$ is defined by \eqref{def:PhiM}, there holds:
\begin{equation}\label{eq:coerL}
\int \frac{|{\Ls} f|^2}{y^{2i}(1 + y^{2k})} \geq c_{M,k} \left(\int \frac{|\partial_{yy}f|^2}{y^{2i}(1 + y^{2k})} +  \frac{|\py f|^2}{y^{2i}(1 + y^{2k + 2})} +  \frac{|f|^2}{y^{2i + 2}(1 + y^{2k + 2})}\right).
\end{equation}
and 
\begin{equation}\label{eq:coerL1}
\int \frac{|{\Ls} f|^2}{y^{2i}(1 + y^{2k})} \geq c_{M,k} \int \left(\frac{|\As f|^2}{y^{2i}(1 + y^{2k + 2})} + \int \frac{|f|^2}{y^{2i + 2}(1 + y^{2k + 2})} \right).
\end{equation}
\end{lemma}
\begin{proof} 
See Lemma A.4 in \cite{IGN16}.
\end{proof}

We are now in a position to  prove the coercivity of ${\Ls}^k$ under a suitable orthogonality condition.  We claim the following:
\begin{lemma}[Coercivity of the iterate of $\Ls$] \label{lemm:coeLk}  Let $k \in \mathbb{N}$ and $M = M(k)$ large enough, then there exists $c_{M,k} > 0$ such that for all $f \in \Dc_{rad}$ satisfying

$$\int \frac{|\As (\Ls^k f)|^2}{y^2} + \sum_{m = 0}^k \int \frac{|\Ls^m f|^2}{y^4(1 + y^{4(k - m)})} + \sum_{m = 0}^{k-1}\frac{|\As(\Ls^{m}f)|^2}{y^6(1 + y^{4(k - m - 1)})} < +\infty,$$
and the orthogonal condition 
$$\left<f, \Ls^m \Phi_M\right> = 0, \quad 0 \leq m \leq k ,$$
there holds:
\begin{align} \label{eq:coerLk}
\Es_{2k + 2}(f) = \int |\Ls^{k+1} f|^2 &\geq c_{M,k}\left\{\int \frac{|\As (\Ls^k f)|^2}{y^2} \right.\nonumber \\
& \quad +\left. \sum_{m = 0}^k \int \frac{|\Ls^m f|^2}{y^4(1 + y^{4(k - m)})} + \sum_{m = 0}^{k-1}\frac{|\As(\Ls^{m}f)|^2}{y^6(1 + y^{4(k - m - 1)})}\right\}.
\end{align}
\end{lemma}
\begin{proof} 
See Lemma A.5 in \cite{IGN16}.
\end{proof}

\section{Interpolation bounds.}
We derive in this section interpolation bounds on $q$ which are the consequence of the coercivity property given in Lemma \ref{lemm:coeLk}. We have the following:
\begin{lemma}[Interpolation bounds]\label{lemm:interbounds} $\quad$\\
$(i)$ Weighted bounds for $q_i$: for $1 \leq m \leq L+1$,
\begin{equation}\label{eq:qmbyE2k}
\int |q_{2m}|^2 + \sum_{i = 0}^{2k - 1}\int \frac{|q_i|^2}{y^2(1 + y^{4m - 2i - 2})} \leq C(M)\Es_{2m}.
\end{equation}
$(ii)$ Development near the origin:
\begin{equation}\label{eq:expandqat0}
q = \sum_{i = 1}^{L+1}c_i T_{L+1 - i} + r_q,
\end{equation}
with bounds 
$$|c_i| \lesssim \sqrt{\Es_{2L+2}},$$
$$|\py^j r_q| \lesssim y^{2L+2 - \frac{d}{2} - j}|\ln(y)|^{L+1}\sqrt{\Es_{2L+2}}, \quad 0 \leq j \leq 2L+1, \;\; y < 1.$$
$(iii)$ Bounds near the origin for $q_i$ and $\py^i q$: for $y \leq \frac 12$,
\begin{align*}
|q_{2i}| + |\py^{2i}q| &\lesssim y^{-\frac{d}{2} + 2}|\ln y|^L+1\sqrt{\Es_{2L+2}}, \quad \text{for}\quad 0 \leq i \leq L,\\
|q_{2i - 1}| + |\py^{2i-1}q|& \lesssim y^{-\frac{d}{2} + 1}|\ln y|^L+1\sqrt{\Es_{2L+2}}, \quad \text{for} \quad 1 \leq i \leq L+1.
\end{align*}
$(iv)$ Weighted bounds for $\py^iq$: for $1 \leq m \leq L+1$,
\begin{equation}\label{eq:weipykq}
\sum_{i = 0}^{2m}\int \frac{|\py^i q|^2}{1 + y^{4m - 2i}} \lesssim \Es_{2m}.
\end{equation}
Moreover, let $(i,j) \in \mathbb{N}\times \mathbb{N}^*$ with $2 \leq i + j \leq 2L+2$, then 
\begin{equation}\label{eq:weipyijq}
\int \frac{|\py^i q|^2}{1 + y^{2j}} \lesssim \left\{\begin{array}{ll}
\Es_{2m} &\quad \text{for}\quad i+j = 2m, \;\; 2 \leq m \leq L+1,\\
\sqrt{\Es_{2m}}\sqrt{\Es_{2(m+1)}}&\quad \text{for}\quad  i+j = 2m + 1,\;\; 2 \leq m \leq L.
\end{array} \right.
\end{equation}
$(v)$ Pointwise bound far away. Let $(i,j) \in \mathbb{N}\times \mathbb{N}$ with $2 \leq i + j \leq 2L+1$, we have for $y \geq 1$,
\begin{equation}\label{eq:pointwise_yg1}
\left|\frac{\py^i q}{y^{j}}\right|^2 \lesssim \frac 1{y^{5}}\left\{\begin{array}{ll}
\Es_{2m} &\quad \text{for}\quad i+j + 1 = 2m, \;\; 2 \leq m \leq L+1,\\
\sqrt{\Es_{2m}}\sqrt{\Es_{2(m+1)}}&\quad \text{for}\quad  i+j = 2m,\;\; 2 \leq m \leq L.
\end{array} \right.
\end{equation}
\end{lemma}
\begin{proof} $(i)$ The estimate \eqref{eq:qmbyE2k} directly follows from Lemma \ref{lemm:coeLk}. \\
$(ii)$ We claim that for $1 \leq m \leq L+1$, $q_{2L+2 -2m}$ admits the Taylor expansion at the origin
\begin{equation}\label{eq:expandq2k}
q_{2L+2-2m} = \sum_{i = 1}^m c_{i,m}T_{m - i} + r_{2m},
\end{equation}
with the bounds 
$$|c_{i,m}| \lesssim \sqrt{\Es_{2L+2}},$$
$$|\py^j r_{2m}| \lesssim y^{2m - \frac{d}{2} - j}|\ln(y)|^m\sqrt{\Es_{2L+2}}, \quad 0 \leq j \leq 2m - 1, \;\; y < 1,$$
The expansion \eqref{eq:expandqat0} then follows from \eqref{eq:expandq2k} with $m = L+1$.

We proceed by induction in $m$ for the proof of \eqref{eq:expandq2k}. For $m = 1$, we write from the definition \eqref{def:Astar} of $\As^*$,
$$r_1(y) = q_{2L+1}(y) = \frac{1}{y^{6}\Lambda Q}\int_0^y q_{2L+2} \Lambda Q x^{6}dx + \frac{d_1}{y^{6}\Lambda Q}.$$
Note from \eqref{eq:qmbyE2k} that $\int \frac{|q_{2L+1}|^2}{y^2} \lesssim \Es_{2L+2}$ and from \eqref{eq:asymLamQ} that $\Lambda Q \sim y$ as $y \to 0$, we deduce that $d_1 = 0$. Using the Cauchy-Schwartz inequality, we derive the pointwise estimate
$$|r_1(y)| \leq \frac{1}{y^d} \left(\int_0^y |q_{2L+2}|^2 x^{6}dx\right)^\frac{1}{2}\left(\int_0^y x^2x^{6} dx\right)^\frac{1}{2} \lesssim y^{-\frac{d}{2} + 1}\sqrt{\Es_{2L+2}}, \quad y < 1.$$ 
We remark that there exists $a \in (1/2,1)$ such that 
$$|q_{2L+1}(a)|^2 \lesssim \int_{y \leq 1} |q_{2L+1}|^2 \lesssim \Es_{2L+2}.$$
We then define 
$$r_2(y) = -\Lambda Q \int_a^y \frac{r_1}{\Lambda Q}dx,$$
and obtain from the pointwise estimate of $r_1$, 
$$|r_2(y)| \lesssim y y^{-\frac{d}{2} + 1}\sqrt{\Es_{2L+2}}\int_a^y \frac{dx}{x} \lesssim y^{-\frac{d}{2} + 2}|\ln(x)|\sqrt{\Es_{2L+2}}, \quad y < 1.$$
By construction and the definition \eqref{def:As} of $\As$, we have 
$$\As r_2 = r_1 = q_{2L+1}, \quad \Ls r_2 = \As^*q_{2L+1} = q_{2L+2} = \Ls q_{2L}.$$
Recall that $\text{span}(\Ls) = \{\Lambda Q, \Gamma\}$ where $\Gamma$ admits the singular behavior \eqref{eq:asymGamma}. From \eqref{eq:qmbyE2k}, we have $\int\frac{|q_{2L}|^2}{y^4} \lesssim \Es_{2L+2} < +\infty$. This implies that there exists $c_2 \in \Rb$ such that 
$$q_{2L} = c_2 \Lambda Q + r_2.$$
Moreover, there exists $a \in (1/2,1)$ such that
$$|q_{2L}(a)|^2 \lesssim \int_{|y| \leq 1} |q_{2L}|^2\lesssim \Es_{2L+2},$$
which follows
$$|c_2| \lesssim \sqrt{\Es_{2L+2}}, \quad |q_{2L}| \lesssim y^{-\frac{d}{2} + 2}|\ln(y)|\sqrt{\Es_{2L+2}}, \quad y < 1.$$
Since $\As r_2 = r_1$, we then write from the definition \eqref{def:As} of $\As$, 
$$|\py r_2| \lesssim |r_1| + \left|\frac{r_2}{y}\right| \lesssim y^{-\frac{d}{2} +2}|\ln(y)|\sqrt{\Es_{2L+2}}, \quad y < 1.$$
This concludes the proof of \eqref{eq:expandq2k} for $m = 1$.   

We now assume that \eqref{eq:expandq2k} holds for $m \geq 1$ and prove it for $m + 1$. The term $r_{2m}$ is built as follows:
$$r_{2m - 1} = \frac{1}{y^{6}\Lambda Q}\int_0^y r_{2m - 2} \Lambda Q x^{6}dx, \quad r_{2m} = - \Lambda Q\int_a^y \frac{r_{2m -1}}{\Lambda Q} dx, \quad a\in (1/2,1).$$
We now use the induction hypothesis to estimate
\begin{align*}
|r_{2m + 1}| &= \left|\frac{1}{y^{6} \Lambda Q}\int_0^y r_{2m} \Lambda Q x^{6}dx \right| \\
&\quad \lesssim \frac{1}{y^d}\sqrt{\Es_{2L+2}}\int_0^yx^{2m + \frac{d}{2}}|\ln(x)|^m dx\\
&\qquad \lesssim y^{2m - \frac{d}{2}} \sqrt{\Es_{2L+2}}\int_0^y |\ln(x)|^m dx \\
& \qquad \quad  \lesssim y^{2m - \frac{d}{2} + 1}|\ln(y)|^m \sqrt{\Es_{2L+2}}.
\end{align*}
Here we used the following identity 
$$I_m = \int_0^y [\ln(x)]^mdx \lesssim y|\ln(y)|^m, \quad m \geq 1, \;\; y < 1.$$
Indeed, we have $I_1 = \int_0^y \ln (x) dx = y\ln(y) - y \lesssim y|\ln(y)|$ for $y < 1$. Assume the claim for $m \geq 1$, we use an integration by parts to estimate for $m +1$
\begin{align*}
I_{m+1} &= \int_0^y[\ln(x)]^m (x\ln(x) - x)'dx\\
&\quad  = y[\ln(y)]^{m+1} - y[\ln(y)]^m - m(I_m - I_{m-1})\lesssim y|\ln(y)|^{m+1}.
\end{align*}
Using an integration by parts yields $\int_{a}^y \frac{[\ln(x)]^m}{x}dx = \frac{[\ln(y)]^{m+1} - [\ln(a)]^{m+1}}{m+1}$. Hence, we have the estimate
\begin{align*}
|r_{2m + 2}| = \left|\Lambda Q\int_a^y \frac{r_{2m +1}}{\Lambda Q} dx\right| &\lesssim y^{2m - \frac{d}{2} + 2}\sqrt{\Es_{2L+2}}\int_a^y \frac{|\ln(x)|^m}{x}dx\\
&\quad \lesssim y^{2m - \frac{d}{2} + 2} |\ln(y)|^{m+1} \sqrt{\Es_{2L+2}}.
\end{align*}

By construction, we have 
$$\As r_{2m+2} = r_{2m + 1}, \quad \Ls r_{2m+2} = r_{2m}.$$
From the induction hypothesis and the definition \eqref{def:Tk} of $T_k$, we write
$$\Ls q_{2L+2-2(m + 1)} = q_{2L+2-2m} = \sum_{i = 1}^m c_{i,m}T_{m - i} + r_{2m} = \sum_{i = 1}^mc_{i,m}\Ls T_{m + 1 - i} + \Ls r_{2m + 2}.$$
The singularity \eqref{eq:asymGamma} of $\Gamma$ at the origin and the bound $\int_{y \leq 1} \frac{|q_{2 L+1 - 2(m+1)}|^2}{y^4} \lesssim \Es_{2L+2}$ allows us to deduce that
$$q_{2L(m + 1)} = \sum_{i=1}^m c_{i,m}T_{m +1 - i} + c_{2m + 2}\Lambda Q + r_{2m + 2}.$$
From \eqref{eq:qmbyE2k}, we see that there exists $a \in (1/2,1)$ such that 
$$|q_{2L+2-2(m+1)}(a)|^2 \lesssim \int_{y \leq 1}|q_{2L+2-2(m+1)}|^2 \lesssim \Es_{2L+2}.$$
Together with the induction hypothesis $|c_{i, m}| \lesssim \sqrt{\Es_{2L+2}}$ and the pointwise estimate on $r_{2m + 2}$, we get the bound $|c_{2m + 2}| \leq \sqrt{\Es_{2L+2}}$.

A brute force computation using the definitions of $\As$ and $\As^*$ and the asymptotic behavior \eqref{eq:asympV}  ensure that for any function $f$, 
\begin{equation}\label{eq:pyjf}
\py^j f = \sum_{i = 0}^j P_{i,j}f_i, \quad |P_{i,j}| \lesssim \frac{1}{y^{j - i}},
\end{equation}
and we estimate
\begin{align*}
|\py^j r_{2m + 2}|&\lesssim \sum_{i = 0}^j \frac{|r_{2m + 2 - i}|}{y^{j-i}}\\
&\quad \lesssim \sqrt{\Es_{2L+2}}\sum_{i = 0}^j \frac{y^{2m + 2 - i - \frac{d}{2}}|\ln(y)|^{m+1}}{y^{j - i}} \lesssim y^{2m + 2 -\frac{d}{2} - j}|\ln(y)|^{m+1} \sqrt{\Es_{2L+2}}.
\end{align*}
This concludes the proof of \eqref{eq:expandq2k} as well as \eqref{eq:expandqat0}.\\

$(iii)$ The proof of $(iii)$ directly follows from \eqref{eq:expandq2k}.\\

$(iv)$ We have from \eqref{eq:pyjf},
$$|\py^kq| \lesssim \sum_{j = 0}^k \frac{|q_j|}{y^{k - j}},$$
and thus, using \eqref{eq:qmbyE2k} and the pointwise bounds given in part $(iii)$ yields
\begin{align*}
\sum_{i = 0}^{2m}\int \frac{|\py^i q|^2}{1 + y^{4m - 2i}} & \lesssim \Es_{2m} + \sum_{i = 0}^{2m - 1} \int_{y < 1} |\py^i q|^2 + \sum_{i = 0}^{2m - 1}\int_{y > 1}\frac{|\py^i q|^2}{y^{4m - 2i}}\\
&\quad \lesssim \Es_{2m} + \Es_{2L+2}\int_{y< 1}y|\ln y|^L+1 dy  + \sum_{i = 0}^{2m - 1} \sum_{j = 0}^i \int_{y > 1}\frac{|q_j|^2}{y^{4m - 2j}} \lesssim \Es_{2m}, 
\end{align*}
which concludes the proof of \eqref{eq:weipykq}. 

The estimate \eqref{eq:weipyijq} simply follows from \eqref{eq:weipykq}. Indeed, if $i + j = 2m$ with $1 \leq m \leq L+1$, we have 
$$\int \frac{|\py^i q|^2}{1 + y^{2j}} = \int \frac{|\py^iq|^2}{1 + y^{4m - 2i}} \lesssim \Es_{2m}.$$
If $i + j = 2m + 1$ with $1 \leq m \leq L$, we write 
\begin{align*}\int \frac{|\py^i q|^2}{1 + y^{2j}} = \int \frac{|\py^iq|^2}{1 + y^{4m - 2i + 2}} &\lesssim \left(\int \frac{|\py^iq|^2}{1 + y^{4m - 2i}} \right)^\frac 12\left(\int \frac{|\py^iq|^2}{1 + y^{4m - 2i + 4}} \right)^\frac 12\\
&\quad \lesssim \sqrt{\Es_{2m}} \sqrt{\Es_{2(m+1)}}.
\end{align*}
 
$(v)$ Let $i,j \geq 0$ with $1 \leq i + j \leq 2L+2 -1$, then $2 \leq i + j + 1 \leq 2L+2$ and we conclude from \eqref{eq:weipyijq} that for $y \geq 1$,
\begin{align*}
\left|\frac{\py^i q}{y^j} \right|^2 &\lesssim \left| \int_y^{+\infty} \partial_x \left(\frac{(\partial_x^i q)^2}{x^{2j}} \right)dx \right| \lesssim \frac{1}{y^{5}} \left\{\int_y^{+\infty} \frac{|\partial_x^i q|^2}{x^{2j + 2}} + \int_y^{+\infty} \frac{|\partial_x^{i+1} q|^2}{x^{2j}} \right\}\\
&\quad \lesssim \frac 1{y^{5}}\left\{\begin{array}{ll}
\Es_{2m} &\quad \text{for}\quad i+j + 1 = 2m, \;\; 1 \leq m \leq L+1,\\
\sqrt{\Es_{2m}}\sqrt{\Es_{2(m+1)}}&\quad \text{for}\quad  i+j+1 = 2m + 1,\;\; 1 \leq m \leq L.
\end{array} \right.
\end{align*} 
This ends the proof of Lemma \ref{lemm:interbounds}.
 
\end{proof}

\section{Proof of \eqref{est:comtor}.}\label{ap:EstComm}
We give here the proof of \eqref{est:comtor}. Before going to the proof, we need the following Leibniz rule for $\Ls^k$.
\begin{lemma}[Leibniz rule for $\Ls^k$] \label{lemm:LeibnizLk}  Let $\phi$ be a smooth function and $k \in \mathbb{N}$, we have 
\begin{equation}\label{eq:LeibnizLk}
\Ls^{k + 1}(\phi f) = \sum_{m = 0}^{k+1}f_{2m}\phi_{2k+2, 2m} + \sum_{m = 0}^k f_{2m + 1}\phi_{2k + 2, 2m + 1},
\end{equation}
and 
\begin{equation}\label{eq:LeibnizALk}
\As\Ls^{k}(\phi f) = \sum_{m = 0}^{k}f_{2m + 1}\phi_{2k+1, 2m + 1} + \sum_{m = 0}^k f_{2m}\phi_{2k+1, 2m},
\end{equation}
where\\
- for $k = 0$,
\begin{align*}
&\phi_{1,0} = -\py \phi, \quad \phi_{1,1} = \phi,\\
&\phi_{2,0} = -\py^2 \phi - \frac{6 +2V}{y}\py \phi, \quad \phi_{2,1}= 2\py \phi, \quad \phi_{2,2} = \phi,
\end{align*}
- for $k \geq 1$
\begin{align*}
&\phi_{2k + 1, 0} = -\py \phi_{2k,0},\\
&\phi_{2k + 1, 2i} = - \py \phi_{2k, 2i} - \phi_{2k, 2i - 1}, \quad 1 \leq i \leq k,\\
&\phi_{2k+1, 2i + 1}= \phi_{2k, 2i} + \frac{6 + 2V}{y}\phi_{2k, 2i + 1} - \py\phi_{2k, 2i+1}, \quad 0 \leq i \leq k-1,\\
&\phi_{2k + 1, 2k + 1} = \phi_{2k, 2k} = \phi,\\
& \quad\\
&\phi_{2k + 2, 0} = \py \phi_{2k+1,0} + \frac{6 + 2V}{y}\phi_{2k+1, 0},\\
&\phi_{2k + 2, 2i} = \phi_{2k + 1, 2i - 1} + \py \phi_{2k + 1, 2i} + \frac{6 + 2V}{y}\phi_{2k + 1, 2i},\quad 1 \leq i \leq k,\\
&\phi_{2k+2, 2i + 1}= -\phi_{2k + 1,2i} + \py\phi_{2k+1, 2i + 1}, \quad 0 \leq i \leq k,\\
&\phi_{2k + 2, 2k + 2} = \phi_{2k + 1, 2k + 1} = \phi.
\end{align*}
\end{lemma}
\begin{proof} 
See Lemma C.1 in \cite{IGN16}.
\end{proof}

\bigskip

Let us now give the proof of \eqref{est:comtor}. By induction and the definition \eqref{def:Llambda}, we have 
$$[\pt, \Ls_\lambda^{L}]v = \sum_{m = 0}^{L-1} \Ls_\lambda^m \left([\pt, \Ls_\lambda] \Ls_\lambda^{L-1 - m}v\right) = \sum_{m = 0}^{L-1} \Ls_\lambda^m \left(\frac{\pt Z_\lambda}{r^2} \Ls_\lambda^{L-1 - m}v\right).$$
Noting that $\frac{\pt Z_\lambda}{r^2} = \frac{b_1 \Lambda Z}{\lambda^4 y^2}$, we make a change of variables to obtain
\begin{align*}
\int \frac{1}{\lambda^2(1 + y^2)}\left|[\pt , \Ls_{\lambda}^{L}]v\right|^2 &= \frac{b_1^2}{\lambda^{4L-1}}\int \frac{1}{1 + y^2} \left|\sum_{m = 0}^{L-1}\Ls^m\left(\frac{\Lambda Z}{y^2} \Ls^{L-1 - m}q \right) \right|^2\\
& \lesssim \frac{b_1^2}{\lambda^{4L-1}} \sum_{m = 0}^{L-1}\int \frac{1}{1 + y^2}\left|\Ls^m\left(\frac{\Lambda Z}{y^2} \Ls^{L-1 - m}q \right) \right|^2.
\end{align*}
For $m = 0$, we use \eqref{eq:estLamZV} and \eqref{eq:Enercontrol} to estimate
$$\int \frac{1}{1 + y^2}\left|\left(\frac{\Lambda Z}{y^2} \Ls^{L-1}q \right) \right|^2 \lesssim \int \frac{|q^2_{2L-2}|}{1 + y^{10}} \lesssim \Es_{2L+2}.$$
For $m = 1, \cdots, L+1 - 2$, we apply \eqref{eq:LeibnizLk} with $\phi = \frac{\Lambda Z}{y^2} = \frac{(6)\Lambda \cos(2Q)}{y^2}$ and note from \eqref{eq:asymQ} that
$$|\phi_{k, i}| \lesssim \frac{1}{1 + y^{4 + 2 + (2k - i)}} \lesssim \frac{1}{1 + y^{4 + (2k - i)}}, \quad k \in \mathbb{N}^*, \;\; 0 \leq i \leq 2k,$$
which yields
\begin{align*}
\int \frac{1}{1 + y^2}\left|\Ls^m\left(\frac{\Lambda Z}{y^2} \Ls^{L-1 - m}q \right) \right|^2 \lesssim \sum_{i = 0}^{2m}\int\frac{q^2_{2L-2 - 2m - i}}{(1 + y^{10 + (4m - 2i)})} \lesssim \Es_{2L+2}.
\end{align*}
Thus, 
$$\int \frac{1}{\lambda^2(1 + y^2)}\left|[\pt , \Ls_{\lambda}^{L}]v\right|^2 \lesssim \frac{b_1^2}{\lambda^{4L-1}} \Es_{2L+2}.$$
Similarly, we use \eqref{eq:LeibnizALk} to get the estimate 
$$\int \left|\As[\pt , \Ls_{\lambda}^{L}]v\right|^2 \lesssim  \frac{b_1^2}{\lambda^{4L-1}}\Es_{2L+2}.$$
This concludes the proof of \eqref{est:comtor}.

\def\cprime{$'$}


\end{document}